\pdfoutput=1
\RequirePackage{ifpdf}
\ifpdf 
\documentclass[pdftex]{sigma}
\else
\documentclass{sigma}
\fi

\usepackage[all]{xy}

\numberwithin{equation}{section}

\newtheorem{Theorem}{Theorem}[section]

\newtheorem{Lemma}[Theorem]{Lemma}
\newtheorem{prop}[Theorem]{Proposition}
 { \theoremstyle{definition}
\newtheorem{defn}[Theorem]{Definition}

\newtheorem{ex}[Theorem]{Example}
\newtheorem{exer}[Theorem]{Exercise}
\newtheorem{rmk}[Theorem]{Remark} }

\DeclareMathOperator{\rep}{Rep}

\DeclareMathOperator{\tr}{Tr}
\DeclareMathOperator{\spec}{Spec}
\DeclareMathOperator{\quot}{Quot}

\DeclareMathOperator{\im}{Im}

\DeclareMathOperator{\rk}{rk}
\DeclareMathOperator{\proj}{Proj}
\DeclareMathOperator{\Stab}{Stab}
\DeclareMathOperator{\diag}{diag}
\DeclareMathOperator{\End}{End}
\DeclareMathOperator{\Ext}{Ext}
\DeclareMathOperator{\Aut}{Aut}

\DeclareMathOperator{\Aff}{Aff}
\DeclareMathOperator{\ai}{a.i.}
\DeclareMathOperator{\codim}{codim}
\DeclareMathOperator{\nil}{nil}
\DeclareMathOperator{\Lie}{Lie}
\DeclareMathOperator{\Fr}{Fr}
\DeclareMathOperator{\cl}{cl}

\newcommand{\et}{\mathrm{\acute{e}t}}

\newcommand{\Mods}{\mathcal{M}_{d}^{\theta\text{\rm -s}}}

\newcommand{\ra}{\rightarrow}
\newcommand{\la}{\leftarrow}

\def\cA{\mathcal A}\def\cC{\mathcal C}
\def\cE{\mathcal E}\def\cF{\mathcal F}\def\cG{\mathcal G}\def\cH{\mathcal H}
\def\cI{\mathcal I}\def\cL{\mathcal L}
\def\cM{\mathcal M}\def\cO{\mathcal O}
\def\cR{\mathcal R}\def\cT{\mathcal T}

\def\cY{\mathcal Y}\def\cZ{\mathcal Z}

\def\AA{\mathbb A}\def\CC{\mathbb C}
\def\FF{\mathbb F}\def\GG{\mathbb G}\def\HH{\mathbb H}

\def\NN{\mathbb N}\def\PP{\mathbb P}
\def\QQ{\mathbb Q}\def\RR{\mathbb R}

\def\ZZ{\mathbb Z}

\def\fg{\mathfrak g}
\def\fk{\mathfrak k}\def\fl{\mathfrak l}

\def\fu{\mathfrak u}

\def\fX{\mathfrak X}

\newcommand{\Id}{\mathrm{Id}}
\newcommand{\op}{\mathrm{op}}

\newcommand{\Mat}{\mathrm{Mat}}
\newcommand{\Hom}{\mathrm{Hom}}
\newcommand{\Br}{\mathrm{Br}}

\def\GL{\mathbf{GL}}
\def\Gl{\mathrm{GL}}

\def\SUl{\mathrm{SU}}
\def\U{\mathbf{U}}
\def\Ul{\mathrm{U}}

\def\Sl{\mathrm{SL}}
\def\PGL{\mathbf{PGL}}
\def\GSp{\mathrm{GSp}}

\renewcommand{\cF}{\mathcal{F}}

\newcommand{\N}{\mathbb{N}}
\newcommand{\A}{\mathbb{A}}
\newcommand{\Z}{\mathbb{Z}}

\newcommand{\Mod}{\mathcal{M}_{d}^{\theta\text{\rm -ss}}(Q)}
\newcommand{\Modgs}{\mathcal{M}_{d}^{\theta\text{\rm -gs}}(Q)}
\newcommand{\Rep}{\rep_{d}(Q)}
\newcommand{\G}{\mathbf{G}_{d}}
\newcommand{\pphi}{\varphi}

\newcommand{\chith}{\chi_\theta}

\newcommand{\RepQbar}{\rep_{d}\big(\overline{Q}\big)}

\newcommand{\LieG}{\fg_{d}}
\newcommand{\LieGL}{\fg\fl}

\newcommand{\crep}{{\cR}{\rm ep}}

\newcommand{\ov}[1]{\overline{#1}}

\begin{document}

\allowdisplaybreaks

\newcommand{\arXivNumber}{1809.05738}

\renewcommand{\thefootnote}{}

\renewcommand{\PaperNumber}{127}

\FirstPageHeading

\ShortArticleName{Parallels between Moduli of Quiver Representations and Vector Bundles over Curves}

\ArticleName{Parallels between Moduli of Quiver Representations\\ and Vector Bundles over Curves\footnote{This paper is a~contribution to the Special Issue on Geometry and Physics of Hitchin Systems. The full collection is available at \href{https://www.emis.de/journals/SIGMA/hitchin-systems.html}{https://www.emis.de/journals/SIGMA/hitchin-systems.html}}}

\Author{Victoria HOSKINS}

\AuthorNameForHeading{V.~Hoskins}

\Address{Freie Universit\"at Berlin, Arnimallee 3, Raum 011, 14195 Berlin, Germany}
\Email{\href{mailto:hoskins@math.fu-berlin.de}{hoskins@math.fu-berlin.de}}
\URLaddress{\url{http://userpage.fu-berlin.de/hoskins/}}

\ArticleDates{Received September 25, 2018, in final form November 18, 2018; Published online December 04, 2018}

\Abstract{This is a review article exploring similarities between moduli of quiver representations and moduli of vector bundles over a smooth projective curve. After describing the basic properties of these moduli problems and constructions of their moduli spaces via geometric invariant theory and symplectic reduction, we introduce their hyperk\"{a}hler analogues: moduli spaces of representations of a doubled quiver satisfying certain relations imposed by a moment map and moduli spaces of Higgs bundles. Finally, we survey a surprising link between the counts of absolutely indecomposable objects over finite fields and the Betti cohomology of these (complex) hyperk\"{a}hler moduli spaces due to work of Crawley-Boevey and Van den Bergh and Hausel, Letellier and Rodriguez-Villegas in the quiver setting, and work of Schiffmann in the bundle setting.}


\Classification{14D20; 14L24; 16G20; 14H60}

\renewcommand{\thefootnote}{\arabic{footnote}}
\setcounter{footnote}{0}

{\small
\setcounter{tocdepth}{2}
\tableofcontents}

\section{Introduction}

The goal of this survey article is to describe several parallels between moduli of bundles and quiver representations. The main topics are divided into 3 parts:
\begin{itemize}\itemsep=0pt
\item Properties and constructions of moduli spaces (Sections~\ref{sec moduli quiver} and~\ref{sec moduli bundle}).
\item Associated hyperk\"{a}hler moduli spaces (Section~\ref{sec hyperkahler}).
\item Cohomology of hyperk\"{a}hler moduli spaces and counting indecomposable objects (Section~\ref{sec counting indecomp Betti}).
\end{itemize}
Within the text there are several exercises for the reader, as well as some interesting open problems.

We will focus on some of the most fundamental and striking similarities between these modu\-li problems; however, it is not possible to properly survey several important results concerning these moduli problems with the care they deserve. We will largely omit the study of the associated moduli stacks and techniques for studying the cohomology of moduli spaces via Harder--Narasimhan recursions on the stack \cite{atiyahbott,HN,reinekeHN}. Moreover, we will not discuss Hall algebras associated to these moduli problems in any depth, or the relationship with Donaldson--Thomas theory; for a comprehensive introduction to Hall algebras see \cite{schiffmann_lecturesHall}.

Moduli of quiver representations generalise many natural problems in linear algebra (for example, the classification of similar matrices via Jordan normal form). Despite their seemingly simple nature, quiver moduli spaces are ubiquitous in algebraic geometry (in fact, every projective variety arises as a quiver grassmannian \cite{reineke_QGrass}). Moreover, the study of such moduli spaces can shed light on related moduli problems and questions in representation theory.

The moduli problem most closely related to that of quiver representations is moduli of vector bundles (or coherent sheaves) on a smooth projective curve. Both moduli problems have associated abelian categories of homological dimension~1 and have associated moduli stacks which are smooth. In order to construct moduli spaces, one must restrict to a class of stable (or semistable) objects, then one obtains smooth moduli spaces of stable objects. These mo\-duli spaces can be constructed as algebraic quotients using geometric invariant theory \cite{mumford} (for bundles, the construction was given by Mumford, Newstead and Seshadri~\cite{mumford,newstead_bundles,seshadri}, and for quivers, this construction was given by King~\cite{king}), or, when over the complex numbers, via symplectic reduction. In fact, quiver moduli spaces have a finite dimensional symplectic construction, whereas moduli of vector bundles have an infinite dimensional gauge-theoretic symplectic construction~\cite{atiyahbott}. For quivers, the algebraic and symplectic quotients are homeomorphic via the Kempf--Ness theorem~\cite{kempf_ness,king}. The Kobayshi--Hitchin correspondence \cite{donaldsonNS,ns,UY} gives the corresponding relationship for the gauge theoretic constructions of moduli spaces of vector bundles. Before proceeding, let us mention one important difference between moduli spaces of bundles and quiver representations: although moduli spaces of semistable vector bundles on a curve provide compactifications of moduli spaces of stable vector bundles, moduli spaces of semistable quiver representations are only projective over an associated affine quiver variety.

We then turn to the study of associated (non-compact) hyperk\"{a}hler moduli spaces. The symplectic constructions of moduli spaces $\cM$ of quiver representations and vector bundles both arise by considering a smooth symplectic action of a Lie group on a complex vector space (in the case of vector bundles, the group and vector space both have infinite dimension). We can upgrade this to a hyperk\"{a}hler setting by taking the cotangent lift of this action and then perform a hyperk\"{a}hler reduction to construct a hyperk\"{a}hler analogue $\cH$ of $\cM$ such that \smash{$T^*\cM \subset \cH$}. The hyperk\"{a}hler reductions we obtain are moduli spaces of representations of a doubled quiver satisfying certain relations imposed by a moment map (closely related to Nakajima quiver varieties~\cite{nakajima}) and moduli spaces of Higgs bundles \cite{Hitchin, simpson_HiggsGm}. We also describe constructions of submanifolds (known as branes~\cite{Kapustin_Witten}) in these hyperk\"{a}hler moduli spaces with particularly rich holomorphic and symplectic geometry \cite{BS1,BS2,BGP,BGPH,fjm,HS_quiver_autos}.

Finally we survey several surprising results relating the counts of absolutely indecomposable objects of these moduli problems over finite fields and the Betti cohomology of their associated hyperk\"{a}hler moduli spaces. For quivers, this is due to work of Crawley-Boevey and Van den Bergh \cite{CBVdB} and Hausel, Letellier and Rodriguez-Villegas \cite{HLRV}, and was motivated by Kac's work in representation theory~\cite{Kac1}. This work inspired Schiffmann \cite{schiffmann} to formulate and prove an analogous statement for bundles, which lead to formulae for the Betti numbers of moduli spaces of Higgs bundles in the coprime setting.

The structure of this article is as follows: Sections~\ref{sec moduli quiver} and~\ref{sec moduli bundle} describe the basic properties and constructions of moduli spaces of quiver representations and vector bundles respectively. In Section~\ref{sec hyperkahler}, we introduce the associated hyperk\"{a}hler moduli spaces and survey some constructions of interesting submanifolds known as branes. In Section~\ref{sec counting indecomp Betti}, we provide the proof of Crawley-Boevey and Van den Bergh relating the counts of absolutely indecomposable quiver representations with the Betti numbers of the associated hyperk\"{a}hler moduli spaces, and then sketch how Schiffmann extends this to bundles.

\section{Moduli spaces of quiver representations}\label{sec moduli quiver}

\subsection{Quiver representations over a field}

A quiver $Q =(V,A,h,t)$ is a finite connected directed graph consisting of finite sets of vertices~$V$ and arrows~$A$ with head and tail maps $h,t\colon A \ra V$ giving the directions of the arrows. Throughout this section, we fix a field~$k$.

\begin{defn}\label{def_rep} A $k$-representation of $Q$ is a tuple $W:=((W_v)_{v\in V}, (\varphi_a)_{a\in A})$ where:
\begin{itemize}\itemsep=0pt
\item $W_v$ is a finite-dimensional $k$-vector space for all $v\in V$;
\item $\varphi_a\colon W_{t(a)} \ra W_{h(a)}$ is a $k$-linear map for all $a\in A$.
\end{itemize}
The dimension vector of $W$ is the tuple $\dim W=(\dim W_v)_{v\in V}$.
A morphism between two $k$-representations $W:=((W_v)_{v\in V}, (\varphi_a)_{a\in A})$ and $W':=((W'_v)_{v\in V}, (\varphi'_a)_{a\in A})$ is a tuple of linear maps $(f_v\colon W_v \ra W_v')_{v \in V}$ such that for all $a \in A$ the following diagram commutes
 \begin{gather*}
\xymatrix{ W_{t(a)} \ar[d]_{f_{t(a)}} \ar[r]^{\varphi_a} & W_{h(a)}\ar[d]^{f_{h(a)}}\\
W'_{t(a)} \ar[r]^{\varphi'_a} & W'_{h(a)}.}\end{gather*}
\end{defn}

\noindent The category $\crep(Q,k)$ of $k$-representations of $Q$ is a $k$-linear abelian category. For two $k$-representations $W$ and $W'$ of $Q$, the set of morphisms between them is a $k$-vector space denoted $\Hom_Q(W,W')$ and similarly one can consider the spaces of extensions between such representations.

\begin{ex}[the Jordan quiver]Let $Q$ be the one loop quiver. Then a $k$-representation of~$Q$ is a vector space $W$ with an endomorphism $\phi\colon W \ra W$. Two representations $(W,\pphi)$ and $(W',\pphi')$ of $Q$ are isomorphic if $\dim W = \dim W'$ and there is an isomorphism $f\colon W \ra W'$ such that $f \circ \pphi = \pphi' \circ f$. In particular, for any representation $(W,\pphi)$ of $Q$ of dimension $n$ we can choose a basis of $W$ to obtain an isomorphic representation $(k^{n}, M)$, where $M \in \Mat_{n \times n}(k)$. A different choice of basis would replace $M$ with a conjugate matrix $S M S^{-1}$. Thus the isomorphism classes of $n$-dimensional $k$-representations of $Q$ are in bijection with conjugacy classes of $(n \times n)$-matrices. For an algebraically closed field $k$, we can classify the latter using Jordan normal forms.
\end{ex}

\begin{exer}\label{exer hom and exts}For $k$-representations $W:=((W_v)_{v\in V}, (\pphi_a)_{a\in A})$ and $W':=((W'_v)_{v\in V}, (\pphi'_a)_{a\in A})$ of a quiver $Q$, show that a tuple $\phi=(\phi_a)_{a \in A} \in \prod \limits_{a \in A} \Hom(W_{t(a)}, W'_{h(a)}) $ determines a representation of $Q$
\begin{gather*} e(W,W',\phi):=\left( (W'_v \oplus W_v)_{v \in V}, \left(\begin{matrix} \pphi'_a & \phi_a \\ 0 & \pphi_a \end{matrix} \right)_{a \in A} \right), \end{gather*}
which fits into a short exact sequence of quiver representations
\begin{gather*} 0 \ra W' \ra e(W,W',\phi) \ra W \ra 0. \end{gather*}
Show that this defines a map $\beta\colon \prod \limits_{a \in A} \Hom(W_{t(a)}, W'_{h(a)}) \ra \Ext^1_Q(W,W')$ which fits in an exact sequence
 \begin{gather*}
0 \longrightarrow \Hom_Q(W,W') \longrightarrow \prod \limits_{v \in V} \Hom(W_v,W_v') \overset{\alpha}{\longrightarrow} \prod \limits_{a \in A} \Hom(W_{t(a)}, W'_{h(a)})\\
\hphantom{0}{} \overset{\beta}{\longrightarrow} \Ext^1_Q(W,W') \longrightarrow 0,
\end{gather*}
where $\alpha$ is defined by $\alpha( (f_v)_{v \in V}) := f_{h(a)} \circ \pphi_a - \pphi'_a \circ f_{t(a)}$. In particular, deduce that
\begin{gather*} \dim \Hom_Q(W,W') - \dim \Ext^1_Q(W,W') = \sum_{v \in V} \dim W_v \dim W'_v - \sum_{a \in A} \dim W_{t(a)} \dim W'_{h(a)}. \end{gather*}
\end{exer}

Following this observation, we define a bilinear form on the free abelian group on the set of vertices $V$.

\begin{defn}The Euler form associated to $Q$ is a bilinear form on $\ZZ^V$ given by
\begin{gather*} \langle d,d' \rangle_Q := \sum_{v \in V} d_v d'_v - \sum_{a \in A} d_{t(a)}d'_{h(a)}, \end{gather*}
where $d = (d_v)_{v \in V}$ and $d'= (d'_v)_{v \in V}$. The symmetrised Euler form is defined by
\begin{gather*} (d,d')_Q:= \langle d,d' \rangle_Q + \langle d',d \rangle_Q,\end{gather*}
and the associated Tits quadratic form is defined by $q_Q(d):=\langle d,d \rangle_Q$.
\end{defn}

By Exercise \ref{exer hom and exts}, the Euler form relates the dimensions of the Hom and Ext groups:
\begin{gather*} \langle \dim W, \dim W' \rangle_Q = \dim \Hom_Q(W,W') - \dim \Ext^1_Q(W,W').\end{gather*}
In fact, we can view this as the Euler characteristic of $\Hom_Q^\bullet(W,W')$ as all higher Ext groups vanish for quiver representations by Exercise \ref{exer quiver hom dim 1} below.

\begin{rmk}\samepage\quad
\begin{enumerate}\itemsep=0pt
\item[1)] The Euler form depends on the orientation of $Q$. The Euler form is symmetric if and only if $Q$ is symmetric (that is, for any two vertices $v$ and $w$, we have $| a\colon v \ra w | = | a\colon w \ra v |$).
\item[2)] The quadratic form $q_Q$ on $\ZZ^V$ associated to the Euler form only depends on the underlying graph of the quiver $Q$. In fact, properties of this quadratic form can be related to the properties of the underlying graph (for example, if $Q$ is a Dynkin quiver, then $q_Q$ is positive definite \cite[Section~4]{CB_notes}.)
\end{enumerate}
\end{rmk}

\begin{exer}[the category of quiver representations has homological dimension~1]\label{exer quiver hom dim 1}
For $U \in \crep(Q,k)$, we can apply $\Hom_Q(-,U)$ to any short exact sequence $0 \ra W' \ra W \ra W'' \ra 0$ in $\crep(Q,k)$ to obtain a long exact sequence
\begin{gather*}
 0 \ra \Hom_Q(W'', U) \ra \Hom_Q(W,U) \ra \Hom_Q(W',U) \ra \Ext^1_Q(W'',U) \ra \\
 \hphantom{0}{} \ra \Ext^1_Q(W,U) \ra \Ext^1_Q(W',U) \ra \Ext^2_Q(W'',U) \ra \cdots.
\end{gather*}
Using the description of $\Ext^1_Q$ in Exercise~\ref{exer hom and exts}, prove that $\Ext^1_Q(W,U) \ra \Ext^1_Q(W',U)$ is surjective, from which it follows that $\Ext^2_Q(W'',U) = 0$ for all $W''$ and $U$ and thus $\crep(Q,k)$ has homological dimension~1.
\end{exer}

\begin{rmk}The path algebra $k(Q)$ of $Q$ over $k$ is the $k$-vector space spanned by all paths in~$Q$ (including a trivial path $e_v$ at each vertex $v$) with multiplication given by concatenation of paths. In general, this is a non-commutative algebra, which is generated by the paths of length~$0$ (the vertices $V$) and the paths of length $1$ (the arrows $A$). We note that the path algebra is a finite dimensional $k$-algebra if and only if $Q$ has no oriented cycles. Moreover, the category $\crep(Q,k)$ is equivalent to the category of left $k(Q)$-modules (see \cite[Proposition~1.2.2]{brion_notes}). One can also calculate Ext groups of quiver representations by taking projective resolutions of the associated $k(Q)$-module.
\end{rmk}

\begin{defn}A quiver representation $W$ is
\begin{enumerate}\itemsep=0pt
\item[1)] \emph{simple} if it has no proper non-zero subrepresentations,
\item[2)] \emph{indecomposable} if it cannot be written as a direct sum of proper subrepresentations.
\end{enumerate}
Clearly every simple representation is indecomposable.
\end{defn}

\begin{exer}[Schur's lemma and simple quiver representations] For a simple $k$-representa\-tion~$W$ of~$Q$, prove that $\End_Q(W) $ is a division algebra using Schur's lemma. Hence, if $k$ is algebraically closed, deduce that $\End_Q(W) \cong k$ and $\Aut_Q(W) \cong k^{\times}$.
\end{exer}

\begin{ex}For any quiver $Q$, there are simple representations $S(v)$ indexed by the vertices of $V$, where $S(v)_w$ is zero for all $w \neq v$ and $S(v)_v =k$ and all the linear maps are zero. In fact, if $Q$ is a quiver without oriented cycles, then these are the only simple representations (see \cite[Proposition~1.3.1]{brion_notes}).
\end{ex}

\begin{exer}Find a quiver $Q$ such that
\begin{itemize}\itemsep=0pt
\item there is a simple representation not of the form $S(v)$,
\item there is a indecomposable representation which is not simple.
\end{itemize}
\end{exer}

\begin{rmk} For a field extension $k \subset k'$ we have a natural functor
\begin{gather*} - \otimes_k k' \colon \ \crep(Q,k) \ra \crep(Q,k')\end{gather*}
given by extension of scalars.
\end{rmk}

\begin{exer}Show that extension of scalars does not in general preserve simple (respectively indecomposable) quiver representations. For example, consider $k = \RR \subset k' = \CC$ with a 2-dimensional representation of the Jordan quiver.
\end{exer}

In fact, the category $\crep(Q,k)$ of $k$-representations of $Q$ is a Krull--Schmidt category, which means that the endomorphism ring of every idempotent representation is local (see Lemma~\ref{lemma end indecomp}) and every representation is isomorphic to a finite direct sum of indecomposable representations (and up to permutation, the indecomposable representations in such a direct sum are uniquely determined up to isomorphism); for details, see \cite[Theorem~1.3.4]{brion_notes}.

\subsection{GIT construction of moduli spaces}

In this section, we describe King's construction \cite{king} of moduli spaces of representations of a~quiver $Q=(V,A,h,t)$ over a field $k$. We fix a dimension vector $d=(d_v)_{v\in V}\in \N^V$. Then every $k$-representation of $Q$ of fixed dimension vector $d$ is isomorphic to a point of the following affine $k$-space
\begin{gather*} \Rep := \prod_{a\in A} \Mat_{d_{h(a)}\times d_{t(a)}}.\end{gather*}
The reductive $k$-group $\GL_d:=\prod\limits_{v\in V} \Gl_{d_v}$ acts algebraically on $\Rep$ by conjugation: for $g = (g_v)_{v \in V} \in \GL_d$ and $M = (M_a)_{a \in A} \in \Rep$, we have
 \begin{gather}\label{action_of_G_on_Rep}
g\cdot M := \big(g_{h(a)} M_a g_{t(a)}^{-1}\big)_{a\in A}
\end{gather}
and the orbits for this action are in bijection with the set of isomorphism classes of $d$-dimensional $k$-representations of $Q$ by Exercise~\ref{exer orbits of action} below. There is a subgroup $\Delta:= \{ (t I_{d_v})_{v\in V} \colon t\in \GG_m\} \subset \GL_d$ acting trivially on $\Rep$ and therefore a quotient of the action of $\GL_d$ is equivalent to a quotient of the action of $\G:= \GL_d/\Delta$. We have that
\begin{gather}\label{eq euler dim}
 \langle d ,d \rangle_Q = \dim \GL_d - \dim \Rep.
\end{gather}

\begin{exer}\label{exer orbits of action}
Show the orbit and stabiliser of $M \in \Rep(k)$ have the following descriptions:
\begin{gather*} \GL_d(k) \cdot M = \{ M' \in \Rep(k)\colon M' \cong M \} \end{gather*}
and
\begin{gather*} \Stab_{\GL_d(k)}(M) \cong \Aut_Q(M). \end{gather*}
Moreover, deduce from \eqref{eq euler dim} and Exercise \ref{exer hom and exts} that
\begin{gather*} \dim \Ext^1_Q(M,M) = \codim (\GL_d(k) \cdot M).\end{gather*}
\end{exer}

One would like to construct a moduli space for quiver representations as a quotient of this action using geometric invariant theory (GIT) \cite{mumford}. Since $\Rep$ is an affine variety and $\GL_d$ is a reductive group, the ring of invariant functions
\begin{gather*} \cO(\Rep)^{\GL_d}=\big\{ f \colon \Rep \ra \AA^1\colon g \cdot f = f \ \forall \, g \in \GL_d \big\} \end{gather*}
is a finitely generated $k$-algebra and the inclusion $\cO(\Rep)^{\GL_d} \hookrightarrow \cO(\Rep)$ induces a $\GL_d$-invariant morphism $\pi\colon \Rep \ra \Rep/\!/ \GL_d:= \spec \cO(\Rep)^{\GL_d}$ of affine varieties, which is the affine GIT quotient. The double quotient notation indicates that this is not an orbit space in general, as $\pi$ identifies orbits whose closures meet.

\begin{exer}Let $k$ be an algebraically closed field. For the Jordan quiver with dimension vector $n$, we have that $\Rep = \Mat_{n \times n} \cong \AA^{n^2} $ with $\GL_n$ acting by conjugation. Show that
\begin{gather*} \cO(\Rep)^{\GL_n} = k[\sigma_1, \dots , \sigma_n], \end{gather*}
where $\sigma_i \in \cO(\Rep)$ are the coefficients of the characteristic polynomial (viewed as functions on $\Rep$ which are invariant under conjugation). In particular, deduce that $\pi\colon \Rep \ra \Rep/\!/ \GL_n \cong \AA^n$ identifies orbits when $n \geq 2$ (for example, use the classification of orbits given by Jordan normal form). In fact, this is a special case of a result of Le Bruyn and Procesi described below, as the coefficients of the characteristic polynomial of a matrix can be computed by taking traces of powers of the matrix.
\end{exer}

The affine GIT quotient $\pi\colon \Rep \ra \Rep/\!/ \GL_d$ may identify all orbits; for example, if there are no non-constant invariant functions, which is the case if $Q$ has no oriented cycles by the following theorem.

\begin{Theorem}[Le Bruyn and Procesi \cite{lBP}]The ring of invariants $\cO(\Rep)^{\GL_d}$ is generated by taking traces along oriented cycles in $Q$.
\end{Theorem}

Instead King constructs a GIT quotient of the $\GL_d$-action on an open subset of $\Rep$ by linearising the action using a stability parameter $\theta=(\theta_v)_{v\in V}\in \Z^V$. The stability parameter $\theta$ determines a character $\chi_\theta\colon \GL_d \ra \GG_m$
\begin{gather}\label{the_character}
\chi_{\theta}( (g_v)_{v \in V} ):=\prod_{v\in V} (\det g_v)^{\theta_v},
\end{gather}
which descends to a character of $\G$ if and only if $\chi_\theta(\Delta)=1$ (that is, $\theta \cdot d := \sum\limits_{v \in V} \theta_v d_v = 0$).

Let $\cL_\theta$ denote the $\GL_d$-linearisation on the trivial line bundle $\Rep\times\A^1$, where $\GL_d$ acts on $\AA^1$ via multiplication by $\chi_\theta$. Then we can use invariant sections of positive powers of~$\cL_\theta$ to construct a GIT semistable set and a GIT quotient. As in \cite{king}, the invariant sections of positive tensor powers $\cL_\theta^n$ of this linearisation are $\chi_\theta^n$-semi-invariant functions; that is, $f\colon \Rep\ra \AA^1$ satisfying $f(g\cdot X) = \chi_\theta(g)^n\, f(X)$, for all $g\in\GL_d$ and all $X\in\Rep$. We let $\cO(\Rep)^{\GL_d,\chi_\theta^n}$ denote the subset of $\chi_\theta^n$-semi-invariant functions; then
\begin{gather*} H^0(\Rep, \cL_\theta^n)^{\GL_d}= \cO(\Rep)^{\GL_d,\chi_\theta^n}. \end{gather*}
Since $\Delta$ acts trivially on $\Rep$, invariant sections of $\cL_\theta^n$ for $n > 0$ only exist if $\chi_\theta(\Delta) = 1$ (i.e., $\theta \cdot d = 0$).

\begin{defn}For the $\GL_d$-linearisation on $\Rep$ given by $\chi_\theta$, we say a point $X \in \Rep$ is GIT semistable if there exists $n >0$ and an $\GL_d$-invariant section $f$ of $\cL_\theta^n$ with $f(X) \neq 0$. We let $\Rep^{\chi_\theta\text{-ss}}$ denote the subset of semistable points.
\end{defn}

The semistable set is an open subset of $\Rep$ and is non-empty only if $\theta \cdot d = 0$. Henceforth, we shall assume that $\theta \cdot d = 0$ in order to have a non-empty semistable set.

Mumford's linearised version of GIT gives us a GIT quotient
\begin{gather*} \Rep^{\chi_\theta\text{-ss}} \ra \Rep /\!/_{\chi_\theta} \GL_d := \proj \left(\bigoplus_{n \geq 0} \cO(\Rep)^{\GL_d,\chi_\theta^n} \right).\end{gather*}

\begin{rmk}The $0$th graded piece $\cO(\Rep)^{\GL_d,\chi_\theta^0} = \cO(\Rep)^{\GL_d}$ is the ring of invariant functions, and we have a projective (and thus proper) morphism
\begin{gather*} p \colon \ \Rep /\!/_{\chi_\theta} \GL_d \ra \Rep /\!/ \GL_d = \spec \cO(\Rep)^{\GL_d} \end{gather*}
to an affine variety. If $Q$ is a quiver without oriented cycles, then $\cO(\Rep)^{\GL_d} = k$ and $\Rep /\!/_{\chi_\theta} \GL_d$ is a projective variety.
\end{rmk}

The $\GL_d$-invariant sections of positive powers of $\cL_\theta$ are also used to determine a GIT notion of stability with respect to $\chi_\theta$ (see, \cite[Definition 2.1]{king} for $k = \ov{k}$, where we note that the notion of stability is modified to account for the presence of the global stabiliser $\Delta$). This determines an open subset $\Rep^{\chi_\theta\text{-s}}$ of $\chi_\theta$-stable points and the GIT quotient restricts to a~quotient $\pi|_{\Rep^{\chi_\theta\text{-s}}}\colon \Rep^{\chi_\theta\text{-s}} \ra \Rep^{\chi_\theta\text{-s}}/\GL_d$ which is a geometric quotient (which in particular, is an orbit space) of the GIT stable set.

Using the Hilbert--Mumford criterion to relate GIT semistability of geometric points with stability for 1-parameter subgroups $\lambda\colon G \ra \GL_d$, King proves that the GIT notion of (semi)stabi\-li\-ty can be translated to a notion of (semi)stabi\-li\-ty for $d$-dimensional representations of~$Q$. For points over a non-algebraically closed field, GIT stability is related to a notion of geometric stability for representations as described below.

\begin{defn}[semistability]\label{defn theta stab} Let $\theta \cdot d = 0$. We say a $d$-dimensional $k$-representation $W$ of~$Q$ is:
\begin{enumerate}\itemsep=0pt
\item[1)] $\theta$-semistable if $\theta \cdot \dim W' \geq 0$ for all $k$-subrepresentations $0 \neq W'\subsetneq W$.
\item[2)] $\theta$-stable if $\theta \cdot \dim W' > 0$ for all $k$-subrepresentations $0 \neq W'\subsetneq W$.
\item[3)] $\theta$-polystable if it is isomorphic to a direct sum of $\theta$-stable representations of equal slope.
\item[4)] $\theta$-geometrically stable if $W \otimes_k k'$ is $\theta$-stable for all field extensions $k'/k$.
\end{enumerate}
There are natural notions of Jordan--H\"{o}lder filtrations, and we say two $\theta$-semistable $k$-represen\-ta\-tions of $Q$ are $S$-equivalent if their associated graded objects for their Jordan--H\"{o}lder filtrations are isomorphic.
\end{defn}

\begin{exer}[rephrasing of stability as a slope-type condition] For $\theta \in \ZZ^V$, we define the slope of a $k$-representation $W$ of $Q$ by
\begin{gather*}\mu_\theta(W) := \frac{\sum\limits_{v\in V} \theta_v\dim_k W_v}{\sum\limits_{v\in V} \dim_k W_v}.\end{gather*}
Let $\theta'_v := \theta_v \sum\limits_{w \in V} d_w - \sum\limits_{w\in V} \theta_w d_w$ for all $v\in V$; then show that $\sum\limits_{v\in V} \theta'_v d_v =0$ and that slope semistability with respect to $\theta$ and $\theta'$ coincide (where slope semistability means all subrepresentations have slope less than or equal to the slope of the representation). Furthermore, show for $d$-dimensional representations of $Q$ that $(-\theta')$-semistability (as in Definition~\ref{defn theta stab}) is equivalent to slope semistability with respect to $\theta'$.
\end{exer}

The slope version of (semi)stability enables one to easily define Harder--Narasimhan (HN) filtrations for quiver representations. In~\cite{reinekeHN}, Reineke used the HN stratification on $\Rep$ (together with the HN system in an associated Hall algebra) to give formulae for the Poincar\'{e} polynomials of moduli spaces of semistable representations of quivers without oriented cycles, when semistability is taken with respect to a generic stability parameter in the following sense.

\begin{defn}A stability parameter $\theta \in \ZZ^V$ is called \emph{generic} with respect to $d$ if $\theta \cdot d = 0$ and for all non-zero $d' \in \NN^V$ with $d' < d$ (i.e., $d' \neq d$ and $d_v' \leq d_v$ for all $v \in V$), we have $\theta \cdot d' \neq 0$.
\end{defn}

\begin{rmk}For a generic stability parameter $\theta$ with respect to $d$, every $\theta$-semistable $k$-representation of $Q$ is also $\theta$-stable.
\end{rmk}

Using the Hilbert--Mumford criterion, which gives a criterion for GIT semistability using 1-parameter subgroups of $\GL_d$, King shows the open subsets $\Rep^{\theta\text{-ss}}$ and $\Rep^{\theta\text{-gs}}$ of $\theta$-semistable and $\theta$-geometrically stable $k$-representations in $\Rep$ coincide with the GIT semistable and stable locus respectively:
\begin{gather*} \Rep^{\theta\text{-ss}} = \Rep^{\chi_\theta\text{-ss}} \qquad \text{and} \qquad \Rep^{\theta\text{-gs}} = \Rep^{\chi_\theta\text{-s}}.\end{gather*}
Hence, the GIT quotient $\Mod:=\Rep /\!/_{\chi_\theta} \G$ is a $k$-variety that co-represents the modu\-li functor of $\theta$-semistable $k$-representations of $Q$ of dimension $d$ (up to $S$-equivalence). Moreover, $\Modgs := \Rep^{\chi_\theta\text{-s}}/\G$ is an open $k$-subvariety of $\Mod$ that co-repre\-sents the moduli functor of $\theta$-geometrically stable $k$-representations of~$Q$ of dimension~$d$ (up to isomorphism). In general, $\Modgs$ does not admit a universal family in general; however, there is a Brauer class controlling the obstruction to admitting a universal family and an associated universal twisted family over~$\Modgs$ (cf.\ \cite[Proposition~4.18]{HS_galois}). We will refer to both these spaces as moduli spaces.

\begin{exer}Consider the $n$-arrow Kronecker quiver with $V = \{v_1,v_2\}$ and $A = \{ a_i\colon v_1 \ra v_2\}_{i=1}^n$. For the dimension vector $d=(1,1)$, we have $\Rep \cong \AA^n$ with the action of $\GL_d \cong \GG_m \times \GG_m$ given by $(t_1,t_2) \mapsto \diag\big(t_2t_1^{-1}, \dots , t_2t_1^{-1}\big)$. If we naturally identify $\cO(\Rep) \cong k[x_1, \dots ,x_n]$ with each variable corresponding to an arrow, then $\cO(\Rep)^{\GL_d} = k$. Moreover, show that for $\theta_+ = (1,-1)$ and $\theta_- = (-1,1)$ we have the following semistable loci and GIT quotients: $\Rep^{\theta_+\text{-ss}} = \varnothing$ and
\begin{gather*} \Rep^{\theta_-\text{-ss}} = \AA^n - \{ 0 \} \mapsto \Rep /\!/_{\! \theta_-} \GL_d = \proj (k[x_1, \dots , x_n]) \cong \PP^{n-1}. \end{gather*}
\end{exer}

\begin{exer}[stable representations are simple]Prove that a $\theta$-stable $k$-representation of~$Q$ is simple. If $k$ is algebraically closed, deduce that the automorphism group of a $\theta$-stable representation of $Q$ is isomorphic to the multiplicative group $k^\times$.
\end{exer}

In fact, for an arbitrary field $k$, the stabiliser group of an $\theta$-geometrically stable $k$-represen\-ta\-tion is the subgroup $\Delta \subset \GL_d$ (for example, see \cite[Corollary~2.14]{HS_galois}). The action of $\G:= \GL_d/\Delta$ on $\Rep^{\theta\text{-gs}}$ is free and the geometric quotient $\Rep^{\theta\text{-gs}} \ra \Modgs$ is a~principal $\G$-bundle; thus $\Modgs$ is smooth. Provided $\Rep^{\theta\text{-gs}} \neq \varnothing$, we have
\begin{gather*} \dim \Modgs = \dim \Rep - \dim \G = \dim \Rep - \dim \GL_d + 1 = 1 - \langle d, d \rangle_Q. \end{gather*}

\begin{rmk}In fact, it is a theorem of Gabriel that the Tits form $q_Q(d):= \langle d, d \rangle_Q$ is positive definite if and only if the underlying graph of $Q$ is a simply-laced Dynkin diagram; this is in turn equivalent to each of the following statements:
\begin{enumerate}\itemsep=0pt
\item[i)] $q_Q(d) \geq 1$,
\item[ii)] there is an open $\G$-orbit in $\Rep$,
\item[iii)] there are only finitely many $\G$-orbit in $\Rep$.
\end{enumerate}
By Exercise~\ref{exer orbits of action} we see (ii) implies (i) and the equivalence of (ii) and (iii) holds as the closure of any orbit is a union of finitely many orbits and, as $\Rep$ is irreducible, any open orbit is dense.
\end{rmk}

For an algebraically closed field $k$, the closed points of $\Mod$ are in bijection with $S$-equivalence classes of $\G(k)$-orbits of $\chith$-semistable rational points, where two $k$-representa\-tions~$M_1$ and~$ M_2$ are $S$-equivalent if their orbit closures intersect in $\Rep^{\chi_\theta\text{-ss}}$. By \cite[Proposition~3.2(ii)]{king}, this is the same as the $S$-equivalence of ${M_1}$ and ${M_2}$ as $\theta$-semistable representations of~$Q$. Moreover, for~$k$ algebraically closed, we have $\Mods(k) = \Rep^{\theta\text{-s}}(k)/\G(k)$.

\begin{rmk}For a non-algebraically closed field, the rational points of these moduli spaces do not in general correspond to rational orbits. In \cite{HS_galois}, for a perfect field $k$, we show using Galois cohomology and descent that there is an injection
\begin{gather*} \Rep^{\theta\text{-gs}}(k)/\G(k) \hookrightarrow \Modgs(k),\end{gather*}
and that the remaining points in $\Modgs(k)$ that do not come from isomorphism classes of $k$-representations can be described as representations of $Q$ over central division algebras over $k$ (or equivalently as twisted quiver representations). More precisely, as the moduli stack
\begin{gather*} \big[\Rep^{\theta\text{-gs}}/\GL_d\big] \ra \Modgs \end{gather*}
is a $\GG_m$-gerbe over $\Modgs$, we can pull this gerbe back along $\spec k \ra \Modgs$ to obtain a type map
\begin{gather*} \cT \colon \ \Modgs(k) \ra \Br(k):=H^2_{\et}(k,\GG_m). \end{gather*}
We prove that a necessary condition for a division algebra $D \in \Br(k)$ to lie in the image of $\cT$ is that $\deg(D):=\sqrt{\dim_k(D)} $ divides the dimension vector $d$; that is, $d = d_D \deg(D)$ for some dimension vector $d_D$. Moreover, we interpret the fibre $\cT^{-1}(D)$ using Galois descent as the space of isomorphism classes of $\theta$-geometrically stable $d_D$-dimensional representations of $Q$ over the division algebra $D$. This gives a complete description of $\Modgs(k)$ in terms of isomorphism classes of representations over division algebras with centre $k$. Let us mention two special cases of this result:
\begin{itemize}\itemsep=0pt
\item For $k = \RR$, we have $\Br(\RR) = \{ \RR, \HH\}$ and so the points in $\Modgs(\RR)$ are rational or quaternionic quiver representations (and the latter only occur if $2 | d$).
\item For a finite field $k = \FF_q$, the Brauer group is trivial; thus $\Modgs(\FF_q)$ is precisely the set of isomorphism classes of $\theta$-geometrically stable $d$-dimensional $\FF_q$-representations of $Q$.
\end{itemize}
\end{rmk}

\subsection{Symplectic construction of complex quiver varieties}\label{sec sympl constr}

Over the complex numbers, the Kempf--Ness theorem \cite{kempf_ness} relates certain geometric invariant theory quotients by reductive groups to smooth symplectic reductions by maximal compact subgroups \cite{kempf_ness}. In the case of quiver moduli, we have a complex reductive group $\GL_d$ acting on a complex affine space $\Rep$, and via a Kempf--Ness theorem, the GIT quotient of $\GL_d$ acting on $\Rep$ with respect to $\chi_\theta\colon \GL_d \ra \GG_m$ is homeomorphic to the smooth symplectic reduction of the action of a maximal compact subgroup of $\GL_d$ on $\Rep$ as described by King \cite{king}. In this section, we briefly explain this alternative symplectic construction.

The complex reductive group $\GL_d$ is the complexification of the maximal compact subgroup
\begin{gather*} \U_d = \prod_{v \in V} U(d_v). \end{gather*}
We consider the Hermitian form $H\colon \Rep \times \Rep \ra \CC$ defined by
\begin{gather*} H(X,Y) = \sum_{a \in A} \tr \big(X_a Y_a^\dagger\big), \end{gather*}
where $Y^\dagger$ is the complex conjugate transpose of $Y$. Let $\omega_{\RR}(-,-) := - \operatorname{Im}H(-,-)$; then $\omega_{\RR}$ is a~(smooth) symplectic form on the manifold $\Rep$. In fact, the complex structure on $\Rep$ is compatible with this form, and so $\Rep$ is naturally a (flat) K\"{a}hler manifold. Moreover, the form $\omega_{\RR}$ is preserved by the action of~$\U_d$, as $H$ is $\U_d$-invariant.

\begin{defn}A \emph{moment map} for a symplectic action of a compact Lie group $K$ on a smooth symplectic manifold $(M,\omega)$ is a smooth map $\mu_{\RR}\colon M \ra \mathfrak{k}^*:= \Lie (K)^*$ which is equivariant with respect to the given $K$-action on $M$ and the coadjoint action of $K$ on $\mathfrak{k}^*$ and lifts the infinitesimal action in the sense that
\begin{gather*} d_m \mu_{\RR}(\xi) \cdot B = \omega_{m}(B_m,\xi) \end{gather*}
for all $m \in M$ and $\xi \in T_mM$ and $B \in \mathfrak{k}$, where $B_m \in T_mM$ denotes the infinitesimal action of~$B$ on~$m$.
\end{defn}

We are interested in the situation where $K$ acts linearly on a complex vector space $M = \CC^n$. A $K$-invariant Hermitian inner product $H$ on $M$ gives $M$ the structure of a K\"{a}hler manifold, as we can write $H = g - {\rm i} \omega$, where $g$ is a metric and $\omega$ a K\"{a}hler form. In this situation, by the following exercise, a moment map always exists but it is not necessarily unique as we can always shift it by a central value of~$\fk^*$.

\begin{exer}Let $K$ act linearly on $M = \CC^n$ and pick a $K$-invariant Hermitian inner product $H= g - {\rm i} \omega$ on $M$. Prove that a moment map for the $K$-action on $(M,\omega)$ is given by $\mu_{\RR}\colon M \ra \mathfrak{k}^*$ with
\begin{gather*} \mu_{\RR}(m) \cdot B := \frac{ {\rm i}}{2} H(B_m , m).\end{gather*}
Furthermore, show that we can shift this moment map by any central value $\chi \in \fk^*$.
\end{exer}

In particular, there is a moment map $\mu_{\RR}\colon \Rep \ra \mathfrak{u}_d^*$ for the action of $\U_d$ on $\Rep$ given by
\begin{gather*} \mu_{\RR}(X) \cdot B = \frac{ {\rm i}}{2} H(B_X , X) = \sum_{a \in A} \tr \big(( B_{h(a)}X_a - X_a B_{t(a)})X_a^\dagger\big). \end{gather*}
By identifying $\mathfrak{u}_d \cong \mathfrak{u}_d^*$ using the Killing form, we obtain a map $\mu_{\RR}^*\colon \Rep \ra \mathfrak{u}_d$ where
\begin{gather*} \mu_{\RR}^*(X) = \sum_{a \in A} [X_a, X_a^{\dagger}] = \left(\sum_{a\colon h(a) =v} X_a X_a^{\dagger} - \sum_{a\colon t(a) = v} X_a^{\dagger} X_a\right)_{v \in V}. \end{gather*}
Moreover, any tuple $\theta = (\theta_v)_{v \in V} \in \ZZ^V$ defines a character $\chi_\theta\colon \GL_d \ra \GG_m$ whose restriction to $\U_d$ has image in $\Ul(1)$. Hence, we can view the derivative $d\chi_\theta|_{\U_d} \colon \mathfrak{u}_d \ra \mathfrak{u}(1) \cong 2 \pi {\rm i} \RR$ as a~coadjoint fixed point of $\mathfrak{u}_d^*$ (often we also denote this by~$\theta$ or~$\chi_\theta$); this coadjoint fixed point can be used to shift the moment map.

Such shifting of the moment map by a central value merely corresponds to considering different fibres of the moment map; this choice can be used to produce different symplectic reductions as follows.

\begin{defn}For a symplectic $K$-action on $(M,\omega)$ with moment map $\mu_\RR\colon M \ra \mathfrak{k}^*$, we define the \emph{$($smooth$)$ symplectic reduction} of the $K$-action on $M$ at a coadjoint fixed point $\chi \in \mathfrak{k}^*$ to be the topological quotient
\begin{gather*} \mu_{\RR}^{-1}(\chi)/K.\end{gather*}
We note that the level set $ \mu_{\RR}^{-1}(\chi) \subset M$ is $K$-invariant by the equivariance of $\mu_{\RR}$.
\end{defn}

If $\chi$ is a regular value of $\mu_{\RR}$ (i.e., the differential of $\mu_\RR$ at every point in the preimage $\mu^{-1}_{\RR}(\chi)$ is surjective), then $\mu^{-1}_{\RR}(\chi) \subset M$ is a smooth submanifold. If $K$ is a compact Lie group acting freely on $\mu_{\RR}^{-1}(\chi)$, then the topological quotient $\mu_{\RR}^{-1}(\chi)/K$ is a smooth manifold by the slice theorem and moreover, it inherits a (smooth) symplectic form from the form $\omega$ on $M$ by the Marsden--Weinstein theorem~\cite{mw}. If in fact, $(M,\omega)$ is a K\"{a}hler manifold (i.e., there is a~compatible complex structure and Riemannian metric on~$M$), then the symplectic reduction is also K\"{a}hler, provided it is smooth. If $K$ acts on $\mu_{\RR}^{-1}(\chi)$ with finite stabilisers, then $\mu_{\RR}^{-1}(\chi)/K$ is a symplectic orbifold, and more generally if~$K$ acts with positive dimensional stabilisers, then $\mu_{\RR}^{-1}(\chi)/K$ can be given the structure of a stratified symplectic manifold.

Let $(M,\omega)$ be a K\"{a}hler manifold that is a smooth affine (or projective) complex variety with a Fubini--Study form. If there is a linear action of a complex reductive group $G$ on $M$ for which a maximal compact subgroup $K < G$ preserves $\omega$, then the Kempf--Ness theorem~\cite{kempf_ness} provides a~homeomorphism between the geometric invariant theory quotient of~$M$ by $G$ and the symplectic reduction of $M$ by $K$. We recall that the GIT quotient depends on a choice of linearisation of the action. In the case when $M$ is projective, this work extends to give a comparison between an algebraic GIT stratification and a symplectic Morse-theoretic stratification~\cite{kirwan,ness}. In the affine setting $M \subset \AA^n$, if $\chi\colon G \ra \GG_m$ is a character which is used to linearise this action, then we can restrict~$\chi$ to maximal compact subgroups and take derivatives to obtain $d\chi|_K\colon \mathfrak{k} \ra \mathfrak{u}(1) \cong 2 \pi {\rm i} \RR$; this defines an element of $\mathfrak{k}^*$, which by abuse of notation we also denote by~$\chi$. Then the Kempf--Ness theorem gives an inclusion $\mu^{-1}_{\RR}(\chi) \hookrightarrow M^{\chi\text{-ss}}$ which induces a homeomorphism
\begin{gather*} \mu^{-1}_{\RR}(\chi)/K \ra M/\!/_\chi G. \end{gather*}
More precisely, the Kempf--Ness theorem states that the $G$-orbit closure of a $\chi$-semistable orbit in $M$ meets the level set $\mu^{-1}_{\RR}(\chi)$ in a unique $K$-orbit and the inclusion $\mu^{-1}_{\RR}(\chi) \hookrightarrow M^{\chi\text{-ss}}$ induces the above homeomorphism; for details in this affine setting, see \cite{king} and \cite{hoskins_quivers}, which also relates the GIT instability stratification with the symplectic Morse-theoretic stratification.

Let us return to the action of $G = \GL_d$ on $M = \Rep$, then $\theta$ determines a charac\-ter~$\chi_\theta$ of $\GL_d$ and a coadjoint fixed point $\theta \in \mathfrak{u}_d^*$. The inclusion $\mu_{\RR}^{-1}(\chi_\theta) \subset \Rep^{\theta\text{-ss}}$ induces a~homeomorphism
\begin{gather}\label{eq KN quiver}
\mu_{\RR}^{-1}(\chi_\theta)/ \U_d \simeq\Rep /\!/_{\chi_\theta} \GL_d
\end{gather}
and if $\theta$ is generic with respect to $d$ (so that semistability and stability with respect to $\theta$ coincide for $d$-dimensional representations of $Q$), then this symplectic reduction is a smooth symplectic (in fact, K\"{a}hler) manifold.

\section{Moduli spaces of vector bundles over curves}\label{sec moduli bundle}

The moduli problem of classifying algebraic vector bundles over a smooth projective curve has many similarities with that of quiver representations, which we explain in this section.

\subsection{Vector bundles over a curve}

Let $X$ be a smooth projective curve over a field $k$. The genus of $X$ is $g(X):= h^0(X,\omega_X)$, where $\omega_X := \Omega^1_X$ is the \emph{canonical line bundle}.

We will often use the equivalence between the category of (algebraic) vector bundles on $X$ and the category of locally free sheaves on $X$. We recall that this equivalence is given by associating to a vector bundle $F \ra X$ the sheaf $\cF$ of sections of $F$. One should be careful when going between vector bundles and locally free sheaves, as this correspondence does not preserve subobjects in general.

Although the category of locally free sheaves is not abelian, the category $\cC{\rm oh}(X)$ of coherent sheaves of $\cO_X$-modules is abelian. The category $\cC{\rm oh}(X)$ has homological dimension 1, as $\Ext$-groups can be described as sheaf cohomology groups:
\begin{gather*} \Ext^i(\cE,\cF) = H^i(X, \cH{\rm om}(\cE,\cF)), \end{gather*}
which vanish for $i \geq 2$ as $\dim X = 1$. Moreover, the first cohomology groups can be described using Serre duality.

We can also define an Euler characteristic by
\begin{gather*} \chi(\cE):= \dim H^0(X,\cE) - \dim H^1(X,\cE) \end{gather*}
and for a pair $\cE$ and $\cF$ of locally free sheaves, we define an Euler form by
\begin{gather*} \langle \cE, \cF \rangle := \chi\big(\cE^\vee \otimes \cF\big)= \dim \Hom(\cE,\cF) - \dim \Ext^1(\cE,\cF).\end{gather*}
In fact, using the Riemann--Roch formula, this Euler characteristic is entirely described by the invariants of these sheaves
\begin{gather*} \langle \cE, \cF \rangle = \rk \cE \deg \cF - \rk \cF \deg \cE + \rk \cE \rk \cF (1-g) \end{gather*}
analogously to the case for quiver representations.

\subsection{Construction of moduli spaces of vector bundles}

In this section, we outline some different constructions of moduli spaces of (algebraic) vector bundles of rank $n$ and degree $d$ over $X$. We start with an algebraic approach using geometric invariant theory which generalises to the construction of moduli spaces of coherent sheaves over projective schemes. We then survey the gauge theoretic construction over the complex numbers as an infinite-dimensional symplectic reduction, which generalises to principal bundles and hyperk\"{a}hler analogues, such as Higgs bundles (cf.\ Section~\ref{sec Higgs}).

\subsubsection{Slope stability for vector bundles}

\begin{defn}The slope of a non-zero vector bundle $E$ over $X$ is the ratio
\begin{gather*} \mu(E):= \frac{\deg E}{\rk E}. \end{gather*}
 A vector bundle $E$ is slope stable (resp. semistable) if every proper non-zero vector subbundle $E' \subset E$ satisfies
 \begin{gather*} \mu(E') < \mu(E) \qquad (\text{resp.} \ \mu(E') \leq \mu(E) \text{ for semistability}). \end{gather*}
 A vector bundle $E$ is polystable if it is a direct sum of stable bundles of the same slope.
\end{defn}

\begin{rmk}Since the degree and rank are both additive on short exact sequences of vector bundles
 \begin{gather*} 0 \ra E \ra F \ra G \ra 0,\end{gather*}
the following statements hold:
 \begin{enumerate}\itemsep=0pt
 \item[1)] If two out of the three bundles have the same slope $\mu$, the third also has slope $\mu$.
 \item[2)] $\mu(E) < \mu(F)$ (resp. $\mu(E) > \mu(F)$) if and only if $\mu(F) < \mu(G)$ (resp. $\mu(F) > \mu(G)$).
 \end{enumerate}
\end{rmk}

\begin{exer} Let $L$ be a line bundle and $E$ a vector bundle over $X$; then show
 \begin{enumerate}\itemsep=0pt
 \item[i)] $L$ is stable,
 \item[ii)] if $E$ is stable (resp.\ semistable), then $E \otimes L$ is stable (resp.\ semistable).
 \end{enumerate}
\end{exer}

If we fix a rank $n$ and degree $d$ such that $n$ and $d$ are coprime, then the notion of semistability for vector bundles with invariants $(n,d)$ coincides with the notion of stability.

\begin{exer}[stable vector bundles are simple]\label{ss lemma}
Let $f\colon E \ra F$ be a non-zero homomorphism of vector bundles on $X$ over $k = \ov{k}$; then prove the following statements.
\begin{enumerate}\itemsep=0pt
 \item[i)] If $E$ and $F$ are semistable, $\mu(E) \leq \mu(F)$.
 \item[ii)] If $E$ and $F$ are stable of the same slope, then $f$ is an isomorphism.
 \item[iii)] Every stable vector bundle $E$ is simple: $\End(E) = k$.
\end{enumerate}
\end{exer}

\subsubsection{GIT construction}

The moduli problem of rank $n$ and degree $d$ vector bundles over $X$ is unbounded, in the sense that there is no finite type $k$-scheme parameterising all such vector bundles. We can overcome this problem by restricting to moduli of semistable vector bundles, which is bounded by work of Le Potier and Simpson~\cite{simpson}. The first construction of moduli spaces of semistable vector bundles over $X$ were given by Mumford \cite{mumford}, Seshadri \cite{seshadri} and Newstead \cite{newstead_bundles,newstead}. In these notes, we will essentially follow the construction due to Simpson \cite{simpson} which generalises the curve case to a higher-dimensional projective scheme. An excellent indepth treatment of the construction following Simpson can be found in the book of Huybrechts and Lehn~\cite{HL}. We will exploit the fact that we are over a curve to simplify some of the arguments; for example, the boundedness of semistable sheaves is significantly easier over a curve, and in fact if we assume that the degree is sufficiently large, we have the following boundedness result.

\begin{Lemma}\label{lemma ss bound} Let $\cF$ be a locally free sheaf over $X$ of rank $n$ and degree $d > n(2g-1)$. If the associated vector bundle $F$ is semistable, then the following statements hold:
 \begin{enumerate}\itemsep=0pt
 \item[$i)$] $H^1(X,\cF)= 0$;
 \item[$ii)$] $\cF$ is generated by its global sections.
 \end{enumerate}
\end{Lemma}
\begin{proof} For i), we argue by contradiction using Serre duality: if $H^1(X,\cF) \neq 0$, then dually there would be a non-zero homomorphism $f\colon \cF \ra \omega_X$. We let $E \subset F$ be the vector subbundle generically generated by the kernel of $f$, which is a vector subbundle of rank $n-1$ with
 \begin{gather*} \deg E \geq \deg \ker f \geq \deg \cF - \deg \omega_X = d -(2g -2). \end{gather*}
In this case, by semistability of $F$, we have
 \begin{gather*} \frac{d - (2g-2)}{n-1}\leq \mu(E) \leq \mu(F) = \frac{d}{n}; \end{gather*}
this gives $d \leq n(2g-2)$, which contradicts our assumption on the degree of $F$.

 For ii), we let $F_x$ denote the fibre of the vector bundle at a point $x \in X$. If we consider the fibre $F_x$ as a torsion sheaf over $X$, then we have a short exact sequence
\begin{gather*} 0 \ra \cF(-x):= \cF \otimes \cO_X(-x) \ra \cF \ra F_x = \cF \otimes k_x \ra 0, \end{gather*}
which gives rise to an associated long exact sequence in cohomology and it suffices to show that $H^1(X, \cF(-x)) =0$. To prove this vanishing, we apply part i) above to the sheaf $\cF(-x)= \cF \otimes \cO_X(-x) $ which is also semistable with $\deg(\cF(-x))=d -n> n(2g-2)$.
\end{proof}

Given a locally free sheaf $\cF$ of rank $n$ and degree $d$ that is generated by its global sections, we can consider the evaluation map
\begin{gather*} \text{ev}_\cF \colon \ H^0(X,\cF) \otimes \cO_X \ra \cF,\end{gather*}
which is, by assumption, surjective. If also $H^1(X,\cF) = 0$, then by the Riemann--Roch formula
\begin{gather*} \chi(\cF) = d + n(1-g) = \dim H^0(X,\cF) - \dim H^1(X,\cF) = \dim H^0(X,\cF);\end{gather*}
 that is, the dimension of the $0$th cohomology is fixed and equal to $N: = d + n(1-g)$. Therefore, we can choose an isomorphism $H^0(X,\cF) \cong k^{N}$ and combine this with the evaluation map for~$\cF$, to produce a surjection
\begin{gather*} q_\cF \colon \ \cO_X^{\oplus N} \twoheadrightarrow \cF\end{gather*}
from a fixed trivial vector bundle. Such surjective homomorphisms from a fixed coherent sheaf are parametrised by a Quot scheme, which is a natural generalisation of the Grassmannian varieties (for a thorough treatement of Quot schemes, see \cite{nitsure}).

Let $Q:=\quot_X^{n,d}\big(\cO_X^{\oplus N}\big)$ be the Quot scheme of rank $n$, degree $d$ quotient sheaves of the trivial rank $N$ vector bundle. Let $Q^{\mu\text{-(s)s}} \subset Q$ denote the open subscheme consisting of quotients $q\colon \cO_X^{\oplus N} \ra \cF$ such that $\cF$ is a slope (semi)stable locally free sheaf and $H^0(q)$ is an isomorphism.

For a semistable sheaf $\cF$, we note that different choices of isomorphism $H^0(X,\cF) \cong k^N$ give rise to different points in $Q^{\mu\text{-ss}}$. Any two choices of the above isomorphism are related by an element in the general linear group $\Gl_N$ and this gives rise to an action of $\Gl_N$ on the Quot scheme $Q$ such that the orbits in $Q^{\mu\text{-(s)s}}(k)$ are in bijective correspondence with the isomorphism classes of (semi)stable locally free sheaves on~$X$ with invariants~$(n,d)$. In fact, the diagonal $\GG_m < \Gl_N$ acts trivially, and so it suffices to take a quotient by the action of $\Sl_N$ to construct a moduli space.

We linearise this action to give an equivariant projective embedding in order to construct a~GIT quotient. There is a natural family of invertible sheaves on the Quot scheme arising from Grothendieck's embedding of the Quot scheme into the Grassmannians: for sufficiently large~$m$, we have a closed immersion
\begin{gather*} Q= \quot_X^{n,d}\big(\cO_X^{\oplus N}\big) \hookrightarrow \text{Gr}\big(H^0\big(\cO_X(m)^{\oplus N}\big), M\big) \hookrightarrow \PP\big(\wedge^M H^0\big(\cO_X(m)^{\oplus N}\big)^\vee\big),\end{gather*}
where $M= mr +d +r(1-g)$. We let $\cL_m$ denote the pull back of $\cO_{\PP}(1)$ to the Quot scheme via this closed immersion. There is a natural linear action of $\Sl_N$ on $H^0\big(\cO_X(m)^{\oplus N}\big)= k^N \otimes H^0(\cO_X(m))$, which induces a linear action of $\Sl_N$ on $\PP\big(\wedge^M H^0\big(\cO_X(m)^{\oplus N}\big)^\vee\big)$; hence, $\cL_m$ admits a linearisation of the $\Sl_N$-action.

This linearised action has a GIT quotient
\begin{gather*} Q^{\text{ss}} \ra Q/\!/_{\cL_m} \Sl_N, \end{gather*}
where $Q^{\text{ss}}$ denotes the GIT semistable locus and, as $Q$ is projective, this GIT quotient is also projective. Provided we take $d$ sufficiently large and $m$ sufficiently large, the notion of GIT semistability for this $\Sl_N$-action coincides with slope semistability; that is $Q^{\mu\text{-ss}} = Q^{\text{ss}}$~\cite{simpson}. Over an algebraically closed field, GIT stability corresponds to slope stability, and over an arbitrary field $k$, GIT stability corresponds to geometric stability (cf.~\cite{langer}). The above GIT quotient is a moduli space
\begin{gather*} \cM_C^{\text{ss}}(n,d):= Q/\!/_{\cL_m} \Sl_N\end{gather*}
for semistable rank $n$ degree $d$ vector bundles over $X$ (up to $S$-equivalence), and its restriction to the GIT stable locus is a moduli space for geometrically stable vector bundles (up to isomorphism). As we are over a curve, the open subscheme $Q^{\text{ss}} \subset Q$ is smooth; however, the GIT quotient $\cM_C^{\text{ss}}(n,d)$ of this smooth variety may be singular, as the action is not necessarily free.

For coprime rank and degree, semistability and stability coincide and, as stable vector bundles are simple, it follows that the GIT quotient is a $\PGL_N$-principal bundle (by Luna's \'{e}tale slice theorem). In this case, the projective moduli space $\cM=\cM_C^{\text{ss}}(n,d)$ is also smooth. Moreover, using the deformation theory of vector bundles, one can describe the Zariski tangent spaces to~$\cM$ by
\begin{gather*} T_{[\cE]}\cM \cong \text{Ext}^1(\cE,\cE).\end{gather*}
In particular, $\dim \cM = n^2(g-1) +1$. Over a higher-dimensional base, we have the same description of the tangent spaces at stable sheaves, except now the obstruction to smoothness of the moduli space (and Quot scheme) lies in a second Ext group, which could be non-zero; see \cite[Corollary~4.52]{HL}.

\subsubsection{Functorial construction}

\'{A}lvarez-C\'{o}nsul and King \cite{ack} provide a construction of moduli spaces of semistable sheaves by functorially embedding this moduli problem into a moduli problem for quiver representations. More precisely, Simpson's GIT construction \cite{simpson} of moduli spaces of sheaves on a polarised variety $(X,\cO_X(1))$ depends on choices of natural numbers $m \gg n\gg 0$ (first one takes $n$ sufficiently large, so all semistable sheaves are $n$-regular and can be parametrised by a Quot scheme, and then one takes $m$ sufficiently to embed this Quot scheme in a Grassmannian and give a linearisation of the action). In \cite{ack}, a functor
\begin{gather*} \Phi_{n,m}:= \Hom(\cO_X(-n) \oplus \cO_X(-m),{-}) \colon \ \cC{\rm oh}(X) \ra \crep({{\rm Kr}_{n,m}})\end{gather*}
from the category of coherent sheaves on $X$ to the category of representations of a Kronecker quiver ${\rm Kr}_{n,m}$ with two vertices $n$, $m$ and $\dim H^0(\cO_X(m-n))$ arrows from $n$ to $m$ is used for $m \gg n\gg 0$ to provide an embedding of the subcategory of semistable sheaves with Hilbert polynomial~$P$ into a subcategory of semistable quiver representations of fixed dimension (where both the semistability parameter and dimension vector depend on $n$, $m$ and $P$). This functorial approach is then used to construct the moduli space of semistable sheaves on $X$ with Hilbert polynomial $P$ using King's GIT construction of quiver moduli spaces.

\subsubsection{Gauge-theoretic construction}\label{sec gauge constr M}

Over $k = \CC$, Atiyah and Bott \cite{atiyahbott} use an alternative gauge theoretic construction of this moduli space as a symplectic (in fact, K\"{a}hler) reduction.

In this complex setting, the curve $X$ can be viewed as a compact Riemann surface. Rather than working in the algebraic category, we can switch to the holomorphic category, by using the GAGA-equivalence, which gives an equivalence between the category of algebraic bundles on $X$ (viewed as an algebraic curve) and the category of holomorphic vector bundles on $X$ (viewed as a complex manifold); for details, see \cite{serreGAGA}. A holomorphic vector bundle can be viewed as a complex vector bundle with a holomorphic structure, which one can equivalently view as a~Dolbeault operator, as the integrability condition holds trivially for dimension reasons. For a~fixed complex vector bundle $E \ra X$, we consider
\begin{gather*} \cC=\cC(E):=\{ \text{holomorphic structures on } E \}; \end{gather*}
this is an infinite-dimensional complex vector space which is modelled on $\Omega^{0,1}(X, \End(E))$. Furthermore, we can pull back holomorphic structures along bundle homomorphisms of $E$ and so this gives an action of
\begin{gather*}\cG_{\CC}:= \Aut(E)\end{gather*}
on $\cC$ such that the orbits are precisely the isomorphism classes of holomorphic structures on $E$. The central group $\CC^* < \cG$, which corresponds to scalar multiples of the identity map on $E$, acts trivially on $\cC$.

In order to construct a quotient of such an action, Atiyah and Bott relate the $\cG_{\CC}$-space $\cC$ to a space of unitary connections. We recall that the bundle of frames of $E$ is a principal $\Gl_n(\CC)$-bundle, where $n = \rk E$. Since $\Ul(n)$ is a maximal compact subgroup of $\Gl_n(\CC)$, any principal $\Gl_n(\CC)$-bundle admits a reduction to $\Ul(n)$, which we can equivalently think of as a Hermitian metric $h$ on $E$. We can thus fix a Hermitian metric $h$ on $E$.

\begin{defn}An affine connection on $E$ is a linear map $\nabla \colon \Omega^0(X,E) \ra \Omega^1(X,E)$ satisfying the Leibniz rule. We say $\nabla$ is $h$-unitary if $dh(s_1,s_2) = h(\nabla(s_1),s_2)+ h(s_1, \nabla(s_2))$ for all sec\-tions~$s_i$ of~$E$.
\end{defn}

Let $\cA= \cA(E,h)$ denote the space of $h$-unitary affine connections on $E$; this is an infinite-dimensional complex affine space which is modelled on $\Omega^1(X, \End(E,h))$, where $\End(E,h)$ denotes the bundle of $h$-skew Hermitian endomorphisms of $E$. We can also view $\cA$ as the space of connections on the principal $\Ul(n)$-bundle associated to $(E,h)$.

\begin{defn}Let $\cG := \Aut(E,h)$ denote the $h$-unitary automorphisms of $E$; then $\cG_{\CC}=\Aut(E)$ is the complexification of $\cG$. We call $\cG$ the \emph{unitary gauge group} and $\cG_{\CC}$ the \emph{complex gauge group}.
\end{defn}

The unitary gauge group $\cG$ acts on $\cA$ of unitary connections and we can relate this to the action of the complex gauge group $\cG_{\CC}$ on $\cC$ using the following isomorphism.

\begin{Lemma}[Atiyah--Bott isomorphism]There is an isomorphism $\cA(E,h) \ra \cC(E)$ given by $\nabla \mapsto \nabla^{(0,1)}$, which we view as a Dolbeault operator on $E$.
\end{Lemma}

As we are working on a curve, there is no integrability condition and $\nabla^{(0,1)}$ defines a holomorphic structure on $E$. The inverse is given by taking the Chern connection $\nabla_{\ov{\partial}_E,h}$ associated to a~holomorphic structure $\ov{\partial}_E$ on $E$ and a Hermitian metric $h$. Locally the Atiyah--Bott isomorphism corresponds to the isomorphism \begin{gather*}\Omega^{1}(\fu(n)) \cong \Omega^{0,1}(\fg \fl_n).\end{gather*}

Although the space $\cA$ and the isomorphism $\cA \cong \cC$ both depend on the choice of Hermitian metric $h$, we can identify the space of Hermitian metrics on $E$ with $\Aut(E)/ \Aut(E,h) = \cG_\CC/\cG$. Thus any two Hermitian metrics on $E$ are related by a complex gauge transformation.

The space $\cA$ has the structure of a smooth symplectic manifold: if we identify $T_{\nabla} \cA \cong \Omega^1(X,\End(E,h))$, then $\omega_{\RR} \colon T_{\nabla} \cA \times T_{\nabla} \cA \ra \RR$ is defined by
\begin{gather*} \omega_{\RR}(\beta,\gamma) := \int_X \tr(\beta \wedge \gamma), \end{gather*}
where $\tr(\beta \wedge \gamma) \in \Omega^2(X)$. In this infinite-dimensional setting, $\omega_{\RR}$ being non-degenerate means that it induces an injection $T_{\nabla} \cA \ra T_{\nabla}^* \cA$. The inner product given by the trace also induces an isomorphism
\begin{gather*} \Lie \cG^* = \Omega^0(X, \End(E,h))^* \cong \Omega^2(X,\End(E,h)). \end{gather*}

We recall that the curvature of $\nabla \in \cA$ is the form $F_\nabla:= \nabla^2 \in \Omega^2(X, \End(E,h))$.

\begin{Lemma}The $\cG$-action on $\cA$ is symplectic with moment map $\mu \colon \cA \ra \Omega^2(X, \End(E,h))\cong \Lie \cG^*$ given by taking the curvature $($modulo a sign$)$: $\mu(\nabla) = - F_\nabla$.
\end{Lemma}

The sign here appears due to our sign conventions for the infinitesimal lifting property of the moment map. We leave the verification of this infinitesimal lifting property and the equivariance of $\mu$ as an exercise.

To construct a symplectic reduction of the $\cG$-action on $\cA$, we need to take the level set at a~coadjoint fixed point. Since the Lie algebra of the unitary group has centre $Z(\fu(n)) = {\rm i} \RR I_n$, we similarly have that all imaginary scalar multiplies of the identity map on $E$ are central in $\Lie \cG$. Fix a Riemannian metric on~$X$ whose associated volume form induces the given orientation on~$X$; then there is an associated Hodge star operator $\star \colon \Omega^k(X) \ra \Omega^{2-k}(X)$. Using this Hodge star operator, we can view the moment map taking values in $\Lie \cG$ by
\begin{gather*} \mu \colon \ \cA \ra \Lie \cG, \qquad \nabla \mapsto -\star F_\nabla. \end{gather*}

\begin{defn}A $h$-unitary connection $\nabla$ on $E$ is projectively flat if $\star F_\nabla \in \Omega^0(X, \End(E,h))$ is a constant element in the centre of $\fu(n)$; that is, an imaginary scalar multiple of $\Id_E$.
\end{defn}

In fact, the scalar appearing for such projectively flat connections is related to the slope $\mu(E)$.

\begin{exer}Using the fact that the degree of $E$ can be defined using the curvature $F_\nabla$ of any connection on $E$ via
\begin{gather*} \deg (E) := \int_X \frac{{\rm i}}{2\pi} \tr( F_\nabla), \end{gather*}
prove that if $\nabla$ is a projectively flat $h$-unitary connection (i.e., $\star F_\nabla= - {\rm i} \mu \mathrm{Id}_E$ for some constant $\mu \in \RR$), then the constant $\mu$ is equal to the slope of $E$ (provided we normalise our Riemmannian metric so the integral of its associated volume form over~$X$ is equal to~$2 \pi$).
\end{exer}

The symplectic reduction of the $\cG$-action on $\cA$ at the central value ${\rm i}\mu(E) \Id_E \in \Lie \cG$ is a~moduli space
\begin{gather*} \cM_{E,h}^{\text{proj.\ flat}}:=\mu^{-1}({\rm i}\mu(E) \Id_E) / \cG \end{gather*}
for unitary gauge equivalence classes of projectively flat $h$-unitary connections on $E$. In fact, as $\cA\cong \cC$ has a compatible complex structure, it is naturally an infinite-dimensional K\"{a}hler manifold and so the associated moduli space inherits a K\"{a}hler structure if $\cG/\Ul(1)$ acts freely on the level set $\mu^{-1}({\rm i}\mu(E) \Id_E)$, which is the case if $E$ has coprime rank and degree.

In order to relate this symplectic reduction with holomorphic structures, we need the following definition.

\begin{defn}A Hermitian--Einstein connection on a complex vector bundle $E$ is a projectively flat affine connection that is unitary for some Hermitian metric on $E$.
\end{defn}

The moduli space of semistable vector bundles is homeomorphic to the moduli space of representations $\pi_1(X) \ra \Ul(n)$ by the Narasimhan--Seshadri correspondence \cite{ns}. An alternative gauge theoretic interpretation of this result was provided by Donaldson \cite{donaldsonNS} and Uhlenbeck and Yau \cite{UY} by relating the modu\-li space of projectively flat $h$-unitary connections on $E$ to the mo\-duli space of holomorphic structures on $E$; this is called the Kobayashi--Hitchin correspondence.

\begin{Theorem}[Kobayashi--Hitchin correspondence \cite{donaldsonNS,ns,UY}] A holomorphic vector bundle~$\cE$ is slope polystable if and only if its underlying complex vector bundle admits a Hermitian--Einstein connection. Moreover, this connection is unique up to unitary gauge transformations.
\end{Theorem}

A holomorphic structure is slope semistable if and only if its $\cG_{\CC}$-orbit closure contains a~holomorphic structure that is polystable; let $\cC^{\text{ss}}$ (resp.~$\cC^{\text{ps}}$) denote the set of semistable (resp.\ polystable) holomorphic structures. By the Kobayashi--Hitchin correspondence, every point in~$\cC^{\text{ss}}$ has $\cG_{\CC}$-orbit closure that meets $\mu^{-1}({\rm i}\mu(E) \Id_E)$ in a unique $\cG$-orbit. The inclusion
\begin{gather*}\mu^{-1}({\rm i}\mu(E) \Id_E) \hookrightarrow \cC^{\text{ss}}\end{gather*}
induces a real-analytic isomorphism between the moduli space of projectively flat unitary connections and the moduli space of $S$-equivalence classes of semistable holomorphic bundles $\cC^{\text{ss}} /\!/ \cG_{\CC} \simeq \cC^{\text{ps}}/\cG_{\CC}$. This homeomorphism can be viewed as an infinite-dimensional version of the Kempf--Ness theorem, as it relates the symplectic reduction of the $\cG$-action on $\cA$ with the $S$-equivalence classes of orbits of the complexified group $\cG_{\CC}$ in the semistable locus $\cC^{\text{ss}}$.

\section{Hyperk\"{a}hler analogues of these moduli spaces}\label{sec hyperkahler}

\subsection{Algebraic symplectic quiver varieties}

Throughout this section we assume that $k$ is a field of characteristic different from $2$.

\begin{defn}[doubled quiver]The double of a quiver $Q = (V,A,h,t)$ is the quiver $\overline{Q}=(V,\overline{A}, h,t)$ where $\overline{A} = A \sqcup A^*$ for $ A^* := \{ a^* \colon h(a) \ra t(a)\}_{a \in A}$.
\end{defn}

One key motivation for introducing the doubled quiver is that
\begin{gather*}\RepQbar = \Rep \times \Rep^* \cong T^*\Rep\end{gather*}
is an algebraic symplectic variety, with the Liouville symplectic form $\omega$ on this cotangent bundle. Explicitly, if $X=(X_a,X_{a^*})_{a \in A}$ and $Y = (Y_a,Y_{a^*})_{a\in A}$ are points in $\RepQbar$, then
\begin{gather}\label{Liouville form} \omega(X,Y) = \sum_{a \in A} \tr(X_aY_{a^*} - X_{a^*}Y_a).\end{gather}

The action of $\GL_d$ on $\RepQbar$ preserves this symplectic form and there is an algebraic moment map $ \mu \colon\RepQbar \ra \LieGL_d^*:= \Lie (\GL_d)$; explicitly, for $X \in \RepQbar$ and $B \in \LieGL_d$ we have
\begin{gather}\label{eqn alg mmap quiver}
\mu(X) \cdot B = \sum_{a \in A} \tr(X_{a^*} (B_X)_a)= \sum_{a \in A} \tr(X_{a^*} (B_{h(a)}X_a - X_a B_{t(a)})),
\end{gather}
where $B_{X}= (B_{h(a)} X_a - X_a B_{t(a)})_{a \in A}$ is the infinitesimal action of $B$ on $(X_a)_{a \in A}$. This algebraic moment map is a $\GL_d$-equivariant morphism satisfying the infinitesimal lifting property
$ d_X \mu(\eta) \cdot B = \omega(B_X, \eta)$. The Killing form on the Lie algebra of each general linear group induces an identification $\LieGL_d \cong \LieGL_d^*$, so we can view the moment map as a morphism $\mu \colon \RepQbar \ra \LieGL_d$ given by $\mu(X) = \sum\limits_{a \in A} [X_a,X_{a^*}]$.

The group $\G= \GL_d /\Delta$ has Lie algebra $\LieG := \{ (B_v)_{v \in V} \in \LieGL_d \colon \sum_{v \in V} \tr(B_v) = 0 \}$ consisting of tuples of matrices with total trace zero. We note that the image of this moment map lies in~$\LieG$.

\begin{defn}Let $\chi\colon \GL_d \ra \GG_m$ be a character and let $\eta \in \LieGL_d$ be a coadjoint fixed point; then $\GL_d$ acts on $\mu^{-1}(\eta)$ by the equivariance property of the moment map. The algebraic
symplectic reduction of the $\GL_d$-action on $\RepQbar$ at $(\chi,\eta)$ is the GIT quotient $\mu^{-1}(\eta) /\!/_{\chi} \GL_d$.
\end{defn}

If GIT semistability and stability for the $\G$-action on $\mu^{-1}(\eta)$ with respect to $\chi$ agree, then the variety $\mu^{-1}(\eta)/\!/_{\chi} \GL_d$ is a smooth algebraic symplectic variety, with algebraic symplectic form induced by the Liouville form on $T^*\rep_d(Q)$ by an algebraic version of the Marsden--Weinstein theorem (see~\cite{ginzburg}).

The closed subvariety $\mu^{-1}(\eta) \hookrightarrow \RepQbar$ induces a closed immersion
\begin{gather*} \mu^{-1}(\eta)/\!/_{\chi_\theta} \GL_d \hookrightarrow \cM^{\theta\text{-ss}}_{d}\big(\overline{Q}\big).\end{gather*}
Nakajima quiver varieties can also be constructed in this manner (for example, see \cite{ginzburg}).

\begin{rmk}
A tuple $\eta=(\eta_v)_{v \in V} \in k^V$ determines an adjoint fixed point $\eta = (\eta_v \text{Id}_{d_v})_{v \in V} \in \LieGL_d(k)$. Moreover, we have that $\mu^{-1}(\eta) = \rep_d(\overline{Q},\cR_\eta)$ is the subvariety of $\rep_d\big(\overline{Q}\big)$ of representations satisfying the relations
 \begin{gather*}\cR_\eta = \bigg\{ \sum_{a\colon t(a) = v} M_a M_{a^*} - \sum_{a\colon h(a) =v} M_{a^*}M_a = \eta_v I_{d_v} \ \forall\, v\in V\bigg\}.\end{gather*}
Hence $\mu^{-1}(\eta)/\!/_{\chi_\theta} \GL_d$ is the moduli space of $\theta$-semistable $d$-dimensional representations of $(\overline{Q},\cR_\eta)$. Under the correspondence between $k$-representations of $\ov{Q}$ and modules over the path algebra $k(\ov{Q})$, the representations satisfying the relations $\cR_\eta$ correspond to modules over certain quotients of $k(\ov{Q})$. More precisely, the category of $k$-representations of $(\overline{Q},\cR_\eta)$ corresponds to the category of modules over the algebra
\begin{gather*}\Pi_\eta:= k(\ov{Q})/\bigg( \sum_{a \in A} [a,a^*] - \sum_{v \in V} \eta_v e_v\bigg),\end{gather*}
where $e_v$ denotes the trivial path at $v$. The algebra $\Pi_0$ at $\eta = 0$ is called the \emph{preprojective algebra} of~$Q$.
\end{rmk}

\begin{exer}\label{exer nec cond for relations via trace}Prove that a necessary condition for the existence of a $k$-representation of $(\overline{Q},\cR_\eta)$ of dimension $d$ is that $\eta \cdot d = \sum\limits_{v \in V} \eta_v d_v = 0$ holds in $k$ (Hint: take traces of the equations defining these relations).
\end{exer}

In fact, the equation $\eta \cdot d = 0$ in $k$ ensures that $\eta \in \LieGL_d(k)$ actually lies in $\LieG(k)$, which is a~necessary condition for $\mu^{-1}(\eta)(k)$ to be non-empty, as $\mu$ has image in $\LieG$.

\begin{Lemma}\label{lemma stab trivial}Let $\theta$ be a generic stability parameter with respect to $d$; then for a field $k$ of charac\-te\-ris\-tic zero or sufficiently large prime characteristic, we have $\rep_d\big(\overline{Q},\cR_\theta\big) = \rep_d\big(\overline{Q},\cR_\theta\big)^{\theta\text{\rm -ss}} $ $= \rep_d\big(\overline{Q},\cR_\theta\big)^{\theta\text{\rm -gs}}$.
\end{Lemma}
\begin{proof}It suffices to prove this claim after base changing to an algebraic closure of $k$ and so we can assume $k$ is algebraically closed and check the statement on closed points. By Exercise~\ref{exer nec cond for relations via trace}, if there exists a $d'$-dimensional $k$-representation of $\big(\overline{Q},\cR_\theta\big)$, then $\theta \cdot d' = 0$ holds in $k$. Since~$\theta$ is generic with respect to~$d$, then for all dimension vectors $d' < d$, the equation $\theta \cdot d' \neq 0$ also holds in $k$, when $k$ has characteristic $0$ or $p \gg 0$. Hence any $k$-representation of $\big(\overline{Q},\cR_\theta\big)$ of dimen\-sion~$d$ is $\theta$-semistable (and $\theta$-stable), as it has no subrepresentations, which proves the claim.
\end{proof}

\subsection{Hyperk\"{a}hler quiver varieties}\label{sec HK quiver}

Over the complex numbers, the algebraic symplectic reduction has a hyperk\"{a}hler structure, as it can be interpreted as a hyperk\"{a}hler reduction via the Kempf--Ness theorem. Indeed the cotangent bundle of a complex vector space is naturally hyperk\"{a}hler and the action of the maximal compact subgroup $\U_d < \GL_d$ on $\RepQbar \cong T^*\Rep$ preserves this hyperk\"{a}hler structure, so one can instead perform a hyperk\"{a}hler reduction~\cite{hitchinetal}.

More generally, we can perform a hyperk\"{a}hler reduction of the cotangent bundle of a complex vector space $M = \CC^n$. A Hermitian form $H$ on $M$ gives a symplectic form on $M$ and an identification $M \cong M^*$. Using the identification $\CC \times \CC \cong \HH$ given by $(m,\alpha) \mapsto x - {\rm j}\alpha$ in each coordinate, we obtain an identification $T^*M \cong M \times M \cong \HH^n$ which we can use to equip~$T^*M$ with a hyperk\"{a}hler structure. More precisely, we obtain complex structures ${\rm I}$, ${\rm J}$ and ${\rm K}$ corresponding to right multiplication by~${\rm i}$, ${\rm j}$ and ${\rm k}$ on $\HH^n$ and the hyperk\"{a}hler metric $g$ is the real part of the Quaternionic inner product
\begin{gather*} Q \colon \ \HH^n \times \HH^n \ra \HH, \qquad (z,w) \mapsto \sum_{l=1}^n z_l w_l^\dagger, \end{gather*}
where $w_l^\dagger$ denotes the quaternionic conjugate. Thus we can write $ Q = g - {\rm i} \omega_{\rm I} - {\rm j}\omega_{\rm J} - {\rm k} \omega_{\rm K}$, such that
\begin{gather*} \omega_{\rm I}(-,-) = g({\rm I}-,-), \qquad \omega_{\rm J}(-,-) = g({\rm J}-,-) \qquad \text{and} \qquad \omega_{\rm K}(-,-) = g({\rm K}-,-).\end{gather*}
We thus obtain a hyperk\"{a}hler structure $(g,{\rm I},{\rm J},{\rm K}, \omega_{\rm I}, \omega_{\rm J}, \omega_{\rm K})$ on $T^*M$. We often write the K\"{a}hler structures as a pair $(\omega_{\RR},\omega_{\CC})$, where $\omega_{\RR} = \omega_{\rm I}$ and $\omega_{\CC} = \omega_{\rm J} + {\rm i} \omega_{\rm K}$, which is the Liouville algebraic symplectic form on~$T^*M$.

Now suppose we additionally have a linear action of a complex reductive group $G$ on $M = \CC^n$ and a maximal compact subgroup $K < G$ for which the Hermitian form $H$ is invariant. Then the induced $K$-action on $T^*M$ preserves the symplectic forms $(\omega_{\RR},\omega_{\CC})$ and there is a hyperk\"{a}hler moment map $\mu_{\rm HK} := (\mu_\RR,\mu_{\CC})$, where $\mu_{\RR}\colon T^*M \ra \fk^*$ is a smooth moment map for the $K$-action and $\mu_{\CC}\colon T^*M \ra \fg^*$ is an algebraic moment map for the $G$-action. Explicitly, we have
\begin{gather*} \mu_{\RR}(m,\alpha) \cdot B := \frac{{\rm i}}{2}(H(B_m,m) - H(B_\alpha,\alpha) ) \qquad \mathrm{and} \qquad \mu_{\CC}(m,\alpha) \cdot C = \alpha(C_m), \end{gather*}
where $B \in \fk$ and $C \in \fg$ and $(m,\alpha) \in T^*M$.

For a pair $(\chi,\eta) \in \fk^* \times \fg^*$ of coadjoint fixed points, the \emph{hyperk\"{a}hler reduction} of the $K$-action on $T^*M$ at $(\chi,\eta)$ is the topological quotient of the $K$-action on $\mu_{\rm HK}^{-1}(\chi,\eta):=\mu_{\RR}^{-1}(\chi) \cap \mu_{\CC}^{-1}(\eta)$. By the Kempf--Ness theorem, this hyperk\"{a}hler reduction is homeomorphic to the GIT quotient of $G$ acting on $\mu^{-1}(\eta)$ with respect to the character of $G$ obtained from $\chi$ by exponentiating and complexifying; thus
\begin{gather*}\mu^{-1}_{\rm HK}(\chi,\eta)/K \cong \mu_{\CC}^{-1}(\eta)/\!/_{\chi} G.\end{gather*}
In particular, if $K$ acts with finite stabilisers on this level set of the hyperk\"{a}hler moment map, then this hyperk\"{a}hler reduction inherits an orbifold hyperk\"{a}hler structure \cite{hitchinetal}.

Let us apply this to the quiver setting: we have $M = \Rep$ and $G = \GL_d$ and we take the Hermitian form $H$ on $\Rep$ as in Section~\ref{sec sympl constr}, which is invariant under the action of $K = \U_d$. The hyperk\"{a}hler metric $g$ on $T^*M \cong \RepQbar$ is given by
\begin{equation*}
g(X,Y) = \operatorname{Re} \bigg( \sum_{a \in \overline{A}} \tr \big(X_a Y^\dagger_a\big) \bigg);
\end{equation*}
thus, $\omega_\RR= \omega_{\rm I}$ is given by
\begin{gather*} \omega_{\RR}(X,Y) = \operatorname{Im} \bigg( \sum_{a \in \overline{A}} \tr \big({X}^\dagger_a Y_a\big) \bigg) \qquad \text{and} \qquad \omega_{\CC} = \omega_{\rm J} + {\rm i} \omega_{\rm K}\end{gather*} is the Liouville algebraic symplectic form $\omega$ described in~\eqref{Liouville form}. Moreover, $\mu_{\CC}\colon \RepQbar \ra \LieGL_d^*$ is the algebraic moment map $\mu$ given by \eqref{eqn alg mmap quiver} and $\mu_{\RR}\colon \RepQbar \ra \fu_{d}^*$ is given by
\begin{gather*} \mu_{\RR}(X) \cdot B = \frac{{\rm i}}{2} \sum_{a \in \overline{A}} \tr \big(B_{h(a)}X_a {X}^\dagger_a - B_{t(a)} {X}^\dagger_{a} {X}_{a}\big). \end{gather*}
Via the identification $\fu_{d} \cong \fu_{d}^*$, we obtain a map $\mu_{\RR}^* \colon \RepQbar \ra \fu_{d}$ given by
\begin{gather*} \mu_{\RR}^*(X) = \frac{{\rm i}}{2}\sum_{a \in \overline{A}} \big[X_a, {X}^\dagger_a\big].\end{gather*}
If $\chi_\theta$-semistability coincides with $\chi_\theta$-stability on $\mu^{-1}(\eta)$, then we obtain a hyperk\"{a}hler structure on the algebraic variety $\mu^{-1}(\eta)/\!/_{\chi_\theta} \GL_d$ via the Kempf--Ness homeomorphism.

\begin{rmk}\label{rmk HK analogue}Let $\theta$ be a generic stability parameter with respect to $d$. Then $\theta$-semistability and $\theta$-stability for $\CC$-representations of $Q$ coincide and the moduli space $\Mod$ of $\theta$-semistable $\CC$-representations of $Q$ is a smooth algebraic variety with a natural K\"{a}hler structure coming from the Kempf--Ness homeomorphism. As in work of Proudfoot \cite{proudfoot}, we can view the hyperk\"{a}hler reduction of $\RepQbar$ at $(\theta,0)$ as a hyperk\"{a}hler analogue of the K\"{a}hler manifold $\Mod$ in the sense that
\begin{gather*} T^* \Mod \subset \cM_d^{\theta\text{-ss}}\big(\ov{Q},\cR_0\big)=\mu^{-1}_{\CC}(0)/\!/_{\chi_\theta} \GL_d \simeq \big(\mu^{-1}_{\RR}(\theta) \cap \mu_{\CC}^{-1}(0)\big)/\U_d \end{gather*}
is contained as a dense open subset (provided $\Mod \neq \varnothing$). Indeed, if $\pi \colon \Rep^{\theta\text{-ss}} \ra \Mod$ denotes the GIT quotient, which is a principal $\G$-bundle as $\theta$ is generic, then for $X \in \Rep^{\theta\text{-ss}}$ we have a short exact sequence
\begin{gather*} 0 \ra T_X(\G \cdot X) \ra T_X \Rep \ra T_{\pi(X)} \Mod \ra 0 \end{gather*}
and dually
\begin{gather*} T^*_{\pi(X)} \Mod = \big\{ \xi \in T_X^*\Rep \colon \xi(A_X) = 0 \ \forall\, A \in \LieG \big\} \\
\hphantom{T^*_{\pi(X)} \Mod}{} = \big\{ \xi \in T_X^*\Rep \colon \mu_{\CC}(X,\xi) = 0 \big\}.\end{gather*} Thus, we have
\begin{gather*} T^* \Mod \cong \big\{ (X,\xi) \in \mu^{-1}_{\CC}(0) \subset T^*\Rep \colon X \!\in\! \Rep^{\theta\text{-ss}} \big\}/\G \subset \mu^{-1}_{\CC}(0)/\!/_{\chi_\theta} \G. \end{gather*}
\end{rmk}

\subsection{Moduli spaces of Higgs bundles}\label{sec Higgs}

Let $X$ be a smooth projective complex curve and fix a rank $n$ and degree $d$. Then the moduli space of Higgs bundles can be viewed as a hyperk\"{a}hler analogue of the moduli space $\cM= \cM_X^{\text{ss}}(n,d)$ of semistable vector bundles of rank $n$ and degree $d$ over $X$. By the gauge theoretic construction, $\cM$ is homeomorphic to a symplectic reduction of the unitary gauge group $\cG$ on the space of unitary connections $\cA$. In this section, we upgrade this to a hyperk\"{a}hler setting by considering the action of $\cG$ on the cotangent bundle $T^*\cA $. This will give us a moduli space $\cH=\cH_X^{\text{ss}}(n,d)$ of semistable Higgs bundles which contains the cotangent bundle $T^*\cM$ as a dense open subset.

The deformation theory of vector bundles give a description of the tangent spaces to $\cM$ at an isomorphism class $[\cE]$ of a stable locally free sheaf:
\begin{gather*}T_{[\cE]}\cM \cong \Ext^1(\cE,\cE) \cong H^1(X,\cE{\rm nd}(\cE))\end{gather*}
and by Serre duality, we have $T^*_{\cE}\cM \cong H^0(X, \cE{\rm nd}(\cE) \otimes \omega_X)$. The elements in this cotangent space are holomorphic Higgs fields on $\cE$ in the sense of the following definition.

\begin{defn}A holomorphic Higgs bundle over $X$ is a pair $(\cE,\Phi)$ consisting of a holomorphic vector bundle $\cE$ over $X$ and a holomorphic homomorphism $\Phi \colon \cE \ra \cE \otimes \omega_X$ called a Higgs field. We define slope semistability for $(\cE,\Phi)$ by checking the inequality of slopes for all holomorphic Higgs subbundles (i.e., holomorphic subbundles $\cE' \subset \cE$ that are $\Phi$-invariant in the sense that $\Phi(\cE') \subset \cE' \otimes \omega_X$).
\end{defn}

\begin{rmk}For coprime rank and degree, semistability and stability coincide for Higgs bundles.
\end{rmk}

We recall that the gauge theoretic construction of the moduli space $\cM=\cM^{\text{ss}}_X(n,d)$ of semistable vector bundles is as the space of $S$-equivalence classes of complex gauge orbits in the space of semistable holomorphic structures $\cC^{\text{ss}}$
\begin{gather*} \cM^{\text{ss}}_X(n,d)= \cC^{\text{ss}} /\!/ \cG_{\CC} \simeq \cC^{\text{ps}}/\cG_{\CC}\end{gather*}
and by the Kobayashi--Hitchin correspondence, this space is homeomorphic to the symplectic reduction of $\cG$ on the space of unitary connections $(\cA,\omega_{\RR})$.

Let us fix a complex vector bundle $E$ and Hermitian metric $h$. The space $\cC$ of holomorphic structures (or Dolbeault operators $\ov{\partial}_E$) on $E$ has cotangent bundle $T^*\cC \cong \cC \times \Omega^{1,0}(X,\End(E))$. We write elements of $T^*\cC$ as pairs $(\ov{\partial}_E,\Phi)$, where $\Phi \in \Omega^{1,0}(X,\End(E))$ defines a~(not ne\-ces\-sarily holomorphic) Higgs field. The cotangent space $T^*\cC$ is an affine space modelled on $\Omega^{0,1}(X, \End(E)) \times \Omega^{1,0}(X, \End(E))$ and its Liouville form is a holomorphic symplectic form $\omega_{\CC}$ for the complex structure ${\rm I}$ (coming from the complex structure on $E \ra X$). Moreover, the natural $\cG_{\CC}$-action on $T^*\cC$ admits a holomorphic moment map $\mu_{\CC}\colon T^*\cC \ra \Lie \cG_{\CC}^*$ given by
\begin{gather*} \mu_{\CC}(\ov{\partial}_E,\Phi):= 2{\rm i} \ov{\partial}_E\Phi.\end{gather*}
We note that the zero level set of this moment map consists of pairs $(\ov{\partial}_E,\Phi)$, where $\Phi$ defines a~holomorphic Higgs field on the holomorphic bundle $\cE=(E,\ov{\partial}_E)$; that is $(\cE,\Phi)$ is a~holomorphic Higgs bundle. We let $\mu^{-1}_{\CC}(0)^{\text{ss}}$ denote the subset of slope semistable holomorphic Higgs bundles and we define the moduli space of Higgs bundles as the holomorphic symplectic reduction
\begin{gather*} \cH_C^{\text{ss}}(n,d):=\mu^{-1}_{\CC}(0)^{\text{ss}} /\!/ \cG_{\CC} \end{gather*}
equal to the set of $S$-equivalence classes of semistable $\cG_{\CC}$-orbits in $\mu^{-1}_{\CC}(0)$ (or equivalently, the set of polystable stable $\cG_{\CC}$-orbits).

In fact, $T^*\cC$ is naturally an infinite dimensional flat hyperk\"{a}hler manifold, as via the Atiyah--Bott isomorphism $\cC \cong \cA$, we can equip $T^*\cC$ with a real symplectic form $\omega_{\RR}$ and associated K\"{a}hler metric (coming from the real symplectic form $\omega_{\RR}$ on $\cA$ in Section~\ref{sec gauge constr M}). More precisely, we can identify $T^*\cA \cong \cA \times \Omega^1(X, \End(E,h))$ on which the unitary gauge group $\cG$ naturally acts. We note that there is an isomorphism
\begin{gather*}T^*\cC \cong \cA \times \Omega^1(X, \End(E,h)) \qquad \text{given by} \quad (\ov{\partial}_E,\Phi) \mapsto \big(\nabla_{\ov{\partial}_E,h},\Phi - \Phi^*\big), \end{gather*}
where $\nabla_{\ov{\partial}_E,h}$ denotes the Chern connection associated to $(\ov{\partial}_E,h)$. In fact, $(\ov{\partial}_E,\Phi) \in T^*\cC$ determines a $\Gl_n(\CC)$-connection $\nabla_{\ov{\partial}_E,h} + \Phi + \Phi^*$ on the associated principal $\Gl_n(\CC)$-bundle. Therefore, we can think of this cotangent bundle as the space of complex connections on $E$. The real moment map for the induced $\cG$-action on $T^*\cC$ is given by
\begin{gather*} \mu_{\RR}(\ov{\partial}_E,\Phi)= -F_{\ov{\partial}_E} - [\Phi,\Phi^*], \end{gather*}
where $F_{\ov{\partial}_E}$ denote the curvature of the associated Chern connection $\nabla_{\ov{\partial}_E,h}$ and $[\alpha,\beta]:= \alpha \wedge \beta + \beta \wedge \alpha$ is the extension of the Lie bracket to Lie algebra-valued forms.

In particular, we have a hyperk\"{a}hler moment map $\mu_{\rm HK} = (\mu_{\RR},\mu_{\CC})$ for the $\cG$-action on $T^*\cC$. The zero level set of the hyperk\"{a}hler moment map is the set of solutions of Hitchin's self-duality equations \cite{Hitchin}. Any such solution determines an associated $\Gl_n(\CC)$-connection which is flat and thus requires $d = 0$. To deal with vector bundles of non-zero degree, we take the level set at the value $(\star {\rm i}\mu(E) \text{Id}_E, 0) \in \Lie \cG^* \times \Lie \cG_{\CC}^*$; then consider the hyperk\"{a}hler reduction
\begin{gather*} \cM_{\mathrm{Hit}}:=\big(\mu^{-1}_{\RR}(\star {\rm i}\mu(E) \text{Id}_E) \cap \mu_{\CC}^{-1}(0)\big) /\cG, \end{gather*}
which is a moduli space of solutions to Hitchin's equations (appropriately modified for $d \neq 0$) up to gauge equivalence. Then $\cM_{\mathrm{Hit}}$ admits a triple of holomorphic structures ${\rm I}$, ${\rm J}$ and ${\rm K}$ and a hyperk\"{a}hler metric on its smooth locus. If $n$ and $d$ are coprime, then $\cM_{\mathrm{Hit}}$ is a smooth hyperk\"{a}hler manifold. A generalisation of the Kobayashi--Hitchin correspondencen correspondence for vector bundles to Higgs bundles due to Hitchin \cite{Hitchin} and Simpson \cite{simpson_HiggsGm}
states that a holomorphic Higgs bundle $(\cE,\Phi)$ is slope polystable if and only if $(\cE,\Phi)$ admits a Hermitian metric $h$ such that
\begin{gather*} -(F_{\ov{\partial}_E} + [\Phi,\Phi^*]) =\star {\rm i}\mu(E)\mathrm{Id}_E.\end{gather*}
Hence, the complex structure $I$ on $\cM_{\mathrm{Hit}}$ gives the moduli space of Higgs bundles.

\begin{rmk}
Let $\cE $ be a semistable vector bundle corresponding to a point in $\cM:= \cM_X^{\text{ss}}(n,d)$; then for any $\Phi \in T_{[\cE]}^* \cM$, we note that $(\cE,\Phi)$ is a semistable Higgs bundle. In fact, we have an inclusion
\begin{gather*} T^*\cM \subset \cH\end{gather*}
and we can view $\cH$ as a hyperk\"{a}hler analogue of $\cM$ analogously to the notion of a hyperk\"{a}hler analogue of the GIT quotient of a complex affine space (see~\cite{proudfoot} and Remark~\ref{rmk HK analogue}). We recall that the quiver moduli space $\Mod$ has a hyperk\"{a}hler analogue given by the moduli space $\cM_d^{\theta\text{-ss}}\big(\ov{Q},\cR_0\big)$ of representations of the doubled quiver satisfying the equations $\cR_0$ imposed by zero level set of the complex moment map.
\end{rmk}

The inclusion $T^*\cM \subset \cH$ is strict in general, as there are unstable vector bundles which can be equipped with a Higgs field for which the associated Higgs pair is stable; for example, this is the case if there are no Higgs subbundles. Indeed we have the following example due to Hitchin~\cite{Hitchin}.

\begin{exer}\label{exer Higgs bundle indecom bundle not} Suppose that $X$ has genus at least 2 and that $\cL$ is a square root of~$\omega_X$. Then prove that $\cE= \cL \oplus \cL^{-1}$ is unstable as a vector bundle, but admits a Higgs field $\Phi$ such that $(\cE,\Phi)$ is stable.
\end{exer}

\subsection{Branes}\label{sec branes}

Branes are submanifold of hyperk\"{a}hler manifolds with particularly rich geometry (in the sense, that they are either Lagrangian or holomorphic with respect to a triple of K\"{a}hler structures). In this section, we summarise some constructions of branes in the quiver and bundle settings arising from fixed loci of automorphisms on these moduli spaces. We will use the language of branes as in \cite{Kapustin_Witten} as follows.

\begin{defn} A brane in a hyperk\"{a}hler manifold $(M,g,{\rm I},{\rm J},{\rm K},\omega_{\rm I},\omega_{\rm J},\omega_{\rm I})$ is a submanifold which is either holomorphic or Lagrangian with respect to each of the three K\"{a}hler structures on $M$. A brane is called of type $A$ (respectively $B$) with respect to a given K\"{a}hler structure if it is Lagrangian (respectively holomorphic) for this K\"{a}hler structure.
\end{defn}

\begin{exer}By using the quaternionic relations between the 3 complex structures, show that there are 4 possible types of branes: $BBB$, $BAA$, $ABA$ and $AAB$.
\end{exer}

We note that the brane-type depends on choosing a triple of K\"{a}hler structures (although often there is a natural choice). All triples of hyperk\"{a}hler structures can be related using hyperk\"{a}hler rotations.

\subsubsection{Branes in hyperk\"{a}hler quiver varieties}

Starting from a quiver $Q$, moduli spaces of representations of the doubled quiver $\ov{Q}$ (satisfying some relations) have a natural algebraic symplectic structure and, over $k = \CC$, a natural hyperk\"{a}hler structure, provided these varieties are smooth (cf.\ Section~\ref{sec HK quiver}). The study of branes in Nakajima quiver varieties was initiated in~\cite{fjm}, where the authors use involutions such as complex conjugation, multiplication by $-1$ and transposition, to construct different branes. In~\cite{HS_quiver_autos} we construct branes associated to quiver automorphisms in the following sense.

\begin{defn}For a quiver $Q = (V,A,h,t)$, a pair of automorphisms $\sigma = (\sigma_V \colon V \ra V, \sigma_A \colon A \ra A)$ is a
\begin{enumerate}\itemsep=0pt
\item[1)] covariant automorphism of $Q$ if $\sigma_A(a)\colon \sigma_V(t(a)) \ra \sigma_V(h(a))$ for all $a \in A$,
\item[2)] contravariant automorphism of $Q$ if $\sigma_A(a)\colon \sigma_V(h(a)) \ra \sigma_V(t(a))$ for all $a \in A$.
\end{enumerate}
\end{defn}

Under certain compatibility conditions of an automorphism $\sigma$ of $Q$ with the dimension vector $d$ and stability parameter $\theta$, we show this automorphism determines an automorphism of $\Mod$ and we describe the components of the fixed locus. In the hyperk\"{a}hler setting, for an automorphism of a doubled quiver $\ov{Q}$ we then describe the geometry of this fixed locus acting on the associated hyperk\"{a}hler reduction in the language of branes.

\begin{Theorem}[\cite{HS_quiver_autos}]Let $\sigma$ be an involution of $\ov{Q}$ such that $\sigma(a^*) = \sigma(a)^*$ for all $a \in A$. For choices of $d$, $\theta$ and $\eta$ that are $\sigma$-compatible, $\sigma$ induces an involution on $\cH:=\mu^{-1}(\eta)/\!/_{\chi_\theta} \GL_d$. If~$\theta$ is generic with respect to $d$, then $\cH$ is a smooth hyperk\"{a}hler manifold and the fixed locus has the following brane type
\begin{gather*} \begin{array}{|c|c|c|} \hline & \mathrm{if } \ \sigma(A) \subset A & \mathrm{if} \ \sigma(A) \subset A^* \\ \hline \cH^{\sigma} & BBB & BAA \\ \hline \cH^{\sigma \circ \tau} & ABA & AAB \\ \hline \end{array}\end{gather*}
where $\tau\colon \CC \ra \CC$ denote complex conjugation.
\end{Theorem}

In particular, we see that all four types of branes ($BBB$, $BAA$, $ABA$ and $AAB$) can be constructed as the fixed locus of an involution. In fact, we can also construct $BBB$-branes as fixed loci of a subgroup of quiver automorphisms of order higher than $2$. Moreover, we provide a~decomposition of these fixed loci using group cohomology and give moduli-theoretic description of each of the components appearing in these decompositions. For a~quiver involution $\sigma$ (or more generally a group of quiver automorphisms), the fixed loci components are described in terms of twisted equivariant quiver representations~\cite{HS_quiver_autos}, and for complex conjugation $\tau$, the components of the fixed locus are described in terms of real or quaternionic quiver representations~\cite{HS_galois}.

\subsubsection{Branes in Higgs moduli spaces}

The gauge theoretic construction of moduli spaces of Higgs bundles naturally generalises from the general linear group to any complex reductive group $G$. In this way, one obtains moduli spaces $\cH_G$ of $G$-Higgs bundles which inherit a hyperk\"{a}hler structure on their smooth locus. We let~${\rm I}$,~${\rm J}$ and ${\rm K}$ denote the complex structures as above, such that ${\rm I}$ corresponds to the original complex structure on~$X$ and gives the moduli space of Higgs bundles. Branes in $\cH_G$ have been constructed in \cite{BS1,BS2,BGP,BGPH} as fixed points sets of involutions on~$\cH_G$ associated to anti-holomorphic involutions on~$G$ and~$X$.

\begin{Theorem}[\cite{BS1,BS2,BGP,BGPH}] Let $\cH:=\cH_G$ be a smooth Higgs moduli space and $\sigma_G \colon G \ra G$ and $\sigma_X \colon X \ra X$ be anti-holomorphic involutions. Then there are induced involutions $\sigma_G$ and~$\sigma_X$ on $\cH$ such that
\begin{enumerate}\itemsep=0pt
\item[$i)$] $\cH^{\sigma_G}$ is a $BAA$-brane,
\item[$ii)$] $\cH^{\sigma_X}$ is a $ABA$-brane,
\item[$iii)$] $\cH^{\sigma_G \circ \sigma_X}$ is a $AAB$-brane.
\end{enumerate}
\end{Theorem}

In \cite{BS1,BS2,BGP,BGPH}, some components of the fixed loci have been given moduli-theoretic descriptions: $\cH^{\sigma_G}$ contains a moduli space of $G^{\sigma_X}$-Higgs bundles, $\cH^{\sigma_X}$ contains a component corresponds to representations of the orbifold fundamental group of $(X,\sigma_X)$, and components of $\cH^{\sigma_G \circ \sigma_X}$ can be described as moduli spaces of pseudo-real Higgs bundles.

Baraglia and Schaposnik \cite{BS2} conjecture that under Langlands duality, which relates the moduli spaces $\cH_G$ and $\cH_{G^L}$ of Higgs bundles for $G$ and its Langlands dual group $G^L$, the $BAA$-brane $\cH_G^{\sigma_G} \subset \cH_G$ corresponds to a $BBB$-brane $\cH_{H} \subset \cH_{G^L}$, where $H < G^L$ is a complex subgroup (the so-called Nadler group) corresponding to the involution $\sigma_G$.

\subsubsection{Open questions on branes in hyperk\"{a}hler moduli spaces}

The current techniques for the construction of branes in hyperk\"{a}hler moduli spaces involve taking the fixed locus of a finite group action. An interesting question is whether all branes can be constructed in this manner. It seems unlikely that this is the case, particularly if one uses the more general notion of brane which comes with a coherent sheaf supported on this subvariety. A related problem is to classify all automorphisms of hyperk\"{a}hler moduli spaces, in case new automorphisms arise and give new constructions of branes; for Higgs bundle moduli spaces, this analysis is performed in~\cite{baraglia}.

\section{Counting indecomposable objects and Betti numbers}\label{sec counting indecomp Betti}

For both quiver representations and vector bundles, there is a surprising link between the counts of absolutely indecomposable objects over finite fields and the Betti numbers of the (complex) hyperk\"{a}hler moduli spaces described above. This was first discovered for indivisible dimension vectors on quivers without loops by Crawley-Boevey and Van den Bergh~\cite{CBVdB}, and was motivated by a conjecture of Kac concerning the non-negativity of the coefficients in the polynomial $\cA_{Q,d}(q)$ counting absolutely indecomposable $d$-dimensional $\FF_q$-representations of~$Q$. The proof of Kac's positivity conjecture for arbitrary $Q$ and $d$ was given by Hausel, Letellier and Rodriguez-Villegas~\cite{HLRV}.

In these works, the key idea is to provide a cohomological interpretation of the coefficients of~$\cA_{Q,d}(q)$. In \cite{CBVdB}, this cohomological interpretation for indivisible dimension vectors is as the Betti numbers of hyperk\"{a}hler quiver varieties associated to the doubled quiver. In \cite{HLRV}, for an arbitrary~$Q$ and~$d$, by attaching legs to each vertex in~$Q$, they obtain as associated quiver~$\tilde{Q}_d$ and indivisible dimension vector~$\tilde{d}$. The generic algebraic symplectic reduction for this extended quiver is smooth, and its compactly supported cohomology admits an action by a finite group generated by the reflections at the new vertices. They interpret the coefficients of the Kac polynomials as the dimensions of the sign isotypical component of this cohomology by ma\-king use of an arithmetic Fourier transform. Furthermore, they give similar cohomological interpretations of the refined Donaldson--Thomas invariants of quivers.

This work on quiver representations inspired Schiffmann \cite{schiffmann} to formulate and prove an analogous statement for bundles in the coprime setting, which lead to formulae for the Betti numbers of moduli spaces of Higgs bundles and eventually gave a proof of the conjectures of Hausel and Rodriguez-Villegas \cite{HRV} on these Betti numbers.

In this section, we focus of the proof of this result in the quiver setting following the arguments of Crawley-Boevey and Van den Bergh. After this proof, we discuss the parallel argument in the bundle setting.

\subsubsection*{The statement in the quiver setting}

Let $Q$ be a quiver and $d$ be a dimension vector. Motivated by questions in representation theory of quiver representations, Kac studied the properties of the count of absolutely indecomposable quiver representations over finite fields \cite{Kac1,Kac2}.

\begin{defn}
Let $q$ be a prime power. Then a quiver representation $W$ over $\FF_q$ is absolutely indecomposable if $W \otimes_{\FF_q} \overline{\FF_q}$ is an indecomposable quiver representation. Let $\cA_{Q,d}(q)$ denote the number of isomorphism classes of absolutely indecomposable representations of $Q$ over $\FF_q$ with dimension vector $d$.
\end{defn}

\begin{exer} For an $\FF_q$-representation $W$ of $Q$ of dimension $d$ prove the following.
\begin{enumerate}\itemsep=0pt
\item[a)] If $W$ is absolutely indecomposable, then $W$ is indecomposable.
\item[b)] The converse holds if $d$ is an indivisible dimension vector.
\end{enumerate}
\end{exer}

Kac proved that $\cA_{Q,d}(q)$ is a polynomial in $q$ with integer coefficients, and conjectured that the coefficients are natural numbers (see Section~\ref{sec Kac results} below).
In order to formulate the result required for the proof of this conjecture for quivers without loops and indivisible dimension vectors given Crawley-Boevey and Van den Bergh \cite{CBVdB}, we recall that there is an algebraic moment map $\mu\colon \RepQbar \ra \LieGL_d$ for the $\GL_d$-action on the space of representations of the doubled quiver over any field $k$. The zero level set of the moment map defines relations~$\cR_0 $ on the doubled quiver~$\overline{Q}$ such that $\mu^{-1}(0) = \rep_d\big(\ov{Q},\cR_0\big)$ is the space of representations of the preprojective algebra. Choose a generic stability parameter $\theta$ with respect to~$d$; then $\theta$-semistability and $\theta$-stability (and $\theta$-geometric stability) coincide for $d$-dimensional $k$-representations of $Q$ (and also for the double quiver $\overline{Q}$). The associated algebraic symplectic reduction
\begin{gather*} X_0:= \mu^{-1}(0)/\!/_{\!\chi_\theta} \G = \cM_d^{\theta\text{-ss}}\big(\ov{Q},\cR_0\big) \end{gather*}
is a moduli space of $\theta$-stable $d$-dimensional representations of $\big(\ov{Q},\cR_0\big)$. Moreover, as semistability coincides with stability and all stable representations are simple, $X_0$ is a smooth algebraic variety which inherits an algebraic symplectic structure from $\RepQbar$. If $k = \CC$, then $X_0$ is a (non-compact) hyperk\"{a}hler manifold such that $T^*\Mod \subset X_0$.

\begin{Theorem}[Crawley-Boevey and Van den Bergh \cite{CBVdB}]\label{main Theorem}
Let $Q$ be a quiver without loops and~$d$ be an indivisible dimension vector. For a generic stability parameter~$\theta$ with respect to~$d$ and for a finite field $\FF_q$ of sufficiently large prime characteristic, we have
\begin{gather*} \cA_{Q,d}(q) = \sum_{i=0}^e \dim H^{2e - 2i}(X_0(\CC),\CC)q^i, \end{gather*}
where $e = \frac{1}{2} \dim X_0 = \dim \Mod$. In particular, $\cA_{Q,d}(q)$ is a polynomial in $q$ with coefficients in~$\NN$.
\end{Theorem}

\subsubsection*{A summary of the strategy of the proof}

Let us first outline the main steps involved in the proof.

\textbf{Step 1:} \emph{Deforming the moment map fibre to produce a cohomologically trivial family.}
We will construct a family $\mathfrak{X} \ra \AA^1$ over any field $k$ whose special fibre over $0$ is $X_0$ and whose general fibre is isomorphic to $X := \cM^{\theta\text{-ss}}_d\big(\ov{Q},\cR_\theta\big)= \mu^{-1}(\theta)/\!/_{\chi_\theta} \G$ by taking $\mathfrak{X}:= \mu^{-1}(L)/\!/_{\!\chi_\theta} \GL_d$ for the line $L \subset \LieG \subset \LieGL_d$ joining $0$ and $\theta$. Working over $k = \CC$, we use the hyperk\"{a}hler structure on $\RepQbar$ to show that this family is topologically trivial (and so the singular cohomology of~$X_0$ and~$X$ are isomorphic). From this we will deduce that~$X$ and~$X_0$ have the same point count over a finite field of sufficiently large characteristic (see Step~6).

\textbf{Step 2:} \emph{Purity of the special fibre $X_0$ via the scaling action.}
We show that the natural dilation action on $\RepQbar$ given by scaling the morphisms over each arrow induces a $\GG_m$-action on $X_0$ that is semi-projective; that is, the fixed locus $(X_0)^{\GG_m}$ is projective and the limit of all points in $X_0$ under the action of $t \in \GG_m$ as $t \ra 0$ exists. Consequently, one can construct a~Bia{\l}ynicki-Birula decomposition of $X_0$ which gives rise to a description of the cohomology (and other algebro-geometric invariants) of $X_0$ in terms of its $\GG_m$-fixed locus, which is smooth and projective. In particular, this enables us to deduce that $X_0$ is cohomologically pure in Step~3.

\textbf{Step 3:} \emph{Purity and point counting over finite fields.}
In this step, we explain how the Poincar\'{e} polynomial of $X$ and $X_0$ can be computed by counting points over finite fields. The Weil conjectures and comparison theorems between singular and $\ell$-adic cohomology, enable one to calculate the Betti numbers of a smooth projective complex variety $Y$ with good reduction~$Z$ mod $p$ by counting the $\FF_q$-points of $Z$, where $q$ is a power of $p$. Unfortunately, $X$ and $X_0$ are not projective; however, we explain that the same conclusions still hold for a smooth variety~$Z$ over $\FF_q$ if $Z$ is pure and has polynomial point count (that is, $|Z(\FF_{q^r})|$ is a polynomial in $q^r$). The plan is to apply this to $X_0$, which is smooth and pure by Step~2. In the next two steps, we will show that $X$ has polynomial point count over finite fields of sufficiently large characteristic.

\textbf{Step 4:} \emph{Point counting for the general fibre $X$ and absolutely indecomposable representations.}
The goal of this step is to relate the $\FF_q$-point count $|X(\FF_q)|$, which is the number of isomorphism classes of $\theta$-stable $d$-dimensional $\FF_q$-representations of $\big(\overline{Q},\cR_\theta\big)$, with the number $\cA_{Q,d}(q)$ of absolutely indecomposable $d$-dimensional $\FF_q$-representations of $Q$, where $q$ is a power of a sufficiently large prime $p$. More precisely, we will show that for $\FF_q$ of large characteristic
\begin{gather*} \cA_{Q,d}(q) = q^{-e} |X(\FF_q)|, \end{gather*}
where $e := \frac{1}{2} \dim X$.

For $p$ sufficiently large, we will show that all points in $\mu^{-1}(\theta)$ are $\theta$-stable and the relationship between these two counts follows from work of Crawley-Boevey \cite{CBmoment} studying the lifting of $Q$-representations to $\big(\ov{Q},\cR_\theta\big)$-representation under the restriction of the projection $\rep_d\big(\overline{Q}\big) \ra \rep_d(Q)$ to the level set $\mu^{-1}(\theta)= \rep_d\big(\ov{Q},\cR_\theta\big)$. More precisely, Crawley-Boevey proves that the image on $\FF_q$-points of $\pi\colon \mu^{-1}(\theta) \ra \rep_d(Q)$ is the set of indecomposable $d$-dimensional $\FF_q$-representations of $Q$ and also describes the fibres using self-extension groups of quiver representations.

\textbf{Step 5:} \emph{Kac's theorem on absolutely indecomposable quiver representations.}
In this step, we survey Kac's work on absolutely indecomposable quiver representations over finite fields. The starting point for this work is a beautiful theorem of Gabriel, which describes the indecomposable complex representations of a quiver whose underlying graph is a Dynkin diagram in terms of the positive roots of the Lie algebra associated to this Dynkin diagram. Kac generalised this work to arbitrary quivers by associating to such a quiver $Q$ (or strictly speaking its underlying graph) a root system $\Delta_Q \subset \ZZ^V$ (for a quiver without loops, this is the root system of an associated Kac--Moody Lie algebra $\fg_Q$). More precisely, he shows that absolutely indecomposable quiver representations of dimension $d$ exists over a finite field precisely when $d$ is a positive root of $\Delta_Q$ and proves that the count $\cA_{Q,d}(q)$ is polynomial in $q$ with integer coefficients. One of Kac's conjectures on $\cA_{Q,d}(q)$ was the non-negativity of the coefficients; the proof of this conjecture follows from \cite{CBVdB,HLRV} as we see in the final step.

\textbf{Step 6:} \emph{Specialisation and relating the cohomology of the special fibre and general fibre.}
Finally we relate various cohomology groups associated to $X$ and $X_0$ in order to prove the main result. In order to pass between the GIT quotients over the field of complex numbers and various finite fields, we first state a result concerning GIT over the integers and base change. Since the varieties $\RepQbar$ and $\GL_d$, as well as the moment map $\mu$, are defined over the integers, the family $\mathfrak{X} \ra \AA^1$ is also defined over the integers. The key result we need is that over an open subset of $\spec \ZZ$ the construction of these GIT quotients commutes with base change and the family $\mathfrak{X} \ra \AA^1$ is smooth.

Using the (topological) triviality of the family $\mathfrak{X} \ra \AA^1$ over $\CC$ and the comparison theorem together with Deligne's base change result for direct images, we obtain isomorphisms between the compactly supported $\ell$-adic cohomology of the base changes of $X_0$ and $X$ to $\ov{\FF_p}$ for $p \gg 0$. By the Grothendieck--Lefschetz trace formula, we deduce that for a finite field $\FF_q$ of sufficiently large characteristic $p$, the point counts of $X$ and $X_0$ coincide
\begin{gather*} |X_0(\FF_q) | = | X(\FF_q)|.\end{gather*}
There is a more direct proof of this equality due to Najakima which utilises the Bia{\l}ynicki-Birula decompositions on $\mathfrak{X}$ and appears an appendix in \cite{CBVdB}; however, we have chosen to present the original proof of Crawley-Boevey and Van den Bergh in Step 1, as it utilises the hyperk\"{a}hler structure in a rather ingenious way.

Over a finite field $\FF_q$ of characteristic $p \gg 0$, the $\FF_q$-variety $X_0$ is pure and smooth and has polynomial point count equal to $q^e\cA_{Q,d}(q)$; hence, this polynomial is the $\ell$-adic Poincar\'{e} polynomial of $X_0 \times_{\FF_q} {\overline{\FF_p}}$ for $p \gg0$ and $\ell \neq p$. Since $X_0$ is the mod $q$ reduction of the complex variety $X_{0,\CC}$, we then deduce Theorem \ref{main Theorem} from the comparison theorem and Poincar\'{e} duality.

\subsection{Deforming the moment map to produce a cohomologically trivial family}

As the stability parameter $\theta$ satisfies $\theta \cdot d = 0$, it determines a central element $(\theta I_{d_v})_{v \in V} \in \LieG$, which we also denote by $\theta$. Let $L=k\theta \subset \LieG$ denote the line joining $\theta$ and $0$. Then we consider the fibres of the moment map over points in $L$; let
\begin{gather*} \mathfrak{X}:= \mu^{-1}(L)/\!/_{\chi_\theta} \G, \end{gather*}
which we naturally view as a family over $L \cong \AA^1$. The special fibre of $\mathfrak{X}$ over $ 0 \in \AA^1$ is precisely the variety $X_0$ considered above and the general fibre of $\mathfrak{X}$ over a non-zero point in $\AA^1$ is isomorphic to the variety
\begin{gather*} X:=\cM^{\theta\text{-ss}}_d\big(\ov{Q},\cR_\theta\big) = \mu^{-1}(\theta)/\!/_{\chi_\theta} \G. \end{gather*}
We note that we can construct the family $\mathfrak{X} \ra \AA^1$ over any field $k$ and also over $\spec \ZZ$, as the varieties $\RepQbar$ and $\GL_d$ and the morphism $\mu$ are all defined over the integers.

\begin{prop}[{\cite[Lemma 2.3.3]{CBVdB}}]Over $k = \CC$, the family $\mathfrak{X} \ra \AA^1$ is topologically trivial.
\end{prop}
\begin{proof}We recall that $\RepQbar$ is hyperk\"{a}hler and so it has a 2-sphere of K\"{a}hler structures, as the multiplicative group $\HH^*$ acts (by right multiplication) on $\RepQbar$; this permutes the complex structures and the subgroup $\SUl(2) \cong \{ \beta \in \HH \colon \beta \beta^\dagger = 1 \}$ acts isometrically with respect to the hyperk\"{a}hler metric. Let us write the hyperk\"{a}hler moment map for the action of the maximal compact subgroup $\U_d < \GL_d$ as a map
\begin{gather*} \mu_{\rm HK} \colon \ \RepQbar \ra \operatorname{Im}(\HH) \otimes_{\RR} \fu_d^*, \qquad X \mapsto {\rm i} \otimes \mu_{\rm I}(X) + {\rm j} \otimes \mu_{\rm J}(X) + {\rm k} \otimes \mu_{\rm K}(X),\end{gather*}
where $\mu_I = \mu_{\RR}$ and $\mu_{\rm J} + {\rm i} \mu_{\rm K} = \mu_{\CC} = \mu$. For the $\HH^*$-action on $\operatorname{Im} (\HH)$ given by $\beta \cdot \alpha = \beta \alpha \beta^\dagger$, the hyperk\"{a}hler moment map is $\HH^*$-equivariant: for $\beta \in \HH^*$ and $X \in \RepQbar$, we have
\begin{gather*} \mu_{\rm HK}(X \cdot \beta) = \beta \mu_{\rm HK}(X) \beta^\dagger. \end{gather*}

We will use the transitivity of the $\HH^*$-action on $\operatorname{Im} (\HH)^\circ := \operatorname{Im} (\HH)- \{ 0 \}$ to construct a~tri\-vialisation $X_0 \times \CC \cong \fX$. Since this action is transitive, for fixed $\alpha \in \operatorname{Im} (\HH)^\circ$ the action map $(-)\cdot\alpha\colon \HH^* \ra\operatorname{Im} (\HH)^\circ$ admits a continuous section $s\colon C \ra \HH^*$ over any contractible subset $C \subset \operatorname{Im} (\HH)^\circ$ containing $\alpha$. For any coadjoint fixed point $\theta \in \fu_d^*$, we obtain a local continuous trivialisation of the hyperk\"{a}hler moment map
\begin{gather*} \mu_{\rm HK}^{-1}(\alpha \otimes \theta) \times C \cong \mu^{-1}_{\rm HK}(C \otimes \theta), \qquad (X,c)\mapsto X \cdot s(c), \end{gather*}
which is $\U_d$-equivariant and so gives rise to a continuous isomorphism
\begin{gather*} \mu^{-1}_{\rm HK}(\alpha \otimes \theta) / \U_d \times C \cong \mu^{-1}_{\rm HK}(C \otimes \theta)/\U_d. \end{gather*}

We apply this to $\alpha = {\rm i} \in C=\{ {\rm i} + {\rm j} \CC \} \subset \operatorname{Im} (\HH)^\circ$. Then $\mu^{-1}_{\rm HK}(\alpha \otimes \theta) \cong \mu_{\RR}^{-1}(\theta) \cap \mu_{\CC}^{-1}(0)$ and
\begin{gather*} \mu^{-1}_{\rm HK}(C \otimes \theta) = \mu^{-1}_{\rm HK}(({\rm i} + {\rm j} \CC)\otimes \theta) = \mu_{\RR}^{-1}(\theta) \cap \mu^{-1}_{\CC}(\CC \theta) = \mu^{-1}_{\RR}(\theta) \cap \mu_{\CC}^{-1}(L)\end{gather*}
and so we obtain a continuous trivialisation over $C \cong \CC$
\begin{gather*} \big(\mu_{\RR}^{-1}(\theta) \cap \mu_{\CC}^{-1}(0)\big)/\U_d \times \CC \cong \big(\mu_{\RR}^{-1}(\theta) \cap \mu^{-1}_{\CC}(L)\big)/\U_d. \end{gather*}
By the Kempf--Ness theorem this gives a homeomorphism
\begin{gather*} X_0 \times \CC = \mu_{\CC}^{-1}(0)/\!/_{\chi_\theta} \GL_d \times \CC \cong \mu_{\CC}^{-1}(L) /\!/_{\chi_\theta} \GL_d = \mathfrak{X}, \end{gather*}
which proves that the family $\mathfrak{X} \ra \CC$ is topologically trivial.
\end{proof}

We will apply this result to deduce that over a finite field $\FF_q$ of sufficiently large prime characteristic the $\FF_q$-varieties $X_0$ and $X$ have the same point count; an algebraic proof is also given by Nakajima in~\cite{CBVdB}.

\subsection[Purity of the special fibre $X_0$ via the scaling action]{Purity of the special fibre $\boldsymbol{X_0}$ via the scaling action}\label{sec Step 2}

In this section, we consider the GIT quotient $X_0$ over a field $k$. We recall that $X_0 \!:=\! \mu^{-1}(0)/\!/_{\chi_\theta} \G$ is projective over the affine variety $\Aff(X_0):=\mu^{-1}(0)/\!/ \GL_d$, which is equal to the spectrum of the ring of $\GL_d$-invariants on $\mu^{-1}(0) = \rep_d\big(\ov{Q},\cR_0\big)$. Thus we have a commutative diagram
\begin{gather*}
\xymatrix{\mu^{-1}(0)^{\theta\text{-ss}} \ar[d]^{} \ar@{^{(}->}[r] &\mu^{-1}(0)\ar[d]^{\pi}\\
X_0 \ar[r]^{p} & \Aff(X_0),}
\end{gather*}
where the map $p$ is projective and the map $\pi$ denotes the affine GIT quotient. Since $\theta$ is generic with respect to $d$, the $k$-variety $X_0$ is smooth (as in the proof of Lemma~\ref{lem smooth large char} below).

There is a dilating $\GG_m$-action on $\RepQbar$ given by scalar multiplication on the matrices over all arrows with a unique fixed point corresponding to the origin. Moreover, the limit of every point in $\RepQbar$ under the action of $ t \in \GG_m$ as $t \ra 0$ exists and is equal to the origin. Hence, this is a semi-projective $\GG_m$-action in the sense of the following terminology introduced in~\cite{HRV_semiproj}.

\begin{defn}A $\GG_m$-action on a smooth quasi-projective variety $Z$ is \emph{semi-projective} if $Z^{\GG_m}$ is projective and for all $z \in Z$ the limit $\lim\limits_{t \ra 0} t \cdot z$ exists in $Z$.
\end{defn}

Here by this limit existing, we mean that the map $\GG_m \ra Z$ given by $ t \mapsto t \cdot z$ extends to a~morphism $\AA^1 \ra Z$ (such an extension is unique if it exists, as $Z$ is separated).

\begin{ex}The moduli space of semistable Higgs bundles of coprime rank and degree over a smooth projective algebraic curve has a semi-projective $\GG_m$-action given by scaling the Higgs field~\cite{simpson_HiggsGm}.
\end{ex}

The key feature of semi-projective $\GG_m$-actions is that they give rise to a Bia{\l}ynicki-Birula decomposition \cite{BB} of $Z$, which gives a description of the cohomology (and other invariants, such as the Chow groups and motive) of $Z$ in terms of that of its fixed locus. Since the fixed locus is smooth and projective, we will deduce that $Z$ is (cohomologically) pure in Section~\ref{sec purity point count}.

The scaling $\GG_m$-action on $\RepQbar$ commutes with the $\GL_d$-action and the algebraic moment map is $\GG_m$-equivariant with respect to this action and the $\GG_m$-action on $\fg\fl_d$ of weight $2$. Hence, there is an induced $\GG_m$-action on $\mu^{-1}(0)$ and its GIT quotients $X_0$ and $\Aff(X_0)$ such that the map $p\colon X_0 \ra \Aff(X_0)$ is $\GG_m$-equivariant. We can then prove that this $\GG_m$-action on $X_0$ is semiprojective as in \cite{HRV_semiproj}.

\begin{prop}\label{prop semiproj Gm X0}This scaling action of $\GG_m$ on $X_0$ is semi-projective.
\end{prop}
\begin{proof}This argument is given in \cite{CBVdB} and in \cite{HRV_semiproj}. We first show that this statement holds for the affine variety $\Aff(X_0)$. Let $x_0:=\pi(0) \in \Aff(X_0)$ denote the image of the origin $0 \in \mu^{-1}(0)$ under the affine GIT quotient $\pi$. Then $x_0$ is fixed by the $\GG_m$-action as $\pi\colon \mu^{-1}(0) \ra \Aff(X_0)$ is $\GG_m$-equivariant. In fact, this is the only $\GG_m$-fixed point in $\Aff(X_0)$ and all other points $x \in \Aff(X_0)$ satisfy $\lim\limits_{t \ra 0} t \cdot x = x_0$, as the same statement holds for $\mu^{-1}(0)$ and thus the $\GG_m$-action on $\Aff(X_0) = \spec\cO\big( \mu^{-1}(0)\big)^{\Gl_d}$ induces a~grading on $\cO\big( \mu^{-1}(0)\big)^{\Gl_d}$ which is concentrated in non-positive degrees and with weight zero piece isomorphic to~$k$.

Since $p$ is projective and $\GG_m$-equivariant, the fixed locus $X_0^{\GG_m} = p^{-1}(x_0)$ is projective and the flow under the $\GG_m$-action as $t \ra 0$ exists for all points in $X_0$. Thus the $\GG_m$-action on~$X_0$ is semi-projective.
\end{proof}

Hence, there is an associated Bia{\l}ynicki-Birula decomposition \cite{BB} of $X_0$ and the flow $X_0 \ra p^{-1}(x_0)$ under this $\GG_m$-action defines a homotopy retract. In particular, the cohomology of~$X_0$ can be described in terms of the cohomology of the smooth projective variety $p^{-1}(x_0)$. By Proposition~\ref{prop semiproj Gm implies pure} below, we deduce that~$X_0$ is (cohomologically) pure.

\subsection{Purity and point counting over finite fields}\label{sec purity point count}

By the Weil conjectures and comparison theorems between singular and $\ell$-adic cohomology, the Betti numbers of a smooth projective complex variety $Y$, which is defined over a number field and has good reduction $Z$ modulo a~prime~$p$, can be calculated by counting points of $Z$ over $\FF_q$ where $q$ is a power of~$p$. In this section, we will explain a generalisation of the above statement to smooth pure varieties.

\begin{ex}Let us consider the point count of $\PP^n$. Over $\FF_q$, we have
\begin{gather*} | \PP^n(\FF_q) | =\frac{ q^{n+1} -1 }{q-1} = 1 + q + q^2 + \cdots + q^n \end{gather*}
and the coefficients are precisely the even Betti numbers of $\PP^n$.
\end{ex}

Let us start by recalling the properties of $\ell$-adic cohomology that we will need to define purity. Let $p$ be a prime number and $q$ be a power of $p$ and fix a prime $\ell \neq p$. For a $\FF_q$-variety $Z$, we write $\ov{Z}:= Z \times_{\FF_q} \ov{\FF_q}$ for the base change to the algebraic closure. The (compactly supported) $\ell$-adic cohomology groups of $\ov{Z}$
\begin{gather*} H^i_c\big(\ov{Z},\QQ_{\ell}\big) := \lim_{\la} H^i_{c,\et}\big(\ov{Z},\ZZ/l^r\ZZ\big) \otimes_{\ZZ_l} \QQ_l\end{gather*}
are finite-dimensional $\QQ_{\ell}$-vector spaces that have many of the properties of the usual (compactly supported) singular cohomology groups defined for varieties over $k \subset \CC$. Let $Z$ and $Y$ be $\FF_q$-varieties; then we have the following properties.
\begin{itemize}\itemsep=0pt
\item Functoriality: for proper morphisms $f\colon Z \ra Y$ we have $H^i_c\big(\ov{Y},\QQ_{\ell}\big) \ra H^i_c\big(\ov{Z},\QQ_{\ell}\big)$.
\item K\"{u}nneth isomorphisms: $H^i_c\big(\ov{Y \times Z},\QQ_{\ell}\big) \cong H^i_c\big(\ov{Y},\QQ_{\ell}\big) \otimes H^i_c\big(\ov{Z},\QQ_{\ell}\big)$.
\item Vanishing properties: $H^i_c\big(\ov{Z},\QQ_{\ell}\big) \neq 0$ only for $0 \leq i \leq 2 \dim Z$.
\item For a Zariski-locally trivial $\AA^n$-fibration $Y \ra Z$, we have $H^i_c\big(\ov{Y},\QQ_{\ell}\big) \cong H_c^{i-2n}\big(\ov{Z},\QQ_{\ell}\big) \otimes H^2_c\big(\AA^1, \QQ_l\big)^{\otimes n}$.
\item Gysin long exact sequences for closed subvarieties $Z \subset Y$ with $U:= Y - Z$
\begin{gather*} \cdots \ra H^i_c\big(\ov{U}, \QQ_{\ell}\big) \ra H^i_c\big(\ov{Y}, \QQ_{\ell}\big) \ra H^i_c\big(\ov{Z}, \QQ_{\ell}\big) \ra H^{i+1}_c\big(\ov{U}, \QQ_{\ell}\big) \ra \cdots. \end{gather*}
\item Poincar\'{e} duality for smooth $\FF_q$-varieties.
\end{itemize}

For a indepth treatement of \'{e}tale cohomology and the Weil conjectures, see the book of Milne~\cite{MilneEtale}.

Let $\FF_q$ be a finite field of positive characteristic $p$. For a $\FF_q$-variety $Z$, we let $\Fr_{Z}\colon \ov{Z} \ra \ov{Z}$ denote the relative Frobenius. The fixed points of the relative Frobenius on $\ov{Z}$ are precisely the set of $\FF_q$-points in $Z$ and similarly the fixed points of $\Fr_{Z}^n$ are $Z(\FF_{q^n})$. In fact, the number of such points can be computed using the induced Frobenius action on $H^i_c\big(\ov{Z},\QQ_{\ell}\big)$.

\begin{Theorem}[the Grothendieck--Lefschetz trace formula] Let $Z$ be a smooth variety over a~finite field $\FF_q$ of characteristic $p >0$. Then for $l \neq p$, we have
\begin{gather*} |Z(\FF_{q^n})| = \sum_{i=0}^{2 \dim Z} (-1)^i \tr\big(\Fr_{Z}^n \colon H^i_c\big(\ov{Z},\QQ_l\big)\big). \end{gather*}
\end{Theorem}

The final part of the Weil conjectures was Deligne's proof of the Riemann hypothesis: for a~smooth and projective $\FF_q$-variety $Z$ all eigenvalues of $\Fr_{Z}$ on $H^i_c(\ov{Z},\QQ_l)$ have absolute value~$q^{i/2}$ (for any choice of embedding $\ov{\QQ_l} \hookrightarrow \CC$). This motivates the following definition of purity.

\begin{defn}An $\FF_q$-variety $Z$ is \emph{$($cohomologically$)$ pure} if all eigenvalues of $\Fr_{Z}$ on $H^i_c\big(\ov{Z},\QQ_l\big)$ have absolute value $q^{i/2}$.
\end{defn}

Thus Deligne proved that all smooth projective varieties are pure. We can now give a standard proof of the purity of a smooth quasi-projective variety with a semi-projective $\GG_m$-action using the Bia{\l}ynicki-Birula decomposition~\cite{BB}. In particular, this will provide a proof of the purity of the $\FF_q$-variety $X_0$ mentioned at the end of Section~\ref{sec Step 2}.

\begin{prop}[{\cite[Lemma A.2]{CBVdB}}]\label{prop semiproj Gm implies pure} Let $Z$ be a smooth quasi-projective $\FF_q$-variety with a~semi-projective $\GG_m$-action; then $Z$ is pure.
\end{prop}
\begin{proof}
The assumptions imply that $Z$ has the following Bia{\l}ynicki-Birula decomposition \cite{BB}. Let $Z^{\GG_m} = \cup_{j \in J} Z_j$ denote the decomposition of the fixed locus into connected components; then there is a decomposition
\begin{gather*} Z = \bigsqcup_{j \in J} Z_j^+, \qquad \text{where} \qquad Z_j^+:=\big\{ z \in Z \colon \lim_{t \ra 0} t \cdot z \in Z_j \big\} \end{gather*}
and the limit map $p_j\colon Z_j^+ \ra Z_j$ is a Zariski locally trivial affine space fibration. By assumption, the smooth varieties $Z_j$ are projective, and thus pure; the same also holds for the smooth strata~$Z_j^+$, as
 $p_j\colon Z_j^+ \ra Z_j$ is a Zariski locally trivial affine fibration. Finally, the fact that~$Z$ is quasi-projective means that there is a filtration of $Z$ by closed subsets whose successive differences are the strata $Z_j^+$, and one can show that the Gysin sequences associated to this filtration of $Z$ split into short exact sequence using the purity of the strata $Z_j^+$. Consequently, we deduce that $Z$ is also pure.
\end{proof}

For certain pure smooth $\FF_q$-varieties, their $\ell$-adic Betti numbers can be described using the following result.

\begin{Lemma}[{\cite[Lemma A.1]{CBVdB}}]\label{lemma purity point count} Let $Z$ be a smooth variety defined over $\FF_q$ which is pure and has polynomial point count over $\FF_{q^r}$; that is $|Z(\FF_{q^r})| = P(q^r)$ for a polynomial $P(t) \in \ZZ[t]$. Then
\begin{gather*} P(q) = \sum_{i \geq 0} \dim H^{2i}_c\big(\ov{Z},\QQ_{\ell}\big)q^{i} \end{gather*}
and in particular $P(t) \in \NN[t]$.
\end{Lemma}

Let us finally note that via the comparison theorem between \'{e}tale and singular cohomology, one can relate this result concerning $\ell$-adic cohomology with the usual singular cohomology. We will return to this statement in the final step.

In fact, we will want to compare the Betti cohomology of a complex variety with the point count of a reduction of this variety to a finite field using a theorem of Katz, which appears as an appendix in \cite{HRV}. For a complex variety $Z_\CC$, we can choose a spreading out $Z_R$ of $Z$ over a~finitely generated $\ZZ$-algebra $R$ (i.e., $Z_\CC = Z_R \times_R \CC$) and let $Z$ be a reduction of $Z_R$ to some finite field $\FF_q$. If $Z$ has polynomial point count $P_Z(t)\in \ZZ[t]$; then the $E$-polynomial of $Z_{\CC}$ (whose coefficients are the virtual Hodge numbers) is given by $E_{Z_\CC} (x, y) = P_Z (xy)$. If additionally the compactly supported cohomology of $Z$ is pure, then $P_Z(q) =E_Z \big(q^{1/2},q^{1/2}\big) = P_c\big(Z_{\CC},q^{1/2}\big)$ (where $P_c$ denotes the compactly supported Poincar\'{e} polynomial).

\subsection[Point count for the general fibre and absolutely indecomposable representations]{Point count for the general fibre and absolutely indecomposable\\ representations}

In this section, we let $p$ be a prime number and $q$ be a power of $p$. The goal is to compute $|X(\FF_q)|$. The first result we need, is that for all sufficiently large primes, all points in $\mu^{-1}(\theta)$ are $\theta$-stable by the following lemma.

\begin{Lemma}\label{lemma stab on theta level set}Let $\theta$ be a generic stability parameter with respect to $d$. Then for a field $k=\FF_q$ of sufficiently large prime characteristic, we have $\mu^{-1}(\theta)^{\theta\text{\rm -ss}} = \mu^{-1}(\theta)^{\theta\text{\rm -s}} = \mu^{-1}(\theta)$.
\end{Lemma}
\begin{proof}This follows by Lemma~\ref{lemma stab trivial} as $\mu^{-1}(\theta) = \rep_d\big(\overline{Q},\cR_\theta\big)$.
\end{proof}

In order to count points of $\mu^{-1}(\theta)$ over $\FF_q$, we will relate such representations of $\big(\ov{Q},\cR_\theta\big)$ with absolutely indecomposable representations of $Q$ using a theorem of Crawley-Boevey \cite[Theorem 3.3]{CBmoment} concerning the liftings of indecomposable representations of $Q$ to $\big(\ov{Q},\cR_\theta\big)$. We recall that there is a natural projection $\rep_d\big(\overline{Q}\big) \ra \rep_d(Q)$, whose restriction to the fibre of the moment map over $\theta$ we denote by
\begin{gather*} \pi \colon \ \mu^{-1}(\theta) \ra \rep_d(Q). \end{gather*}

\begin{Theorem}[Crawley-Boevey \cite{CBmoment}]\label{CB_geom_mmap} For $\theta$ generic with respect to $d$, the image of $\pi\colon \mu^{-1}(\theta) $ $\ra \rep_d(Q)$ on $\FF_q$-points is the set of indecomposable representations. Moreover, the fibre of $\pi$ over an indecomposable $d$-dimensional $\FF_q$-representation $W$ of $Q$ is identified with the dual of the space of self-extensions of~$W$
\begin{gather*} \pi^{-1}(W)\cong \Ext^1_Q(W,W)^*. \end{gather*}
\end{Theorem}
\begin{proof}For $W \in \rep_d(Q)(\FF_q)$, we consider the dual of the exact sequence in Exercise \ref{exer hom and exts}
\begin{gather*} 0 \ra \Ext^1_Q(W,W)^* \ra \rep_d(Q^{\op}) \ra \fg\fl_d^* \ra \End(W)^* \ra 0. \end{gather*}
Then $W$ lifts to a representation of $\big(\ov{Q},\cR_\theta\big)$ if and only if $\theta$ is in the image of $\rep_d(Q^{\op}) \ra \fg\fl_d^*$; that is, $\sum\limits_{v \in V} \theta_v \tr(f_v)=0$ for any $f \in \End(W)$. If $W = W_1 \oplus W_2$ and $f \in \End(W)$ is the projection onto $W_1$, then it follows that $\theta \cdot \dim(W_1) = 0$. Since $\theta$ is assumed to be generic with respect to $d$, we see that only indecomposable representations of~$Q$ can lift to $\big(\ov{Q},\cR_\theta\big)$. To prove that an indecomposable representation lifts, one uses the fact that~$\End(W)$ is local for~$W$ indecomposable (see Lemma~\ref{lemma end indecomp}).
\end{proof}

A final technical tool required to relate the point count of $X$ with absolutely indecomposable representations of $Q$ over finite fields of large characteristic is Burnside's formula for the number of orbits under a finite group action.

\begin{Lemma}[Burnside's formula]\label{burnside}
Let $G$ be a finite group acting on a finite set $Y$, then
\begin{gather*} | Y/G | := \frac{1}{|G|} \sum_{g \in G} | Y^g | = \frac{1}{|G|} \sum_{y \in Y} |\Stab_G(y)|.\end{gather*}
\end{Lemma}

Now we can state and prove the main result of this section.

\begin{prop}[Crawley-Boevey and Van den Bergh]\label{prop point count X}
Let $d$ be indivisible and $\theta$ be generic with respect to $d$. Then for a prime $p \gg 0$ and $q$ a power of $p$, we have
\begin{gather*} \cA_{Q,d}(q) = q^{-e} |X(\FF_q)|, \end{gather*}
where $e := \frac{1}{2} \dim X$.
\end{prop}
\begin{proof}For primes $p \gg 0$ and $q=p^r$, we have that all points in the $\FF_q$-variety $\mu^{-1}(\theta)$ are $\theta$-stable by Lemma \ref{lemma stab on theta level set}. Hence $\mu^{-1}(\theta) \ra X=\mu^{-1}(\theta)/\!/_{\chi_\theta} \G$ is a principal $\G$-bundle. Furthermore, as the Brauer group of $\FF_q$ is trivial, the rational points of $X$ are isomorphism classes of $\FF_q$-representations of $\big(\overline{Q},\cR_\theta\big)$ so
\begin{gather*} X(\FF_q) \cong \mu^{-1}(\theta)(\FF_q)/\G(\FF_q)\end{gather*}
and as $\G$-acts freely on $\mu^{-1}(\theta)$, we have
\begin{gather}\label{ptcount_X}
|X(\FF_q)| = \frac{| \mu^{-1}(\theta)(\FF_q)|}{|\G(\FF_q)|}.
\end{gather}

We now relate this point count to $\cA_{Q,d}(q)$ using Theorem \ref{CB_geom_mmap}. Since $d$ is indivisible, $\FF_q$-representations of dimension $d$ are indecomposable if and only if they are absolutely indecomposable. We let $\rep_d(Q)^{\ai}$ denote the constructible subset of absolutely indecomposable $d$-dimensional representations of $Q$. By definition of $\cA_{Q,d}(q)$, we have
\begin{gather*} \cA_{Q,d}(q):= \big| \rep_d(Q)^{\ai}(\FF_q) / \G(\FF_q) \big| \end{gather*}
and by Burnside's formula (Lemma \ref{burnside}) this equals
\begin{gather*} \cA_{Q,d}(q)= \frac{1}{| \G(\FF_q) |} \sum_{W \in \rep_d(Q)^{\ai}(\FF_q)} q^{-1} |\End_Q(W)|, \end{gather*}
where we use that $\Stab_{\G}(W) \cong \Aut_Q(W)/\GG_m $ and so $| \Stab_{\G}(W)| = q^{-1} |\End_Q(W)|$ by Lemma \ref{lemma end indecomp} below. Then by Theorem \ref{CB_geom_mmap}, we obtain
\begin{gather}\label{ptcount_A}
\cA_{Q,d}(q)= \frac{1}{| \G(\FF_q) |} \sum_{W \in \mu^{-1}(\theta)(\FF_q)} q^{-1} \frac{|\End(\pi(W))|}{|\Ext^1(\pi(W),\pi(W))|}.
\end{gather}
Since $\dim \pi(W) = d$, we have that
\begin{gather*} \langle d,d \rangle_Q = \dim \End_Q(\pi(W)) - \dim \Ext^1_Q(\pi(W),\pi(W))\end{gather*}
by Exercise \ref{exer hom and exts}. Therefore, combining \eqref{ptcount_X} and \eqref{ptcount_A} we obtain
\begin{gather*} \cA_{Q,d}(q) = q^{ \langle d,d \rangle_Q -1} |X(\FF_q)|.\end{gather*}
Finally, we recall from \eqref{eq euler dim} that $ \langle d,d \rangle_Q = \dim \GL_d - \dim \rep_d(Q)$ and so $\dim \cM^{\theta\text{-ss}}_d(Q) = 1 - \langle d,d \rangle_Q$, as $\GG_m \cong \Delta \subset \GL_d$ acts trivially. Since $X$ is an algebraic symplectic reduction of the action on $\RepQbar$ at a regular value, we have $\dim X = 2 \dim \cM^{\theta\text{-ss}}_d(Q)$. Thus $1 - \langle d,d \rangle_Q = \frac{1}{2} \dim X$, which completes the proof.
\end{proof}

It remains for us to describe the endomorphism ring of an absolutely indecomposable representation.

\begin{Lemma}\label{lemma end indecomp}Let $W$ be an indecomposable $\FF_q$-representation of $Q$. Then the following statements hold.
\begin{enumerate}\itemsep=0pt
\item[$i)$] Every endomorphism of $W$ is either nilpotent or invertible.
\item[$ii)$] $\End_Q(W)$ is local with nilpotent radical $\End_Q^{\nil}(W)$.
\item[$iii)$] $k_W:= \End_Q(W) /\End_Q^{\nil}(W)$ is a finite field containing $\FF_q$.
\end{enumerate}
If $W$ is absolutely indecomposable, then $k_W = \FF_q$ and
\begin{gather*} \frac{|\End_Q(W)|}{|\Aut_Q(W)|} = \frac{q}{q-1}. \end{gather*}
\end{Lemma}
\begin{proof}By the fitting lemma, for an endomorphism $f$ of $W$ we have $W = \ker (f^r) \oplus \im(f^r)$ for some $r$ as $W$ has finite length, and thus either $f$ is nilpotent or invertible. As a corollary, any finite-dimensional algebra which has only $0$ and $1$ as idempotents, is a local ring with nilpotent radical. This proves the first two statements. For any local ring with nilpotent radical, the quotient by this ideal is a division algebra. Hence $k_W:=\End_Q(W)/\End_Q^{\nil}(W)$ is a finite division algebra and by Wedderburn's theorem, we deduce that $k_W$ is a finite field $k_W$ containing $\FF_q$. This proves the first three statements.

Let $n =[k_W : \FF_q]$ and $W' = W \otimes_{\FF_q} k_W$; then as $\FF_q$ is perfect, we have
\begin{gather*} k_{W'} = \frac{\End_Q(W')}{\End_Q^{\nil}(W')} \cong \frac{\End_Q(W)}{\End_Q^{\nil}(W)} \otimes_{ \FF_q} k_W = k_W \otimes_{\FF_q} k_W = k_W^{\oplus n}. \end{gather*}
Hence $W'$ is a direct sum of $n$ pairwise non-isomorphic indecomposable $k_W$-representations. In particular, if $W$ is absolutely indecomposable, then $W'$ is indecomposable and thus $ k_W = \FF_q$.

For an absolutely indecomposable representation $W$, let $p\colon \End_Q(W) \!\ra\! \End_Q(W)\!/\!\End_Q^{\nil}(W)$ $\cong \FF_q$ denote the projection; then as $\End_Q^{\nil}(W) = p^{-1}(0)$ and $\Aut_Q(W) = p^{-1}(\FF_q^\times)$, we have
\begin{gather*} \frac{|\End_Q^{\nil}(W)|}{|\Aut_Q(W)|} = \frac{1}{q-1}. \end{gather*}
The final formula then follows, as $|\End_Q(W)| = |\End_Q^{\nil}(W)| + |\Aut_Q(W)|$.
\end{proof}

\subsection{Kac's theorem on absolutely indecomposable quiver representations}\label{sec Kac results}

The starting point for Kac's work \cite{Kac1,Kac2} is a remarkable discovery of Gabriel, which relates the indecomposable representations of quivers of finite representation type\footnote{A quiver $Q$ is of finite representation type if there are only finitely many isomorphism classes of indecomposable representations of $Q$.} and the positive roots of semisimple Lie algebras. Before stating this theorem, we recall that for a quiver without oriented cycles, the simple objects in $\crep(Q,k)$ are in bijection with the set of vertices $V$. Hence, the dimension vector induces an isomorphism
\begin{gather*}\dim \colon \ K_0(\crep(Q,k)) \ra \ZZ^V\end{gather*}
from the Grothendieck group of this category to the free abelian group generated by $V$.

\begin{Theorem}[Gabriel]Let $Q$ be a connected quiver without oriented cycles.
\begin{enumerate}\itemsep=0pt
\item[$1)$] $Q$ is of finite type if and only if the underlying graph of $Q$ is simply-laced Dynkin diagram.
\item[$2)$] In this case, if $\mathfrak{g}_Q$ denotes the corresponding semisimple complex Lie algebra for this Dynkin diagram, then $\dim\colon K_0(\crep(Q,k)) \ra \ZZ^V$ induces a bijection between the set of isomorphism classes of indecomposable representations of $Q$ and the set of positive roots of~$\mathfrak{g}$.
\end{enumerate}
\end{Theorem}

A nice exposition of this result is given in~\cite{brion_notes}. Bernstein, Gelfand and Ponomarev \cite{BGelP} provided a proof of this result which enhances this remarkable link between quiver representations and Lie algebras, by using reflection functors associated to the vertices of $Q$ to construct all indecomposable representations of a quiver $Q$ of finite representation type from simpler ones analogously to the way all positive roots in the corresponding Lie algebra arise from the simple roots by reflections given by elements of the Weyl group.

Kac \cite{Kac1} associates to a quiver $Q$ with $n$ vertices a root system $\Delta_Q \subset \ZZ^n$ and Weyl group $W_Q$ that only depend on the underlying graph of $Q$ as follows; in \cite{Kac_loops} the necessary modifications for quivers with loops is given. For a quiver $Q$, we recall that the Euler form $\langle -,-\rangle_Q$ on the lattice $\ZZ^n$ defines a matrix $B_Q =(b_{ij})$ where
\begin{gather*} b_{ij}:= \begin{cases} 1 - |a\colon i \ra i| & \mathrm{if } \ i = j, \\ - |a\colon i \ra j| & \mathrm{if }\ i \neq j. \end{cases} \end{gather*}
The symmetrised form $(-,-)_Q$ has associated symmetric matrix $A_Q = B + B^t_Q$. If $Q$ is a~quiver without loops, then $A_Q$ is a symmetric generalised Cartan matrix. Let $\{\alpha_1, \dots , \alpha_n\}$ denote the standard basis of $\ZZ^n$ and we define the set of fundamental roots $\Pi_Q=\{ \alpha_i\colon a_{ii} = 2 \}$ to be the basis vectors corresponding to vertices without loops. Then each fundamental root $\alpha_i \in \Pi_Q$ determines a reflection $r_i \in \Aut(\ZZ^n)$ defined by $r_i(\alpha_j) = \alpha_j - a_{ij} \alpha_i$ and we define the Weyl group $W_Q$ to be the subgroup generated by the fundamental reflections. There is an associated root system $\Delta_Q = \Delta_Q^+ \cup -\Delta_Q^+$ where $\Delta_Q^+$ is a set of positive roots, which decompose into real and imaginary roots. The real roots are the images of the fundamental roots under the Weyl group; these are the only roots if $Q$ is of finite representation type. For the construction of the imaginary roots, see \cite[Section~1.1]{Kac_loops}. If $Q$ is a quiver without loops, then there is a~(typically infinite-dimensional) Lie algebra $\fg_Q$ called the \emph{Kac--Moody Lie algebra} associated to the symmetric generalised Cartan matrix~$A_Q$.

\begin{Theorem}[Kac \cite{Kac1,Kac2}]\label{Theorem kac}\
\begin{enumerate}\itemsep=0pt
\item[$1)$] The number of $\cA_{Q,d}(q)$ of absolutely indecomposable quiver representations over $\FF_q$ does not depend on the orientation of $Q$ and satisfies $\cA_{Q,w(d)}(q)=\cA_{Q,d}(q)$ for $w \in W_Q$.
\item[$2)$] The map $\dim\colon K_0(\crep(Q,k)) \ra \ZZ^V$ induces a surjective map from the set of isomorphism classes of indecomposable representations over $k = \ov{k}$ of $Q$ onto the set of positive roots of~$\Delta^+_Q$.
\item[$3)$] $\cA_{Q,d}(q)$ is a polynomial in $q$ with integral coefficients.
\end{enumerate}
\end{Theorem}

For the polynomial behaviour of $\cA_{Q,d}(q)$ it suffices to prove that the number $\cI_{Q,d}(q)$ of isomorphism classes of indecomposable $d$-dimensional $\FF_q$-representations of $Q$ is given by a~polynomial in~$q$ by using standard reductions involving Galois descent. By the Krull--Schmidt theorem and induction on $d$, it then suffices to show the count $\cM_{Q,d}(q)$ of isomorphism classes of $d$-dimensional $\FF_q$-representations of $Q$ is polynomial in~$q$. Kac computes $\cM_{Q,d}(q)$ using Burnside's theorem, where one must sum over all conjugacy classes of $\GL_d$ for all $d$ by enumerating all possible Jordan normal forms using polynomials (giving the splitting field) and partitions (giving the sizes of the Jordan blocks). He then deduces the polynomial behaviour of~$\cM_{Q,d}(q)$ (and thus~$\cA_{Q,d}(q)$). Kac proves the independence of the orientation of $Q$ using reflection functors and the fact that indecomposable representations correspond to orbits in~$\Rep$ with unipotent stabiliser group.

Hua \cite{Hua} provided more explicit formulae for the polynomials $\cA_{Q,d}(q)$ by considering gene\-ra\-ting functions for these counts, where one sums over all dimension vectors by introducing formal variables $\{X_v ; v \in V \}$. For each $d = (d_v)_{v \in V}$, we write $X^d = \prod\limits_{v \in V} X_v^{d_v}$; then the Krull--Schmidt theorem for $\crep(Q,k)$ gives a formal identity
\begin{gather}\label{gen series}
 \sum_{d \in \NN^V}\cM_{Q,d}(q)X^d = \prod_{ d \in \NN^V \ \{ 0 \}} \big(1 - X^d\big)^{-\cI_{Q,d}(q)}.
\end{gather}

For very simple quivers and low dimension vectors, it is possible to directly calculate $\cA_{Q,d}(q)$.

\begin{exer}
For each of the following quivers and dimension vectors, calculate the Kac polynomial $\cA_{Q,d}(q)$:
\begin{enumerate}\itemsep=0pt
\item[a)] The Jordan quiver with dimension vector $n \in \NN$.
\item[b)] The 2-arrow Kronecker quiver with dimension vector $d = (1,1)$.
\end{enumerate}
\end{exer}

\subsection[Specialisation and relating the cohomology of the special fibre and general fibre]{Specialisation and relating the cohomology of the special fibre\\ and general fibre}

In this final step, we will relate various cohomology groups associated to~$X$ and~$X_0$ in order to prove the main result. In order to pass between the field of complex numbers and various finite fields, we will need to first state some results concerning GIT over the integers and base change. Indeed the affine space $\rep_d\big(\overline{Q}\big)$, the group $\GL_d$, and the moment map $\mu$ are all defined over the integers, and so we can instead consider the above GIT quotients over $\spec \ZZ$ using Seshadri's GIT over a (Nagata) base ring. Although extensions of base fields commute with taking invariants (and thus taking the semistable set and the formation of the GIT quotient commute with base field extensions), the same is not true over rings.

Let us consider the following set up: let $R$ be a finitely generated $\ZZ$-algebra with a maximal ideal $\mathfrak{p} \subset R$ such that $R/\mathfrak{p} \cong \FF_q$ and fix an embedding $R \hookrightarrow \CC$. Then for a variety $\cZ$ over $S:=\spec R$, we can construct by base change:
\begin{enumerate}\itemsep=0pt
\item[1)] an $\FF_q$-variety $Z:= \cZ \times_{\spec R} \spec \FF_q$ (the reduction of $\cZ$ mod $q$) and
\item[2)] a complex variety $Z_{\CC}:=\cZ \times_{\spec R} \spec \CC$.
\end{enumerate}
Now suppose that $\cG$ is a reductive group scheme over $S$ acting on $\cZ$ with respect to an ample linearisation, then we want to know whether formation of the GIT quotient commutes with these various base changes. In fact, for our purposes, it suffices to understand this for $R = \ZZ_N=\ZZ \big[\frac{1}{N}\big]$ where $N \in \ZZ$ such that $p \nmid N$ and so we can apply the following result.

\begin{Lemma}[{\cite[Appendix B]{CBVdB}}]For $S:= \spec \ZZ_N$, we let $\cG$ be a reductive group scheme over~$S$ acting on a quasi-projective $S$-scheme $\cZ$ with respect to an ample linearisation. Then there is a non-empty open subscheme $U \subset S$ over which the formation of the GIT semistable set and GIT quotient commutes with base change; that is, for all points $s\colon \spec k \ra U$, we have
\begin{gather*} \cZ^{\text{\rm ss}} \times_{S} k = (\cZ \times_{S} k)^{\text{\rm ss}} \qquad \text{and} \qquad (\cZ/\!/ \cG) \times_{S} k \cong (\cZ \times_S k)/\!/(\cG \times_S k). \end{gather*}
\end{Lemma}

We note that an open subset $U \subset S:=\spec \ZZ_N$ has the form $U = \spec \ZZ_{M}$ for some $N | M$, and this means we just need to replace $N$ by a sufficiently large multiple. If moreover $p \nmid N$, then we can base change from $S= \spec \ZZ_N$ to $\FF_p$ when applying the above result. More precisely, for a variety $\cZ$ over $S$ we have the following base changes
\begin{gather*}
\xymatrix{Z_{\overline{\FF_p}} \ar[d] \ar[r] & Z_{\FF_{p}} \ar[r] \ar[d] &\cZ \ar[d]& Z_{\CC} \ar[l] \ar[d]\\
\spec \overline{\FF_p} \ar[r] & \spec \FF_{p^r} \ar[r] & S & \spec \CC \ar[l]}
\end{gather*}
and provided $N$ is sufficiently large these base changes all commute with the formation of GIT quotients and semistable sets by the above lemma.

We will also need the following preliminary result concerning the smoothness of $X_0$ and $X$ over finite fields of large characteristic.

\begin{Lemma}\label{lem smooth large char}Let $\theta$ be generic with respect to $d$, then there is a non-empty open subset $U \subset \spec \ZZ$ over which the morphism $f\colon \mathfrak{X} \ra \AA^1$ is smooth.
\end{Lemma}
\begin{proof}
It suffices to prove that $f$ is smooth after base changing to $k=\overline{\QQ}$, as by a spreading out argument it is sufficient to prove that $f$ is smooth over $\QQ$ (as then the same statement holds over $\ZZ$ after inverting finitely many primes) and smoothness can be checked after any field extension. Since $\theta$ is generic, the notions of $\theta$-stability and $\theta$-semistability for $d$-dimensional $k$-representations coincide. As any $\theta$-stable $k$-representation of $\overline{Q}$ is simple, we see that $\G$ acts freely on $\rep_d\big(\overline{Q}\big)^{\theta\text{\rm -s}}$. Hence the infinitesimal action at these points is trivial and so it follows that the restriction of the moment map $\mu$ to $\rep_d\big(\overline{Q}\big)^{\theta\text{\rm -s}}$ is smooth, as $\mu$ lifts the infinitesimal action. Consequently, for the line $L = k\theta$ in the Lie algebra of $\G$, we see that the induced morphism $\mu^{-1}(L)^{\theta\text{\rm -s}} \ra L$ is smooth, and as $\G$ acts freely on $\mu^{-1}(L)^{\theta\text{\rm -s}}$, the $\G$-quotient $\mu^{-1}(L)^{\theta\text{\rm -s}} \ra \mathfrak{X}$ is also smooth. Hence, we deduce that $f\colon \mathfrak{X} \ra L\cong \AA^1$ is also smooth over~$k$.
\end{proof}

We are now in a position to complete the proof. The first goal is to relate the point count of~$X_0$ and~$X$ in large characteristic.

\begin{prop}\label{prop point counts agree}For a finite field $\FF_q$ of sufficiently large characteristic $p$, we have
\begin{gather*} |X_0(\FF_q) | = |X(\FF_q)|. \end{gather*}
\end{prop}
\begin{proof}Using the (topological) triviality of the family $\mathfrak{X} \ra \AA^1$ over $\CC$ and the comparison theorem together with Deligne's base change result for direct images, we deduce that for $p \gg 0$ and $\ell \neq p$, there are isomorphisms
\begin{gather*} H_c^i\big(X \times_{\FF_q} {\overline{\FF_p}},\QQ_{\ell}\big) \cong H_c^i\big(X_0 \times_{\FF_q} {\overline{\FF_p}},\QQ_{\ell}\big)\end{gather*}
in $\ell$-adic cohomology that are compatible with the Frobenius endomorphisms. By applying the Grothendieck--Lefschetz trace formula to both $X$ and $X_0$, which are both smooth $\FF_q$-varieties in large characteristic by Lemma~\ref{lem smooth large char}, we deduce the claim.
\end{proof}

\begin{rmk}A more direct proof of this result is given by Nakajima as an appendix in~\cite{CBVdB}, which involves comparing the Bia{\l}ynicki-Birula decomposition on the total space of the family~$\mathfrak{X}$ with the decompositions on the fibres of this family.
\end{rmk}

Putting all of the above together, we obtain the proof of Crawley-Boevey and Van den Bergh.

\begin{proof}[Proof of Theorem \ref{main Theorem}]Let $\FF_q$ be a finite field of sufficiently large characteristic $p$ so that the construction of the GIT quotient $\mathfrak{X}$ commutes with base change and the family $\mathfrak{X} \ra \AA^1$ is smooth. By Proposition~\ref{prop point count X} and Theorem~\ref{Theorem kac}, we see that $X$ has polynomial point count given by
\begin{gather*} |X(\FF_q)|= q^{-d}\cA_{Q,d}(q). \end{gather*}
Provided $p \gg 0$, this point count coincides with that of $X_0$ by Proposition~\ref{prop point counts agree}. Since $X_0$ is pure by Propositions~\ref{prop semiproj Gm X0} and~\ref{prop semiproj Gm implies pure}, we deduce that the $\ell$-adic Poincar\'{e} polynomial of the $\FF_q$-variety~$X_0$ for $q = p^r$ and $p$ sufficiently large is given by
\begin{gather}\label{eq ladic PP X0}
\cA_{Q,d}(t) = t^{-e} \sum_{i\geq 0} \dim H^{2i}_c\big(X_{0} \times_{\FF_q} \overline{\FF_p},\QQ_{\ell}\big)t^i,
\end{gather}
where $e = \frac{1}{2} \dim X_0$ and $\ell \neq p$ is prime.

Now consider the family $\mathfrak{X} \ra \AA^1$ over $\spec \ZZ_N$ for sufficiently large $N$ indivisible by $p$, then by base change we can obtain the $\ov{\FF_p}$-variety $X_{0} \times_{\FF_q} \overline{\FF_p}$ and the complex variety $X_{0,\CC}$ and these base changes commute with the formation of the GIT quotient. In particular, the complex variety $X_{0,\CC}$ is defined over $\QQ$ and the $\FF_p$-variety $X_{0}$ is a mod $p$ reduction of this complex variety. By smooth base changes results and the comparison theorem [SGA4:3, Expos\'{e}~XVI, Theorem~4.1], we obtain from~\eqref{eq ladic PP X0} the corresponding equality for the Poincar\'{e} polynomial of the sheaf cohomology of $X_{0,\CC}$ with values in the constant sheaf $\CC$
\begin{gather*} \cA_{Q,d}(q) = q^{-e} \sum_{i\geq 0} \dim H^{2i}_c(X_{0,\CC},\CC)q^i.\end{gather*}
By Poincar\'{e} duality for the smooth variety $X_{0,\CC}$ of dimension $2e$, we deduce
\begin{gather*} \cA_{Q,d}(q) = \sum_{i=0}^e \dim H^{2e - 2i}(X_{0,\CC}(\CC),\CC)q^i, \end{gather*}
where now we switch from sheaf cohomology to the singular cohomology of the analytic varie\-ty~$X_{0,\CC}(\CC)$.
\end{proof}

\subsection{A brief survey of Schiffmann's results for bundles}

Let $X$ be a smooth projective curve over a finite field $\FF_q$. Then the category of coherent sheaves over $X$ is an abelian category of homological dimension 1 with a group homomorphism
\begin{gather*} \cl\colon \ K_0(\cC{\rm oh}(X)) \ra \ZZ^2, \qquad [\cF] \mapsto \cl(\cF) = (\rk(\cF),\deg(\cF)) \end{gather*}
through which the Euler form factors
\begin{gather*} \langle \cE, \cF \rangle = \rk \cE \rk \cF (1-g) + \rk \cE \deg \cF - \rk \cF \deg \cE. \end{gather*}
Moreover, this is a Krull--Schmidt category and so there is naturally a notion of (absolutely) indecomposable objects.

For coprime rank $n$ and degree $d$, Schiffmann discovered an analogous relationship between the count $\cA_{n,d}(X)$ of isomorphism classes of indecomposable vector bundles on $X/\FF_q$ and the Betti cohomology of the moduli space of semistable Higgs bundles on $X$. In fact, the case of vector bundles involves several technical issues which did not arise in the quiver setting:
\begin{enumerate}\itemsep=0pt
\item[1)] For quivers, the count of absolutely indecomposable representations of $Q$ was polynomial in the size of the finite field and was independent of the orientation. Moreover, there was a representation theoretic interpretation: the underlying graph of the quiver determined a~root system and Weyl group, under which the polynomial counting absolutely indecomposable representations was invariant. For bundles on curves, one would like to understand the behaviour of this count as $X$ varies in the moduli space of genus $g$ curves over $\FF_q$ and give a representation theoretic interpretation of the associated polynomial.
\item[2)] There are many more vector bundles than quiver representations: while the stack of quiver representations of fixed dimension vector is a finite type stack, the stack of vector bundles of fixed class is only locally of finite type. Hence, to prove the polynomial behaviour of the count of indecomposable bundles, one cannot uses the Krull--Schmidt theorem and count all vector bundles, as this is infinite.
\item[3)] Although the moduli space of Higgs bundles admits a gauge theoretic construction as a~holomorphic symplectic reduction, we need an algebraic version for working over finite fields. Furthermore, to relate Higgs bundles to indecomposable vector bundles, it is ne\-ces\-sary to employ a similar trick to the above trick of Crawley-Boevey and Van den Bergh: one needs to find a suitable family of algebraic symplectic reductions over the affine line that contains the moduli space of Higgs bundles as the special fibre.
\end{enumerate}

The description of the behaviour of this count as $X$ varies in the moduli space of genus $g$ curves is achieved by Schiffmann in \cite{schiffmann}: he proves that there is a polynomial (depending on $(n,d)$ and the genus $g$ of the curve) in the Weil numbers of a curve over a finite field which gives the counts $\cA_{n,d}(X)$ for any curve $X$ over a finite field by evaluation at the Weil numbers of $X$; moreover, by work of Mellit \cite{mellit}, this polynomial is actually independent of the degree $d$. To state this precisely, we recall that the Weil numbers of a genus $g$ smooth projective curve $X/\FF_q$ are the eigenvalues $\sigma_1, \dots , \sigma_{2g}$ of the Frobenius acting on $H^1_{\et}\big(X \times_{\FF_q} \ov{\FF_q}, \ov{\QQ_l}\big)$. If we fix an embedding $\ov{\QQ_l} \hookrightarrow \CC$, then we can view the Weil numbers of $X$ as a tuple of complex numbers of absolute value $q^{\frac{1}{2}}$ and order them as complex conjugate pairs $(\sigma_{2i-1},\sigma_{2i})$ which satisfy $\sigma_{2i-1}\sigma_{2i} = q$ for all $1 \leq i \leq g$. This tuple gives rise to a point in the torus
\begin{gather*} T_g :=\big\{(\alpha_1, \dots , \alpha_{2g}) \in \GG_m^{2n} \colon \sigma_{2i-1}\sigma_{2i} = \sigma_{2j-1}\sigma_{2j} \ \forall\, 1\leq i,j \leq g \big\}\end{gather*}
and the natural action of $W_g :=S_g \ltimes (S_2)^g$ takes care of the choices in the above ordering into complex conjugate pairs. Let $\pi \colon T_g(\CC) \ra T_g(\CC)/W_g$ denote the quotient map and define $\sigma_X:= \pi(\sigma_1, \dots , \sigma_{2g})$ to be the image of the Weil numbers of~$X$.

Let $R_g:= \QQ[z_1, \dots, z_{2g}\colon z_{2i-1}z_{2i} = z_{2j-1}z_{2j} \ \forall\, 1 \leq i \leq g]^{W_g}$. Then we can evaluate any element in $R_g$ at $\sigma_X$ for any genus $g$ smooth projective curve $X$ over a finite field. In genus $0$, we set $R_0 = \QQ[q^{\pm}]$.

\begin{Theorem}[Schiffmann \cite{schiffmann}]
For a fixed genus $g$ and pair $(n,d) \in \NN \times \ZZ$, there is a unique element $\cA_{g,n,d} \in R_g$ such that for any smooth projective geometrically connected curve $X$ of genus $g$ over a finite field, we have
\begin{gather*} \cA_{g,n,d} (\sigma_X) = \cA_{n,d}(X).\end{gather*}
\end{Theorem}

Mellit \cite{mellit} showed this polynomial $\cA_{g,n,d}$ is actually independent of the degree $d$ and so we can write simply $\cA_{g,n}$; this proof is combinatorial and does not give a geometric reason for this independence. Moreover, these polynomials exists in rank $n = 0$ and count absolutely indecomposable torsion sheaves.

In fact, Schiffmann also provides a representation theoretic interpretation of $W_g$, $T_g$ and~$R_g$: the Frobenius $\Fr_X$ on $H^1_{\et}\big(X \times_{\FF_q} \ov{\FF_q}, \ov{\QQ_l}\big) \cong \ov{\QQ_l}^{2g}$ is an element of the general symplectic group $\GSp\big(H^1_{\et}\big(X \times_{\FF_q} \ov{\FF_q}, \ov{\QQ_l}\big)\big)$ where we equip this vector space with the intersection form. The character ring of this general symplectic group is~$R_g$ and $(T_g,W_g)$ are a maximal torus and associated Weyl group of $\GSp(2g,\ov{\QQ_l})$.

Schiffmann's proof of this theorem gives rise to explicit (but complicated) formulae for these polynomials; these were later substantially combinatorially simplified by Mellit~\cite{mellit}. The fact that the number of absolutely indecomposable vector bundles over~$X$ of fixed class is finite follows from the observation that any sufficiently unstable coherent sheaf is decomposable as its Harder--Narasimhan filtration must split in some place, and so the stack of absolutely indecomposable vector bundles is a constructible substack of a finite type stack. Therefore, $\cA_{n,d}(X)$ also counts isomorphism classes of absolutely indecomposable coherent sheaves over~$X$ of this class. The standard arguments for quivers involving Galois cohomology and the Krull--Schmidt theorem still apply to reduce the problem to counting all isomorphism classes of coherent sheaves on~$X$ of this class, but this number is infinite for $n > 0$. In fact, we should also point out that this count is not the same as the stacky volume of the stack of coherent sheaves, where one weights the count by the inverses of the size of the automorphism groups (this stacky volume has a~very elegant formula involving the Zeta function of the curve~\cite{Harder}). Instead, Schiffmann uses a suitable truncation of the category of coherent sheaves on~$X$ given by the subcategory of positive coherent sheaves (i.e., sheaves whose HN subquotients have positive degrees). Similarly to the case of quivers, one can perform a unipotent reduction and partition the stack of positive sheaves by Jordan normal types.

Furthermore, the coefficients of these polynomials satisfy some form of positivity. For this, Mozgovoy and Schiffmann~\cite{MS, schiffmann} relate the polynomial $\cA_{g,n}(t, \dots, t) \in \QQ[t]$ with the (compactly supported) Poincar\'{e} polynomial of moduli spaces of semistable Higgs bundles over a genus $g$ smooth complex projective curve for any $d$ coprime to $n$ (recall that $\cA_{g,n,d} = \cA_{g,n}$ is independent of $d$ by work of Mellit~\cite{mellit}).

\begin{Theorem}[Schiffmann \cite{schiffmann}]Let $X_{\CC}$ be a smooth complex projective curve of genus $g$ and $\cH_{X_\CC}^{\text{\rm ss}}(n,d)$ denote the moduli space of semistable Higgs bundles over~$X_{\CC}$ of coprime rank and degree. Then
\begin{gather*} \sum_{i \geq 0} H^i_c\big(\cH_{X_\CC}^{\text{\rm ss}}(n,d),\QQ\big)t^n = t^{2(1+(g-1)n^2)}\cA_{g,n}(t, \dots, t).\end{gather*}
\end{Theorem}

We note that for coprime $(n,d)$, the notions of absolutely indecomposable and indecomposable coincide. In fact, Schiffmann's proof of this theorem is inspired by the work of Crawley-Boevey and Van den Bergh: it involves relating $\cA_{g,n}$ to the point count of moduli spaces of semistable Higgs bundles on a curve over a finite field (provided the characteristic is sufficiently large) by fitting this moduli space into a family over $\AA^1$. Indeed the forgetful map from the stack of stable Higgs bundles to the stack of vector bundles does not land in the indecomposable locus (for example, consider Exercise~\ref{exer Higgs bundle indecom bundle not}). Therefore, one needs a slightly perturbed model of the Higgs bundle moduli space to compare with indecomposable vector bundles.

To construct such a family, Schiffmann uses a variant of the construction of the functorial construction of moduli of sheaves due to \'{A}lvarez-C\'{o}nsul and King~\cite{ack} which depends on a choice of two polarising line bundles $(\cL_1,\cL_2)$ on $X$ (rather than two twists of the same bundle); the choice of two line bundles enables a construction of a family of a algebraic symplectic reductions $\cY \ra \AA^1$ over any field~$k$ such that $\cY_0 =\cH_{X}^{\text{ss}}(n,d)$ and, moreover, the fibre~$\cY_1$ can be compared with indecomposable vector bundles. More precisely, $\cY$ is constructed as the GIT quotient of the preimage of a line under an algebraic moment map on the cotangent bundle of an affine variety. The affine variety in question arises as a closed subvariety in the representation space of a Kronecker quiver with $h^0\big(\cL_2^\vee \otimes \cL_1\big)$ arrows, where the dimension vector is determined by the class $(n,d)$ and the degrees of the pair of line bundles using the Euler form for sheaves on~$X$ (for the detailed construction, see \cite[Sections~6.3--6.8]{schiffmann}).

After constructing the family $\cY \ra \AA^1$ with $\cY_0 =\cH_{X}^{\text{ss}}(n,d)$, Schiffmann shows that over any finite field $\FF_q$ of sufficiently large characteristic the following statements hold:
\begin{enumerate}\itemsep=0pt
\item[1)] the schemes $\cY_0$ and $\cY_1$ are smooth (which uses standard properties about GIT quotients of free group actions, see \cite[Lemma 6.5]{schiffmann}),
\item[2)] the point counts of $\cY_0$ and $\cY_1$ over $\FF_q$ coincide (which is proved by using a contracting $\GG_m$-action on $\cY$, see \cite[Proposition 6.9]{schiffmann}),
\item[3)] the point count of $\cY_1$ is polynomial (this is proved by relating this to the count $\cA_{d,n}(X)$ by the following formula, for details see \cite[Lemma 6.4 and Section~6.9]{schiffmann}):
\begin{gather}\label{eq pt count Y1}
|\cY_1(\FF_q)| = q^{1+(g-1)r^2}\cA_{d,n}(X).
\end{gather}
\end{enumerate}
The moduli space of stable Higgs bundles over a finite field is already known to be cohomologically pure (for example, see \cite[Section~1.3]{HRV_semiproj}) and so \eqref{eq pt count Y1} also gives an explicit formula for the $\ell$-adic Poincar\'{e} polynomial of the moduli space of stable Higgs bundles on $X$. Finally by spreading out a smooth projective curve $X_{\QQ}$ of genus $g$ defined over $\QQ$ to some localisation $R:=\ZZ\big[\frac{1}{N}\big]$ of the integers, one can relate the $\ell$-adic Poincar\'{e} polynomial of $\cH_{X_{\QQ}}^{\text{ss}}(n,d)$ with the moduli space $\cH_{X_R \times {\FF_q}}^{\text{ss}}(n,d)$ of the base change of $X_{\QQ}$ to a finite field $\FF_q$ of large characteristic, and by the comparison theorem one can also relate this to the singular Poincar\'{e} polynomial (with $\CC$-coefficients) of the base change $X_{\CC} = X_R \times_R \CC$ (in the above, we use cohomology with compact supports). Since the diffeomorphism class of the complex variety $\cH_{X_{\CC}}^{\text{ss}}(n,d)$ is independent of the choice of smooth projective curve $X_{\CC}$ of genus $g$ (as they are all diffeomorphic to the genus $g$ character variety for $\Gl_n$), this enables Schiffmann to relate $\cA_{g,n}$ with the Poincar\'{e} polynomial of $\cH_{X_{\CC}}^{\text{ss}}(n,d)$ for any $X_{\CC}$ (see \cite[Section~6.10]{schiffmann} for further details).

\subsection{Representation theoretic interpretations of the Kac polynomials}

The Kac polynomials $\cA_{Q,d}$ and $\cA_{g,n}$ for quivers and curves have representation theoretic interpretations given by considering Hall algebras associated to quivers and curves, which is beautifully surveyed in \cite{schiffmann_Kacpoly}.

In the quiver case, Kac conjectured that the constant term of $\cA_{Q,d}$ for a quiver $Q$ without loops was the dimension of the $d$-th root space in the decomposition of the associated Kac--Moody algebra $\fg_Q$; this conjecture was proved for indivisible dimension vectors $d$ in \cite{CBVdB} and in general by Hausel~\cite{Hausel_KacConj}. Hall algebras then enter the picture as a way to construct $\fg_Q$ from the moduli stack of all quiver representations. More precisely, the objects in the Hall algebra $H_Q$ associated to the category of $\FF_q$-representations of $Q$ are functions $[\Rep/\G](\FF_q) \ra \CC$ and one can construct a so-called spherical Hall algebra as the subalgebra of $H_Q \otimes \CC\big[\ZZ^V\big]$ generated by the characteristic functions of the simple representations and the group algebra $ \CC\big[\ZZ^V\big]$. By work of Ringel and Green, this spherical Hall algebra of $Q$ is isomorphic to the positive Borel subalgebra of the Drinfeld--Jimbo quantum enveloping algebra of~$\fg_Q$. The full quantum enveloping algebra of $\fg_Q$ is given by taking the Drinfeld double of this spherical Hall algebra (see \cite[Section~2]{schiffmann_Kacpoly}). The relationship between the Kac polynomial $\cA_{Q,d}$ and moduli spaces of stable representations of the doubled quiver $\ov{Q}$ with relations $\cR_0$ imposed by the zero level set of the moment can be extended to the stack of all representations of $\big(\ov{Q},\cR_0\big)$, which can be viewed as the cotangent stack of the moduli stack of representations of $Q$. More precisely, the Poincar\'{e} polynomials of the stacks of representations of $\big(\ov{Q},\cR_0\big)$ for varying $d$ can be expressed in terms of the Kac polynomials~$\cA_{Q,d}$ for varying $d$ (see \cite[Theorem~4.2]{schiffmann_Kacpoly}); this formula combines a purity result of Davison~\cite{Ben} with a point count of Mozgovoy~\cite{mozgovoy}. Furthermore, work of Schiffmann and Vasserot~\cite{SV} equips the Borel--Moore homology of the stack of representations of~$\big(\ov{Q},\cR_0\big)$ with an associative algebra structure, known as the two-dimensional cohomologoical Hall algebra of the quiver (see \cite[Section~4.2]{schiffmann_Kacpoly}); this $2d$-cohomological Hall algebra is conjectured to be a~deformation of the universal enveloping algebra of a graded Lie algebra attached to~$Q$, whose Hilbert series coincides with the Kac polynomial.

In the curves case, the representation theoretic interpretation of $\cA_{g,n}$ involves a spherical Hall algebra, which is constructed analogously to above, but replacing the characteristic functions of simple quiver representations with constant functions on the moduli stack of rank~$n$ degree~$d$ coherent sheaves with $n \leq 1$. There is also a corresponding $2d$-cohomological Hall algebra in the curves cases associated to the stack of Higgs sheaves, which can be viewed as the cotangent stack of the moduli stack of coherent sheaves~\cite{schiffmann_sala}. The representation theoretic interpretation of these Hall algebras is ongoing work of Schiffmann and collaborators; for a nice overview of this work, see \cite[Sections~5--8]{schiffmann_Kacpoly}.

\subsection{Related open questions}

As described in Section~\ref{sec branes}, there are several interesting actions on moduli spaces of quiver representations and vector bundles on curves, as well as their hyperk\"{a}hler analogues (for example, these actions often arise from automorphisms of either the base field, the quiver or the curve). One can also count indecomposable objects in a category which respect an automorphism (or subgroup of automorphisms). A natural question is to consider the count of such absolutely indecomposable invariant objects and study their properties; for example, one could ask whether one obtain polynomial counts and whether the coefficients are non-negative. An even more ambitious question is whether there should be a cohomological interpretation of the non-negativity of the coefficients. In a more representation theoretic direction, one would hope to be able to attach some sort of a root system to such a category with an action by an automorphism group, such that the invariants of indecomposable objects respecting these automorphisms are the positive roots.

This question was partially investigated for quiver representations respecting a so-called admissible quiver automorphism by Hubery \cite{hubery_thesisarticle}, who showed that the number of isomorphism classes of absolutely indecomposable invariant representations over a finite field is a rational polynomial in the size of the field, and moreover, this is independent of the orientation and invariant under the Weyl group. Furthermore, he showed that the dimensions of the indecomposable representations are specified by the positive roots of an associated symmetrisable Kac--Moody Lie algebra. It would be interesting to investigate whether the coefficients could be described in terms of the cohomology of an associated brane in the corresponding hyperk\"{a}hler quiver variety given by taking the fixed locus of this automorphism group.

\subsection*{Acknowledgements}

This article is based on lecture notes for the fifth workshop on the Geometry and Physics of Higgs bundles and the author would like to thank Laura Schaposnik for the organisation of this workshop. The author would also like to thank the participants of a seminar on counting indecomposable quiver representations held at the Freie Universit\"{a}t Berlin for interesting discussions related to this topic. The author is supported by the Excellence Initiative of the DFG at the Freie Universit\"{a}t Berlin.

\LastPageEnding


\begin{thebibliography}{99}\addcontentsline{toc}{section}{References}
\footnotesize\itemsep=0pt

\bibitem{ack}
\'{A}lvarez C\'{o}nsul L., King A., A functorial construction of moduli of
 sheaves, \href{https://doi.org/10.1007/s00222-007-0042-5}{\textit{Invent. Math.}} \textbf{168} (2007), 613--666,
 \href{https://arxiv.org/abs/math.AG/0602032}{math.AG/0602032}.

\bibitem{atiyahbott}
Atiyah M.F., Bott R., The {Y}ang--{M}ills equations over {R}iemann surfaces,
 \href{https://doi.org/10.1098/rsta.1983.0017}{\textit{Philos. Trans. Roy. Soc. London Ser.~A}} \textbf{308} (1983),
 523--615.

\bibitem{baraglia}
Baraglia D., Classification of the automorphism and isometry groups of {H}iggs
 bundle moduli spaces, \href{https://doi.org/10.1112/plms/pdw014}{\textit{Proc. Lond. Math. Soc.}} \textbf{112} (2016),
 827--854, \href{https://arxiv.org/abs/1411.2228}{arXiv:1411.2228}.

\bibitem{BS1}
Baraglia D., Schaposnik L.P., Higgs bundles and {$(A,B,A)$}-branes,
 \href{https://doi.org/10.1007/s00220-014-2053-6}{\textit{Comm. Math. Phys.}} \textbf{331} (2014), 1271--1300,
 \href{https://arxiv.org/abs/1305.4638}{arXiv:1305.4638}.

\bibitem{BS2}
Baraglia D., Schaposnik L.P., Real structures on moduli spaces of {H}iggs
 bundles, \href{https://doi.org/10.4310/ATMP.2016.v20.n3.a2}{\textit{Adv. Theor. Math. Phys.}} \textbf{20} (2016), 525--551,
 \href{https://arxiv.org/abs/1309.1195}{arXiv:1309.1195}.

\bibitem{BGelP}
Bernstein I.N., Gel'fand I.M., Ponomarev V.A., Coxeter functors and {G}abriel's
 theorem, \href{https://doi.org/10.1070/RM1973v028n02ABEH001526}{\textit{Russian Math. Surveys}} \textbf{28} (1973), no.~2, 17--32.

\bibitem{BB}
Bia{\l}ynicki-Birula A., Some theorems on actions of algebraic groups,
 \href{https://doi.org/10.2307/1970915}{\textit{Ann. of Math.}} \textbf{98} (1973), 480--497.

\bibitem{BGP}
Biswas I., Garc\'{i}a-Prada O., Anti-holomorphic involutions of the moduli
 spaces of {H}iggs bundles, \href{https://doi.org/10.5802/jep.16}{\textit{J.~\'{E}c. Polytech. Math.}} \textbf{2}
 (2015), 35--54, \href{https://arxiv.org/abs/1401.7236}{arXiv:1401.7236}.

\bibitem{BGPH}
Biswas I., Garc\'{i}a-Prada O., Hurtubise J., Pseudo-real principal
 {$G$}-bundles over a real curve, \href{https://doi.org/10.1112/jlms/jdv055}{\textit{J.~Lond. Math. Soc.}} \textbf{93}
 (2016), 47--64, \href{https://arxiv.org/abs/1502.00563}{arXiv:1502.00563}.

\bibitem{brion_notes}
Brion M., Representations of quivers, in Geometric Methods in Representation
 Theory.~{I}, \textit{S\'{e}min. Congr.}, Vol.~24, Soc. Math. France, Paris,
 2012, 103--144.

\bibitem{CB_notes}
Crawley-Boevey W., Lectures on representations of quivers, 1992, {a}vailable at
 \url{https://www.math.uni-bielefeld.de/~wcrawley/quivlecs.pdf}.

\bibitem{CBmoment}
Crawley-Boevey W., Geometry of the moment map for representations of quivers,
 \href{https://doi.org/10.1023/A:1017558904030}{\textit{Compositio Math.}} \textbf{126} (2001), 257--293.

\bibitem{CBVdB}
Crawley-Boevey W., Van~den Bergh M., Absolutely indecomposable representations
 and {K}ac--{M}oody {L}ie algebras (with an appendix by {H}iraku {N}akajima),
 \href{https://doi.org/10.1007/s00222-003-0329-0}{\textit{Invent. Math.}} \textbf{155} (2004), 537--559,
 \href{https://arxiv.org/abs/math.RA/0106009}{math.RA/0106009}.

\bibitem{Ben}
Davison B., The integrality conjecture and the cohomology of preprojective
 stacks, \href{https://arxiv.org/abs/1602.02110}{arXiv:1602.02110}.

\bibitem{donaldsonNS}
Donaldson S.K., A new proof of a theorem of {N}arasimhan and {S}eshadri,
 \href{https://doi.org/10.4310/jdg/1214437664}{\textit{J.~Differential Geom.}} \textbf{18} (1983), 269--277.

\bibitem{fjm}
Franco E., Jardim M., Marchesi S., Branes in the moduli space of framed
 sheaves, \href{https://doi.org/10.1016/j.bulsci.2017.04.002}{\textit{Bull. Sci. Math.}} \textbf{141} (2017), 353--383,
 \href{https://arxiv.org/abs/1504.05883}{arXiv:1504.05883}.

\bibitem{ginzburg}
Ginzburg V., Lectures on {N}akajima's quiver varieties, in Geometric methods in
 representation theory.~{I}, \textit{S\'{e}min. Congr.}, Vol.~24, Soc. Math.
 France, Paris, 2012, 145--219, \href{https://arxiv.org/abs/0905.0686}{arXiv:0905.0686}.

\bibitem{Harder}
Harder G., Chevalley groups over function fields and automorphic forms,
 \href{https://doi.org/10.2307/1971073}{\textit{Ann. of Math.}} \textbf{100} (1974), 249--306.

\bibitem{HN}
Harder G., Narasimhan M.S., On the cohomology groups of moduli spaces of vector
 bundles on curves, \href{https://doi.org/10.1007/BF01357141}{\textit{Math. Ann.}} \textbf{212} (1975), 215--248.

\bibitem{Hausel_KacConj}
Hausel T., Kac's conjecture from {N}akajima quiver varieties, \href{https://doi.org/10.1007/s00222-010-0241-3}{\textit{Invent.
 Math.}} \textbf{181} (2010), 21--37, \href{https://arxiv.org/abs/0811.1569}{arXiv:0811.1569}.

\bibitem{HLRV}
Hausel T., Letellier E., Rodriguez-Villegas F., Positivity for {K}ac
 polynomials and {DT}-invariants of quivers, \href{https://doi.org/10.4007/annals.2013.177.3.8}{\textit{Ann. of Math.}}
 \textbf{177} (2013), 1147--1168, \href{https://arxiv.org/abs/1204.2375}{arXiv:1204.2375}.

\bibitem{HRV}
Hausel T., Rodriguez-Villegas F., Mixed {H}odge polynomials of character
 varieties (with an appendix by {N}icholas {M}.~{K}atz, \href{https://doi.org/10.1007/s00222-008-0142-x}{\textit{Invent. Math.}}
 \textbf{174} (2008), 555--624, \href{https://arxiv.org/abs/math.AG/0612668}{math.AG/0612668}.

\bibitem{HRV_semiproj}
Hausel T., Rodriguez-Villegas F., Cohomology of large semiprojective
 hyperk\"{a}hler varieties, \textit{Ast\'{e}risque} \textbf{370} (2015),
 113--156, \href{https://arxiv.org/abs/1309.4914}{arXiv:1309.4914}.

\bibitem{Hitchin}
Hitchin N.J., The self-duality equations on a {R}iemann surface, \href{https://doi.org/10.1112/plms/s3-55.1.59}{\textit{Proc.
 London Math. Soc.}} \textbf{55} (1987), 59--126.

\bibitem{hitchinetal}
Hitchin N.J., Karlhede A., Lindstr\"{o}m U., Ro\v{c}ek M., Hyperk\"{a}hler
 metrics and supersymmetry, \href{https://doi.org/10.1007/BF01214418}{\textit{Comm. Math. Phys.}} \textbf{108} (1987),
 535--589.

\bibitem{hoskins_quivers}
Hoskins V., Stratifications associated to reductive group actions on affine
 spaces, \href{https://doi.org/10.1093/qmath/hat046}{\textit{Q.~J.~Math.}} \textbf{65} (2014), 1011--1047,
 \href{https://arxiv.org/abs/1210.6811}{arXiv:1210.6811}.

\bibitem{HS_quiver_autos}
Hoskins V., Schaffhauser F., Group actions on quiver varieties and
 applications, \href{https://arxiv.org/abs/1612.06593}{arXiv:1612.06593}.

\bibitem{HS_galois}
Hoskins V., Schaffhauser F., Rational points of quiver moduli spaces,
 \href{https://arxiv.org/abs/1704.08624}{arXiv:1704.08624}.

\bibitem{Hua}
Hua J., Counting representations of quivers over finite fields,
 \href{https://doi.org/10.1006/jabr.1999.8220}{\textit{J.~Algebra}} \textbf{226} (2000), 1011--1033.

\bibitem{hubery_thesisarticle}
Hubery A., Quiver representations respecting a quiver automorphism: a
 generalisation of a theorem of {K}ac, \href{https://doi.org/10.1112/S0024610703004988}{\textit{J.~London Math. Soc.}}
 \textbf{69} (2004), 79--96, \href{https://arxiv.org/abs/math.RT/0203195}{math.RT/0203195}.

\bibitem{HL}
Huybrechts D., Lehn M., The geometry of moduli spaces of sheaves,
 \href{https://doi.org/10.1017/CBO9780511711985}{\textit{Cambridge Mathematical Library}}, 2nd~ed., Cambridge University Press, Cambridge, 2010.

\bibitem{Kac1}
Kac V.G., Infinite root systems, representations of graphs and invariant
 theory, \href{https://doi.org/10.1007/BF01403155}{\textit{Invent. Math.}} \textbf{56} (1980), 57--92.

\bibitem{Kac_loops}
Kac V.G., Some remarks on representations of quivers and infinite root systems,
 in Representation Theory,~{II} ({P}roc. {S}econd {I}nternat. {C}onf.,
 {C}arleton {U}niv., {O}ttawa, {O}nt., 1979), \href{https://doi.org/10.1007/BFb0088472}{\textit{Lecture Notes in Math.}},
 Vol.~832, Springer, Berlin, 1980, 311--327.

\bibitem{Kac2}
Kac V.G., Root systems, representations of quivers and invariant theory, in
 Invariant Theory ({M}ontecatini, 1982), \href{https://doi.org/10.1007/BFb0063236}{\textit{Lecture Notes in Math.}}, Vol.~996, Springer, Berlin, 1983, 74--108.

\bibitem{Kapustin_Witten}
Kapustin A., Witten E., Electric-magnetic duality and the geometric {L}anglands
 program, \href{https://doi.org/10.4310/CNTP.2007.v1.n1.a1}{\textit{Commun. Number Theory Phys.}} \textbf{1} (2007), 1--236,
 \href{https://arxiv.org/abs/hep-th/0604151}{hep-th/0604151}.

\bibitem{kempf_ness}
Kempf G., Ness L., The length of vectors in representation spaces, in Algebraic
 Geometry ({P}roc. {S}ummer {M}eeting, {U}niversity {C}openhagen,
 {C}openhagen, 1978), \href{https://doi.org/10.1007/BFb0066647}{\textit{Lecture Notes in Math.}}, Vol.~732, Springer,
 Berlin, 1979, 233--243.

\bibitem{king}
King A.D., Moduli of representations of finite-dimensional algebras,
 \href{https://doi.org/10.1093/qmath/45.4.515}{\textit{Quart.~J. Math. Oxford Ser.~(2)}} \textbf{45} (1994), 515--530.

\bibitem{kirwan}
Kirwan F.C., Cohomology of quotients in symplectic and algebraic geometry,
 \href{https://doi.org/10.1007/BF01145470}{\textit{Mathematical Notes}}, Vol.~31, Princeton University Press, Princeton,
 NJ, 1984.

\bibitem{langer}
Langer A., Semistable sheaves in positive characteristic, \href{https://doi.org/10.4007/annals.2004.159.251}{\textit{Ann. of
 Math.}} \textbf{159} (2004), 251--276.

\bibitem{lBP}
Le~Bruyn L., Procesi C., Semisimple representations of quivers, \href{https://doi.org/10.2307/2001477}{\textit{Trans.
 Amer. Math. Soc.}} \textbf{317} (1990), 585--598.

\bibitem{mw}
Marsden J., Weinstein A., Reduction of symplectic manifolds with symmetry,
 \href{https://doi.org/10.1016/0034-4877(74)90021-4}{\textit{Rep. Math. Phys.}} \textbf{5} (1974), 121--130.

\bibitem{mellit}
Mellit A., Poincar\'e polynomials of moduli spaces of {H}iggs bundles and
 character varieties (no punctures), \href{https://arxiv.org/abs/1707.04214}{arXiv:1707.04214}.

\bibitem{MilneEtale}
Milne J.S., Lectures on etale cohomology, 2013, {a}vailable at
 \url{http://www.jmilne.org/math/}.

\bibitem{mozgovoy}
Mozgovoy S., Motivic {D}onaldson--{T}homas invariants and {M}c{K}ay
 correspondence, \href{https://arxiv.org/abs/1107.6044}{arXiv:1107.6044}.

\bibitem{MS}
Mozgovoy S., Schiffmann O., Counting higgs bundles and type a quiver bundles,
 \href{https://arxiv.org/abs/1705.04849}{arXiv:1705.04849}.

\bibitem{mumford}
Mumford D., Fogarty J., Kirwan F., Geometric invariant theory,
 \href{https://doi.org/10.1007/978-3-642-57916-5}{\textit{Ergebnisse der Mathematik und ihrer Grenzgebiete~(2)}}, Vol.~34, 3rd~ed., Springer-Verlag, Berlin, 1994.

\bibitem{nakajima}
Nakajima H., Instantons on {ALE} spaces, quiver varieties, and {K}ac--{M}oody
 algebras, \href{https://doi.org/10.1215/S0012-7094-94-07613-8}{\textit{Duke Math.~J.}} \textbf{76} (1994), 365--416.

\bibitem{ns}
Narasimhan M.S., Seshadri C.S., Stable and unitary vector bundles on a compact
 {R}iemann surface, \href{https://doi.org/10.2307/1970710}{\textit{Ann. of Math.}} \textbf{82} (1965), 540--567.

\bibitem{ness}
Ness L., A stratification of the null cone via the moment map (with an appendix
 by {D}avid {M}umford), \href{https://doi.org/10.2307/2374395}{\textit{Amer.~J. Math.}} \textbf{106} (1984),
 1281--1329.

\bibitem{newstead_bundles}
Newstead P.E., Characteristic classes of stable bundles of rank~{$2$} over an
 algebraic curve, \href{https://doi.org/10.2307/1996247}{\textit{Trans. Amer. Math. Soc.}} \textbf{169} (1972),
 337--345.

\bibitem{newstead}
Newstead P.E., Introduction to moduli problems and orbit spaces, \textit{Tata
 Institute of Fundamental Research Lectures on Mathematics and Physics},
 Vol.~51, Tata Institute of Fundamental Research, Bombay, Narosa Publishing
 House, New Delhi, 1978.

\bibitem{nitsure}
Nitsure N., Construction of {H}ilbert and {Q}uot schemes, in Fundamental
 Algebraic Geometry, \textit{Math. Surveys Monogr.}, Vol.~123, Amer. Math.
 Soc., Providence, RI, 2005, 105--137, \href{https://arxiv.org/abs/math.AG/0504590}{math.AG/0504590}.

\bibitem{proudfoot}
Proudfoot N.J., Hyperk\"ahler analogues of {K}\"ahler quotients, Ph.D.~Thesis,
 University of California, Berkeley, 2004, \href{https://arxiv.org/abs/math.AG/0405233}{math.AG/0405233}.

\bibitem{reinekeHN}
Reineke M., The {H}arder--{N}arasimhan system in quantum groups and cohomology
 of quiver moduli, \href{https://doi.org/10.1007/s00222-002-0273-4}{\textit{Invent. Math}.} \textbf{152} (2003), 349--368,
 \href{https://arxiv.org/abs/math.QA/0204059}{math.QA/0204059}.

\bibitem{reineke_QGrass}
Reineke M., Every projective variety is a quiver {G}rassmannian,
 \href{https://doi.org/10.1007/s10468-012-9357-z}{\textit{Algebr. Represent. Theory}} \textbf{16} (2013), 1313--1314,
 \href{https://arxiv.org/abs/1204.5730}{arXiv:1204.5730}.

\bibitem{schiffmann_sala}
Sala F., Schiffmann O., Cohomological {H}all algebra of {H}iggs sheaves on a
 curve, \href{https://arxiv.org/abs/1801.03482}{arXiv:1801.03482}.

\bibitem{schiffmann_lecturesHall}
Schiffmann O., Lectures on {H}all algebras, in Geometric Methods in
 Representation Theory.~{II}, \textit{S\'{e}min. Congr.}, Vol.~24, Soc. Math.
 France, Paris, 2012, 1--141, \href{https://arxiv.org/abs/math.RT/0611617}{math.RT/0611617}.

\bibitem{schiffmann}
Schiffmann O., Indecomposable vector bundles and stable {H}iggs bundles over
 smooth projective curves, \href{https://doi.org/10.4007/annals.2016.183.1.6}{\textit{Ann. of Math.}} \textbf{183} (2016),
 297--362, \href{https://arxiv.org/abs/1406.3839}{arXiv:1406.3839}.

\bibitem{schiffmann_Kacpoly}
Schiffmann O., Kac polynomials and {L}ie algebras associated to quivers and
 curves, \href{https://arxiv.org/abs/1802.09760}{arXiv:1802.09760}.

\bibitem{SV}
Schiffmann O., Vasserot E., On cohomological {H}all algebras of quivers:
 generators, \href{https://arxiv.org/abs/1705.07488}{arXiv:1705.07488}.

\bibitem{serreGAGA}
Serre J.-P., G\'{e}om\'{e}trie alg\'{e}brique et g\'{e}om\'{e}trie analytique,
 \href{https://doi.org/10.5802/aif.59}{\textit{Ann. Inst. Fourier, Grenoble}} \textbf{6} (1956), 1--42.

\bibitem{seshadri}
Seshadri C.S., Space of unitary vector bundles on a compact {R}iemann surface,
 \href{https://doi.org/10.2307/1970444}{\textit{Ann. of Math.}} \textbf{85} (1967), 303--336.

\bibitem{simpson_HiggsGm}
Simpson C.T., Higgs bundles and local systems, \href{https://doi.org/10.1007/BF02699491}{\textit{Inst. Hautes \'{E}tudes
 Sci. Publ. Math.}} \textbf{75} (1992), 5--95.

\bibitem{simpson}
Simpson C.T., Moduli of representations of the fundamental group of a smooth
 projective variety.~{I}, \href{https://doi.org/10.1007/BF02698887}{\textit{Inst. Hautes \'{E}tudes Sci. Publ. Math.}}
\textbf{79} (1994), 47--129.

\bibitem{UY}
Uhlenbeck K., Yau S.-T., On the existence of {H}ermitian--{Y}ang--{M}ills
 connections in stable vector bundles, \href{https://doi.org/10.1002/cpa.3160390714}{\textit{Comm. Pure Appl. Math.}}
 \textbf{39} (1986), S257--S293.

\end{thebibliography}
\end{document}